\documentclass[final,3p,times]{elsarticle}

\usepackage{amssymb}
\usepackage{amsmath}
\usepackage{amsthm}

\usepackage{fancyhdr}

\usepackage{dsfont}

\RequirePackage{graphicx}
\RequirePackage{pdfpages}
\RequirePackage{xcolor}		
\RequirePackage{array}
\RequirePackage{booktabs}		
\RequirePackage{color}
\usepackage{float}
\usepackage{comment}
\usepackage{caption}
\usepackage{subcaption}

\numberwithin{equation}{section}
\theoremstyle{plain}
\newtheorem{Thrm}{Theorem}[section]
\newtheorem{Assu}{Assumption}

\newtheorem{Lem}[Thrm]{Lemma}
\newtheorem{Cor}[Thrm]{Corollary}
\newtheorem{Prop}[Thrm]{Proposition}
\theoremstyle{remark}
\newtheorem{Ex}{Example}
\newtheorem{Assum}[Assu]{Assumption}

\newtheorem{Rem}{Remark}

\newcommand{\N}{\mathbb{N}}

\newcommand{\R}{\mathbb{R}}

\newcommand{\BB}{\mathbb{B}}
\newcommand{\E}{\mathbb{E}}

\renewcommand{\P}{\mathbb{P}}

\newcommand{\A}{\mathcal{A}}
\newcommand{\B}{\mathcal{B}}
\newcommand{\D}{\mathcal{D}}

\newcommand{\J}{\mathcal{J}}

\newcommand{\LL}{\mathcal{L}}
\newcommand{\NN}{\mathcal{N}}

\newcommand{\HH}{\mathcal{H}}

\newcommand{\Cov}{\mathrm{Cov}}
\newcommand{\Var}{\mathrm{Var}}

\newcommand{\Span}{\mathrm{span}}
\newcommand{\rank}{\mathrm{rank}}
\newcommand{\ran}{\mathrm{ran}}
\renewcommand{\ker}{\mathrm{ker}}
\newcommand{\supp}{\mathrm{supp}}

\newcommand{\tr}{\mathrm{tr}}

\renewcommand{\phi}{\varphi}
\renewcommand{\epsilon}{\varepsilon}

\newcommand{\<}{\langle}
\renewcommand{\>}{\rangle}
\newcommand{\wt}{\widetilde}
\newcommand{\wh}{\widehat}
\newcommand{\ol}{\overline}

\newcommand{\hide}[1]{} 

\allowdisplaybreaks

\graphicspath{ {Simulationen/} }

\journal{Journal Of Functional Analysis}

\begin{document}

\begin{frontmatter}



\title{Sobolev-Mercer Expansions and Applications to Stochastic Processes}

\author{Daniel Constantin Rademacher} 
\ead{d.rademacher@tugraz.at}
\ead[url]{https://www.tugraz.at/institute/stat/home}

\affiliation{organization={TU Graz},
            addressline={Kopernikusgasse 24/III}, 
            city={Graz},
            postcode={8010}, 
            state={Styria},
            country={Austria}}

\begin{abstract}

We establish a fundamental extension of Mercer’s celebrated theorem by introducing a class of higher-order kernel operators
acting on Sobolev spaces $H^k(\Theta)$, where $\Theta \subset \R^d$ is a bounded domain and $k\in\N_0$
corresponds to the order of weak differentiability. 
These operators transcend the classical $L^2$-framework by explicitly incorporating the information encoded in the (weak) derivatives of the kernel and 
correspond precisely to the Hilbert-Schmidt mappings between Sobolev spaces.
The spectral decomposition of these operators then yields Mercer-type expansions that are optimal in $H^k(\Theta\times\Theta)$.
Notably, we derive from the embedding properties of Sobolev spaces, that for $k>d$, these expansions also converge uniformly without requiring 
the kernel to be positive definite. 
For positive definite kernels, we confirm the nuclearity of these higher-order operators and establish a significant refinement of Mercer's Theorem, 
which ensures uniform convergence of all term-wise derivatives and provides explicit convergence rates (including derivatives) tied to the spectral decay. 
These results lead to novel spectral representations of Reproducing Kernel Hilbert Spaces and have subtle implications for stochastic analysis. 
Applied to the covariance kernels of weakly differentiable random fields, our theory provides refined Karhunen-Loève expansions that facilitate the 
simultaneous mean-square optimal approximation of both the process and its derivatives.  
\end{abstract}


\end{frontmatter}





\section{Introduction}\label{Sec:Introduction}
Since its publication in 1909, Mercer’s theorem \cite{Mercer1909FunctionsOfPositiveAndNegativeType} - named after James Mercer - has remained a cornerstone 
in the spectral theory of compact operators. 
Originally emerging from the foundational work of Hilbert \cite{Hilbert1904GrundzügeTheoriederIntegralgleichungenErsteMitteilung} and 
Schmidt \cite{Schmidt1907ZurTheorieLinearerundNichtlinearerIntegralgleichungen}, 
the theorem now underpins several diverse fields. 
Most notably, it is central to \emph{integral equations} and \emph{Fredholm theory}, serves as the basis for the \emph{Karhunen–Loève representation}
\cite{Karhunen1947LineareMethodenInWahrscheinlichkeitsrechnung} in stochastic processes, 
and has more recently become a vital theoretical tool for \emph{kernel methods} in machine learning (e.g., \cite{Cucker2001OnTheMathematicalFoundationsOfLearning}). 
The advancements presented in this work significantly refine these classic results, which now emerge as specific instances of a more general framework. 
To provide context for these developments, we begin with a concise review of the theorem's current status.
 
\paragraph{Mercer's theorem} 
In its original form Mercer's theorem asserts that for any continuous and positive definite kernel 
$h\colon [a,b]\times [a,b] \to \R$, there exists an orthonormal system $\{ e_j\mid j\in \N  \}$ of continuous functions in $L^2([a,b])$ such that 
\begin{align}\label{Eq:MercerExpansion_Intro}
	h(x,y) = \sum_{j} \lambda_j\, e_j(x)\,e_j(y),
\end{align}
where the convergence is absolute and uniform. 
Here, the $e_j$'s are obtained as the normalized eigenfunctions of the $L^2$-integral operator associated to $h$,
\begin{align}\label{Eq:AssociatedOperator_Intro}
	T_h\colon L^2([a,b]) \to L^2([a,b]), \quad [T_hf](y) = \int_a^b h(x,y)\,f(x) \,dx, 
\end{align}
and the $\lambda_j$'s are the corresponding non-zero eigenvalues. While the spectral theorem for compact operators ensures $L^2$-convergence
for any square-integrable kernel, the uniform convergence of \eqref{Eq:MercerExpansion_Intro} provided by Mercer's theorem is a significantly stronger result.
It additionally requires $h$ to be continuous and positive definite, meaning $h$ is symmetric and satisfies
\begin{align}\label{Eq:PositiveDefiniteKernel_Intro}
	\sum_{i,j\leq n} c_i\, h(x_i,x_j)\, c_j \geq 0
\end{align}
for all $c_1,\ldots,c_n\in \R$, $x_1,\ldots,x_n \in [a,b]$ and $n\in \N$.
Indeed, Mercer's primary objective was to establish conditions under which a continuous and symmetric kernel $h$ yields an  
associated operator \eqref{Eq:AssociatedOperator_Intro} that is non-negative, i.e. satisfies for all continuous $f\colon [a,b] \to \R$,
\begin{align}\label{Eq:NonNegativeOperator_Intro}
	\<T_h f,f\>_{L^2} = \int_a^b \int_a^b f(x)\,h(x,y)\,f(y) \,dx\,dy \geq 0.
\end{align}
He first demonstrated the equivalence of \eqref{Eq:PositiveDefiniteKernel_Intro} and \eqref{Eq:NonNegativeOperator_Intro} - though Young 
\cite{Young1909ANoteOnAClassOfSymmetricFunctions} subsequently showed that
\eqref{Eq:PositiveDefiniteKernel_Intro} implies \eqref{Eq:NonNegativeOperator_Intro} for all Lebesgue-integrable $f$ - and then derived 
the uniform convergence of \eqref{Eq:MercerExpansion_Intro}. 
An immediate and vital consequence of this uniform convergence is the explicit trace formula 
\begin{align}\label{Eq:TraceFormula_Intro}
	\tr(T_h) = \int_a^b h(x,x) \,dx = \sum_j \lambda_j.
\end{align}  
This identity equates the integral of the kernel along its diagonal with the sum of its (non-negative) eigenvalues.
Notably, this formula implies that the eigenvalues are summable, confirming
that the operator $T_h$ is nuclear (or trace-class) for a continuous positive definite kernel.
These results also extend naturally to complex-valued kernels via straight forward adaptation.

\paragraph{Karhunen-Loève representation} 
Suppose $X=(X_t \mid a\leq t\leq b)$ is a square-integrable, centred stochastic process with a continuous covariance kernel 
$\gamma_X(s,t) = \E[X_s\,X_t]$. Expanding the process $X$ in terms of the eigenfunctions $e_j$ of the operator 
\eqref{Eq:AssociatedOperator_Intro} associated to the kernel $\gamma_X$ then gives the $L^2(\P)$-representation 
\begin{align}\label{Eq:KL-Representation_Intro}
	X_t = \sum_j \< X,e_j \>_{L^2} e_j(t),
\end{align}
where $L^2(\P)$ denotes the Lebesgue space of random variables with finite variance. 
While a stochastic process can be represented by various series expansions using complete sets of deterministic functions,
the Karhunen-Loève (KL)-expansion \eqref{Eq:KL-Representation_Intro} is of particular interest due to its biorthogonality: 
both the eigenfunctions $e_j$ and the corresponding scores $Z_j:= \<X,e_j\>_{L^2}$ are orthogonal in $L^2([a,b])$ and $L^2(\P)$, respectively. 
Consequently, the entire stochastic information of the process is encapsulated within a set of uncorrelated random variables with $\Var(Z_j) = \lambda_j$. 
Furthermore, according to the trace formula \eqref{Eq:TraceFormula_Intro}, the expectation of the $L^2$-norm (total variance) satisfies 
\begin{align}\label{Eq:TotalVariance_Intro}
	\E\|X\|_{L^2}^2 = \int_a^b \E [X_t^2] \, dt = \int_a^b \gamma_X(t,t) \,dt = \sum_j \lambda_j.
\end{align}
Truncating the series \eqref{Eq:KL-Representation_Intro} thus provides an optimal finite-dimensional approximation of $X$ with respect to the 
$L^2$-mean square error (total mean square error).
By Mercer's theorem the convergence of \eqref{Eq:KL-Representation_Intro} in mean square is also uniform in $t$, allowing
the truncation error to be controlled consistently across the interval. Finally, for Gaussian processes, the scores $Z_j$ are stochastically 
independent, normally distributed random variables and the series \eqref{Eq:KL-Representation_Intro} also converges almost surely uniformly (i.e., pathwise) 
in case of continuous sample paths.\\
These properties render the KL-representation not only a fundamental tool in the theory of stochastic processes but also a cornerstone of applications 
in statistics, data science and engineering - a prototypical example being principal component analysis for the dimension reduction in complex 
stochastic problems.

\paragraph{Relevant literature}
Due to their fundamental importance across the various applications mentioned above, numerous generalizations 
of Mercer's findings, the KL-expansion, and related results have emerged in the literature ever since.
Without claiming to be exhaustive, a few relevant sources are listed below:
Regarding the underlying domain, the interval $[a,b]$ can be replaced by any compact subset of $\R^d$, or more generally, any compact Hausdorff 
space equipped with a finite Borel measure (e.g., \cite{Riesz1955FunctionalAnalysis}).
The trace formula \eqref{Eq:TraceFormula_Intro} was later established for a broader class of kernels by Brislawn \cite{Brislawn1988KernelsOfTraceClassOperators}.
Extensions to unbounded domains in $\R$ were obtained by Buesco \cite{Buescu2004PositiveIntergralOperatorsInUnboundedDomains}
and Buesco et al. \cite{Buescu2004PositiveDefinitenessIntegralEqAndFourierTrafo}.
Furthermore, Mercer's theorem has been extended to matrix-valued kernels by De Vito et al. \cite{DeVito2013ExtensionMercerMatrixKernel}
and only recently, the equivalence of conditions \eqref{Eq:PositiveDefiniteKernel_Intro} and \eqref{Eq:NonNegativeOperator_Intro} 
was confirmed for this setting by Neumann and Tuschmann \cite{Neumann2025MercerYoungTheoremForMatrixKernel}.
A different generalization of Mercer's theorem by Ferreira and Menegatto \cite{Ferreira2009EigenvaluesOfIntegralOperatorsDPKernels} 
emphasized the connection between positive definite kernels, their integrability over the diagonal, and the nuclearity of the associated $L^2$-integral 
operators.
Similarly, Steinwart and Scovel \cite{Steinwart2012MercersTheoremOnGeneralDomains} and Steinwart \cite{Steinwart2019GenericKL-Expasions} explored the extend
to which topological assumptions on the domain can be relaxed while ensuring that Mercer-like expansions and KL-representations still converge 
(almost everywhere) pointwise or uniformly.    
The notion of a KL-representation has also been recently adapted to random measures by Vergara \cite{Vergara2026KarhunenLoeveForRandomMeasures} and 
to point processes by Picard et al. \cite{Picard2026PCAforPointProcesses}.\\
The term-wise differentiability of \eqref{Eq:MercerExpansion_Intro} and \eqref{Eq:KL-Representation_Intro} was first investigated by 
Kadota \cite{Kadota1967TermByTermDiffMercer, Kadota1967DifferentiationKL-Expansion}. 
The decay rate of eigenvalues for the operator \eqref{Eq:AssociatedOperator_Intro} was linked to the kernel's order of differentiability
by Reade \cite{Reade1983EigenvaluesOfPDKernels, Reade1983EigenvaluesOfPDKernelsII, Reade1986PositiveDefiniteC^pKernel}, 
Ha \cite{Ha1986EigenvaluesDifferentiableDPKernel} and later, among others, 
by Buesco \cite{Buescu2007EigenvalueDistrMercerLikeKernels}. 
Most recently, Takhanov \cite{Takhanov2023SpeedOfConvergenceMercer} 
demonstrated how kernel smoothness affects the rate of uniform convergence 
by deriving the first explicit bounds using the powerful Gagliardo-Nirenberg inequality.

\paragraph{Contributions and outline} 
As discussed in the introduction, Mercer expansions of the form \eqref{Eq:MercerExpansion_Intro} and the 
related KL-representations \eqref{Eq:KL-Representation_Intro} fundamentally rely on the spectral decomposition of the 
classic $L^2$-integral operator \eqref{Eq:AssociatedOperator_Intro}. 
The objective of this work is to transcend this limitation, enabling a more refined analysis for (weakly) differentiable kernels.\\
In section \ref{Sec:KernelExpansions}, we show that the classic $L^2$-integral operator is a 
special case within a broader class of so called kernel operators $T_{h,k}$ of order $k\in\N_0$ (see \eqref{Eq:AssociatedOperator} for the definition). 
These operators can be associated with any sufficiently weakly differentiable kernel $h\colon \Theta\times \Theta \to \R$ on a bounded domain $\Theta \subset \R^d$.
This framework allows to incorporate the additional information encoded in the weak derivatives of a kernel more accurately and we show 
(Lemma \ref{Lem:CharacterizationHS-Opertors}), that operators of this type correspond precisely to the class of Hilbert-Schmidt mappings between Sobolev spaces.\\ 
Leveraging the spectral theory for compact operators, we utilize the resulting eigenvalues and eigenvectors 
to construct series representations of the form \eqref{Eq:MercerExpansion_Intro} for symmetric kernels, which 
turn out to be \emph{optimal for simultaneously approximating both the kernel and its derivatives}.
If $k>d$ and the boundary of $\Theta \times\Theta$ satisfies the cone property, we derive (Theorem \ref{Thrm:KernelExpansion}) uniform convergence of the series 
- including derivatives of order less than $k-d$ - from a suitable Sobolev embedding, notably without requiring positive definiteness.  
Furthermore, we also show that $k>d/2$ is sufficient to ensure the universal uniform boundedness of the eigenfunctions. 
These mark the first positive findings of this kind for arbitrary symmetric indefinite kernels.\\   
Without relying on a Sobolev embedding we then establish a significant refinement of Mercer's theorem for strongly differentiable positive 
definite kernels (Theorem \ref{Thrm:Mercer}) by showing uniform convergence of all term-wise derivatives independent of the dimension and the geometry of the domain.
Provided $\Theta$ does satisfy the cone property and $k>d/2$, the universal boundedness of the eigenfunctions translates into rates for the uniform convergence based 
solely on eigenvalue decay (Corollary \ref{Cor:MercerRateOfConvergence}) in this case - including derivatives of order less than $k-d/2$.  
Moreover, we confirm non-negativity and nuclearity of higher-order kernel operators (Corollary \ref{Cor:NuclearNorm}),
leading to a direct extension of the trace formula \eqref{Eq:TraceFormula_Intro}.\\
As an immediate consequence (Theorem \ref{Thrm:RKHS_SpectralRep}), this leads to novel spectral representations for the reproducing kernel Hilbert space (RKHS) 
associated with a positive definite kernel.\\
In section \ref{Sec:ApplicatonStochasticProcesses}, we also present a second application by linking higher-order kernel operators to covariance operators of weakly differentiable 
random fields:
We establish (Lemma \ref{Lem:CovarianceOperator}) a generalized version of the variance identity \eqref{Eq:TotalVariance_Intro}
and refine the KL-representation \eqref{Eq:KL-Representation_Intro} accordingly (Theorem \ref{Thrm:KLExpansionSobolev}), 
including specific rates (Corollary \ref{Cor:KLExpansionSobolev}) and path-wise convergence for Gaussian elements (Corollary \ref{Cor:KLGaussian}).\\  
Finally, we apply this theory to the integrated Brownian motion process (Examle \ref{Ex:IntegratedBM}). 
By analytically deriving its $H^1$-eigenvalues and -eigenfunctions (Example \ref{Ex:KLExpansionSobolev}), we provide a 
direct comparison to standard $L^2$-theory, undistorted by numerical inaccuracies.

\section{Preliminaries}\label{Sec:Preliminaries}
To facilitate a smooth transition and avoid lengthy digressions in the subsequent sections, we consolidate essential notation 
some theoretical background.\\

Let $d \in \N=\{1,2,\ldots\}$. Throughout this text, the symbol $\Theta$ is reserved for a $d$-dimensional bounded domain - that is, an open and 
bounded subset of $\R^d$. 
For a point $x=(x_1,\ldots,x_d)\in\R^d$, let $\|x\|_2 = (\sum_{i\leq d} |x_i|^2)^{1/2}$ denote its euclidean 
norm. We define $\BB_{\epsilon}(x) := \{ y\in\R^d\mid \|x-y\|_2<\epsilon \}$ as the open ball of radius $\epsilon>0$ centred at $x$, 
writing $\BB_\epsilon$ as shorthand for $\BB_\epsilon(0)$.
For any sets $G,O\subseteq \R^d$, $\ol{G}$ denotes the closure of $G$ in $\R^d$ and
the notation $G \subset \subset O$ indicates that $\ol{G}$ is compact (i.e. closed and bounded) and $\ol{G} \subseteq O$. 
A function $f\colon G \to \R$ is said 
to have compact support in $G$ if $\supp f = \ol{\{ x\in G \mid f(x)\not= 0 \}}\subset\subset G$. 
For a bounded function $f\colon G \to \R$, we denote its supremum by $\|f\|_{\infty} = \sup_{x\in G} |f(x)|$.\\
Let $\alpha=(\alpha_1,\ldots,\alpha_d) \in \N_0^d$, where $\N_0=\N\cup \{0\}$, be a multi-index 
with degree $|\alpha| = \sum_{i\leq d} \alpha_i$.
The statement ``\,let $|\alpha| \leq k$\,'' serves as shorthand for ``\,let $\alpha=(\alpha_1,\ldots,\alpha_d) \in \N_0^d$ be such that $|\alpha| 
\leq k$\,''. 
Furthermore, let $\partial_j = \partial/\partial x_j$ denote the partial derivative with respect to $x_j$. We then write  
$\partial^{[\alpha]} = \partial_1^{\alpha_1}\ldots\partial_d^{\alpha_d}$ for a differential operator 
of order $|\alpha|$. Specifically, for a function $f\colon \R^d \supset G \to \R$ sufficiently differentiable at $x\in G$: 
\begin{align*}
	f^{[\alpha]}(x) 
	:= \partial^{[\alpha]}f(x) 
	= \frac{\partial^{\alpha_1}}{\partial x_1}\ldots\frac{\partial^{\alpha_d}}{\partial x_d}f(x_1,\ldots,x_d).
\end{align*}
We refer to any function $h\colon \R^{2d} \supset \Theta\times \Theta \to \R$ as a \emph{kernel}. 
A kernel $h\colon \Theta\times \Theta\to \R$ is called positive definite if it is symmetric (i.e., $h(x,y)=h(y,x)$) and 
satisfies \eqref{Eq:PositiveDefiniteKernel_Intro} for all $c_1,\ldots,c_n\in \R$, $x_1,\ldots,x_n \in \Theta$ and $n\in \N$.
For two multi-indices $\alpha,\beta \in \N_0^d$ the notation $\partial^{[\alpha,\beta]}h$ denotes the mixed-derivative 
\begin{align*}
	\partial^{[\alpha,\beta]}h(x,y) 
	= \frac{\partial^{\alpha_1}}{\partial x_1}\ldots\frac{\partial^{\alpha_d}}{\partial x_d}\,
	\frac{\partial^{\beta_1}}{\partial y_1}\ldots\frac{\partial^{\beta_d}}{\partial y_d}
	h(x_1,\ldots,x_d,y_1,\ldots,y_d).
\end{align*}

\paragraph{Operator theory}
Let $\LL(E,F)$ denote the space of bounded linear operators from a normed space $(E, \|\cdot\|_E)$ into 
another normed space $(F, \|\cdot\|_F)$ with $\LL(E)$ serving as shorthand for $\LL(E,E)$. 
If $F$ is a Banach space, then $\LL(E,F)$ is complete with respect to the operator norm $\|T\|_{\LL} = \sup \{\|Tu\|_F \mid\,\|u\|_E \leq 1\}$. 
For $T\in \LL(E,F)$ let $\ran(T)=\{ Tu \mid u\in E \} \subseteq F$ denote its range (or image), $\ker(T) = \{ u\in E\mid Tu=0 \} \subset E$ 
its kernel (or null-space) and $\rank(T) = \dim \ran (T) \in \N_0\cup\{\infty\}$ its rank.
An operator  $T\in\LL(E,F)$ is called compact if the image of the unit ball $\BB_E = \{ u\in E \mid \|u\|_E \leq 1 \}$
is relatively compact in $F$ - this is, if the closure $\ol{T(\BB_E)}$ is compact in $F$.\\ 
Let $E' = \LL(E,\R)$ denote the continuous dual space of $E$.
A subset $E_0' \subseteq E'$ is said to be separating if it separates the elements of $E$: for any $u \not=v \in E$, there 
exists a functional $\varphi \in E_0'$ such that $\varphi(u) \not= \varphi(v)$.\\ 
If $E=H$ is a Hilbert space with inner product $\<\cdot,\cdot\>_H$, then $H$ is identified with its dual $H'$ via the Riesz 
representation theorem: for every $\varphi\in H'$ exists a unique $v\in H$ such that $\varphi(u) = \<u,v\>_H$ and $\|\varphi\|_\LL = \|v\|_H$. 
For any $T\in \LL(H,K)$, where $K$ is another Hilbert space, 
$T^* \in \LL(K,H)$ denotes the adjoint operator. 
An operator $T\in \LL(H)$ is said to be self-adjoint if $T = T^*$ and non-negative if $\<Tu,u\>_H \geq 0$ for all $u\in H$. 
By the spectral theorem (e.g., \cite{DunfordSchwarz1988Linear Operators}, \cite{Werner2018Funktionalanalysis}) every compact operator $T\in\LL(H,K)$ 
admits a canonical representation $T = \sum_j s_j\, v_j \otimes u_j$. Here,
$\{u_j\}$ and $\{ v_j \}$ are orthonormal bases for $\ol{\ran(T^*)}\subseteq H$ and $\ol{\ran(T)}\subseteq K$, respectively. 
The notation $v \otimes u\colon H \to K$ for $u\in H$ and $v\in K$ denotes the rank-one operator defined by $v\otimes u(f) := \<f,u\>_H \,v$.
The singular values $\{s_j\}$ are always assumed to be 
arranged in descending order, $s_1 \geq s_2 \geq \ldots \geq 0$, and satisfy $s_j \to 0$.
The triple $\{ (s_j,u_j,v_j) \mid j\in\N \}$ is called \emph{singular system} of $T$.
The decay rate of the singular values allows for the interpretation as a measure of regularity. A compact operator $T\in \LL(H,K)$ 
belongs to the $p$-Schatten-class $S_p(H,K)$ for $1\leq p < \infty$, if its $p$-Schatten-norm $\|T\|_p = (\sum_{j} |s_j|^p)^{1/p}$ is finite. 
Endowed with this norm $S_p(H,K)$ is a separable Banach space and $S_p(H) = S_p(H,H)$.
Two special classes are of particular importance: the nuclear (or trace-class) operators $N(H,K) = S_1(H,K)$ and the 
Hilbert-Schmidt operators $HS(H,K) = S_2(H,K)$. 
Let $\{e_j\mid j\in\N\}$ be any orthonormal basis of $H$. Then the Hilbert-Schmidt norm $\|\cdot\|_{HS} = \|\cdot\|_2$ is induced by 
the inner product $\<T,S\>_{HS} = \sum_j \<Te_j,Se_j\>_K$ making $HS(H,K)$ a separable Hilbert space. 
Moreover, for $T \in N(H)$, the trace is defined as $\tr(T) = \sum_j \<Te_j,e_j\>_H$ and if $T$ is non-negative, then $\|T\|_N = \|T\|_1 = \tr(T)$. 
Finally, we recall that $\|\cdot\|_{\LL} \leq \|\cdot\|_{HS} \leq \|\cdot\|_{N}$ and thus $N(H,K) \subset HS(H,K) \subset \LL(H,K)$.  

\paragraph{Spaces of continuous functions} 
For any $k\in \N_0$, let $C^k(\Theta)$ denote the space of 
$k$-times continuously differentiable functions on $\Theta$; in particular, $C(\Theta) \equiv C^0(\Theta)$ is the space 
of continuous functions on $\Theta$. 
We define $C^{\infty}(\Theta) = \bigcap_{k} C^k(\Theta)$ and denote by $C_c^{\infty}(\Theta)$ the 
subspace of functions in $C^{\infty}(\Theta)$ with compact support in $\Theta$.\\
Any uniformly continuous function $f \in C(\Theta)$ possesses a unique continuous extension to the closure $\ol\Theta$. 
Accordingly, $C^k(\ol\Theta)$ is defined as the vector space of all functions $f\in C^k(\Theta)$ such that
$f^{[\alpha]} = \partial^{[\alpha]}f$ is uniformly continuous on $\Theta$ for $|\alpha| \leq k$.
Equipped with the norm $\| f \|_{C^k(\ol\Theta)} = \sum_{|\alpha| \leq k} \|f^{[\alpha]}\|_\infty$ the space $C^k(\ol\Theta)$ is 
a Banach space. 
The slightly larger space $C_B^k(\Theta)$ consists of all functions $f\in C^k(\Theta)$ for which $f^{[\alpha]}$ is bounded 
but not necessarily uniformly continuous for $|\alpha|\leq k$. This space, equipped with the norm 
$\| f \|_{C_B^k(\Theta)} = \max_{|\alpha| \leq k} \|f^{[\alpha]}\|_\infty$, is also a Banach space and plays a 
central role in Sobolev embeddings (see Remark \ref{Rem:ContinuousRepresentatives}).
For the product domain $\Theta\times \Theta$, the spaces $C^k(\Theta \times \Theta)$, $C^{\infty}(\Theta \times \Theta)$ etc. are defined 
accordingly.

\paragraph{Weak derivatives and Sobolev spaces}
Let $L_{loc}^1(\Theta)$ and $L_{loc}^1(\Theta \times \Theta)$ denote the spaces of locally Lebesgue-integrable functions on $\Theta$ and  
$\Theta\times\Theta$, respectively. Motivated by partial integration, a function
$f\in L_{loc}^1(\Theta)$ is said to have a weak derivative $g \in L_{loc}^1(\Theta)$ of order $\alpha \in \N_0^d$ if, for all $\varphi \in C_c^{\infty}(\Theta)$, 
the following holds:
\begin{align}\label{Eq:WeakDerivativeFunction}
	\int_\Theta f(x) \varphi^{[\alpha]}(x) \,dx 
	= (-1)^{|\alpha|} \int_\Theta g(x) \varphi(x) \,dx. 
\end{align}
In this case, $g$ is unique in $L_{loc}^1(\Theta)$ (i.e. almost everywhere) and is denoted by $\partial^{(\alpha)}f$ or $f^{(\alpha)}$. 
Similarly, for two multi-indices $\alpha,\beta \in \N_0^d$, a kernel $h\in L_{loc}^1(\Theta\times\Theta)$ has an $(\alpha,\beta)$-weak derivative 
$g\in L_{loc}^1(\Theta\times\Theta)$ if, for all $\varphi \in C_c^{\infty}(\Theta\times\Theta)$:
\begin{align}\label{Eq:WeakDerivativeKernel}
	\int_\Theta \int_{\Theta} h(x,y) \partial^{[\alpha,\beta]}\varphi(x,y) \,dx\,dy 
	= (-1)^{|\alpha|+|\beta|} \int_\Theta g(x,y) \varphi(x,y) \,dx \,dy. 
\end{align}
Again, $g$ is unique in $L_{loc}^1(\Theta\times\Theta)$ and denoted by $\partial^{(\alpha,\beta)}h$. 
Note that if a weak derivative is continuous, then it coincides with the classical (strong) derivative.
Let $L^2(\Theta)$ denote the Lebesgue space of square-integrable functions on $\Theta$ with the inner product 
$\<f,g\>_{L^2} = \int_\Theta f(x)g(x) \,dx$.
For $k\in \N_0$, the Sobolev space $H^k(\Theta)$ is then defined as
\begin{align*}
	H^k(\Theta)
	= \{ f\in L^2(\Theta) \mid \partial^{(\alpha)}f ~\text{exists and}~ \partial^{(\alpha)}f \in L^2(\Theta) 
	~\text{for all}~|\alpha|\leq k \}.
\end{align*}
These subspaces of $L^2(\Theta) = H^{0}(\Theta)$ are separable Hilbert spaces with inner product
\begin{align*}
	\<f,g\>_{H^k} 
	= \sum_{|\alpha|\leq k} \<f^{(\alpha)}, g^{(\alpha)}\>_{L^2} 
	= \sum_{|\alpha|\leq k} \int_\Theta f^{(\alpha)}(x)\, g^{(\alpha)}(x) \, dx.
\end{align*}
For the product domain $\Theta\times \Theta$ the Sobolev spaces $H^k(\Theta\times \Theta)$ are defined accordingly. Moreover, 
for $k_\alpha, k_{\beta} \in \N_0$ we also define the Sobolev space of mixed order
\begin{align}\label{Eq:SobolevSpaceMixedOrder}
		H^{k_\alpha,k_\beta}(\Theta \times \Theta)
		:= \big\{ h\in L^2(\Theta \times \Theta) \mid \partial^{(\alpha,\beta)} h~\text{exists and}~\partial^{(\alpha,\beta)}h \in L^2 
		~\text{for}~|\alpha| \leq k_{\alpha},~|\beta|\leq k_{\beta}  \big\}.
\end{align}
This is also a Hilbert space when endowed with the inner product 
\begin{align}\label{Eq:InnerProductMixedOrder}
	\< h,g \>_{H^{k_\alpha,k_\beta}} 
	= \sum_{|\alpha| \leq k_{\alpha}} \sum_{|\beta| \leq k_{\beta}} \< \partial^{(\alpha,\beta)}h, \partial^{(\alpha,\beta)}g \>_{L^2},
\end{align}
see Lemma \ref{Lem:U^kIsAhilbertSpace}.
It should be emphasized that the elements of such Sobolev spaces are, strictly speaking, equivalence classes of functions
that are equal almost everywhere (a.e.). However, following common practice, we will use $f\in H^{k}(\Theta)$
to refer to both the equivalence class and, where appropriate, a specific representative.

\begin{Rem}[The space $H_\kappa(\Theta)$]\label{Rem:H_Kappa}
From an algebraic or topological point of view it can be convenient to represent $f \in H^k(\Theta)$ as a collection of its weak derivatives 
$(f^{(\alpha)})_{|\alpha|\leq k}$, that is, as an element of the product space
$L_\kappa^2(\Theta) = \prod_{j=1}^\kappa L^2(\Theta)$, where $\kappa=\kappa(k,d)=\binom{k+d}{k}$ denotes the cardinality of 
$\{ \alpha \in \N_0^d\mid |\alpha| \leq k \}$. 
Endowed with the norm $\|f\|_{L_\kappa^2}^2 = \sum_{j=1}^\kappa \|f_j\|_{L^2}^2$ for $f=(f_1,\ldots,f_\kappa)$, $L_\kappa^2(\Theta)$ is a (direct sum) 
Hilbert space. 
Since the embedding $\J\colon H^k(\Theta) \to L_\kappa^2(\Theta)$ defined by  $\J f = (f^{(\alpha)})_{|\alpha| \leq k}$ preserves the norm 
$\|\J f\|_{L_\kappa^2} = \|f\|_{H^k}$, it is also an isometric isomorphism from $H^k(\Theta)$ onto its image
\begin{align*}
	H_{\kappa}(\Theta) := \J(H^k(\Theta)) \subset L_\kappa^2(\Theta).
\end{align*}  
Because $H^k(\Theta)$ is complete and $L_\kappa^2(\Theta)$ is a separable Hilbert space, $H_\kappa(\Theta)$ is likewise a separable Hilbert space.
\end{Rem}
    
\paragraph{Sobolev embeddings} The prominence and utility of Sobolev spaces stems primarily from their embedding characteristics 
(e.g., \cite{AdamsFournier1977ConeCondition}, \cite{Adams1975SobolevSpace}). 
These embeddings guarantee, among other things, the existence of continuous representatives 
for sufficiently smooth domains, provided that certain dimension-dependent regularity conditions are met. 
A basic assumption on the domain $\Theta$ in this context is the following: 
\begin{Assum}[Cone property]
Given two points $x,z\in\R^d$ and two open balls $\BB_{\epsilon}(x)$ and $\BB_{\tau}(z)$ such that $x\notin\BB_{\tau}(z)$, the set 
$C_x = \BB_{\epsilon}(x) \cap \{ x+\lambda(y-x) \mid y\in \BB_{\tau}(z), \lambda > 0\}$ is called a \emph{finite cone} in $\R^d$ with vertex $x$. 
A bounded domain $\Theta$ is said to have the \emph{cone property} if there exists a finite cone $C$ such that every point $x\in\Theta$ is the 
vertex of a finite cone $C_x \subseteq \Theta$ that is congruent to $C$ (i.e. $C_x$ is obtained from $C$ via rigid motion).
\end{Assum}

It is worth noting that common domains, such as rectangular or, more generally, convex domains, satisfy this property.
Furthermore, all results in this work requiring the cone property remain valid under the measure theoretic \emph{weak cone condition}
introduced in \cite{AdamsFournier1977ConeCondition}, which is even less restrictive.     

\begin{Rem}[Continuous representatives]\label{Rem:ContinuousRepresentatives}
Suppose $\Theta$ has the cone property and $k-d/2 > s$ for some $s\in \N_0$. 
By the Sobolev embedding theorem (e.g. \cite{Adams1975SobolevSpace}), $H^{k}(\Theta)$ is then 
continuously embedded into $C_B^s(\Theta)$. This means each $f\in H^{k}(\Theta)$ can be redefined on a set of measure zero 
in such way, that the redefined function $\wt f$ (which equals $f$ in $H^{k}$) belongs to $C_B^s(\Theta)$.
Specifically, $\wt f$ satisfies $\| \wt f \|_{C_B^s} \leq C \|f\|_{H^{k}}$, where the embedding constant $C>0$ depends on $\Theta$ only 
through the dimension $d$ and various parameters of the cone $C$.
\end{Rem}

\section{Sobolev-Mercer Expansions}\label{Sec:KernelExpansions}
This section presents a systematic approach to kernel expansions in the Sobolev spaces $H^k(\Theta)$.
These investigations are grounded in an explicit characterisation of Hilbert-Schmidt operators between those spaces in terms of
weakly differentiable kernels. By combining this characterization with
the spectral theorem, we obtain series representations for (weakly) differentiable kernels that extend and improve 
upon the classical Mercer expansion. 

\subsection{Higher Order Kernel Operators}\label{SubSec:HS-OperatorsOnSobolevSpaces}
Let $k_\alpha, k_\beta \in \N_0$ and suppose $h\in H^{k_\alpha,k_\beta}(\Theta \times \Theta)$. 
We can associate with the kernel $h$ an operator $T_h\colon  H^{k_\alpha}(\Theta) \to L^2(\Theta)$ defined via the rule, 
\begin{align}\label{Eq:AssociatedOperator}
	[T_h f] (y) = \int_\Theta \sum_{|\alpha|\leq k_\alpha} \partial^{(\alpha,0)}h(x,y)f^{(\alpha)}(x) \,dx \quad\text{a.e.}
\end{align} 
Indeed, $T_h$ is well-defined, since applying the Cauchy-Schwarz inequality twice yields
\begin{align*}
	\|T_h f\|_{L^2}^2 
	\leq \sum_{|\alpha|\leq k_\alpha}  \| \partial^{(\alpha,0)} h \|_{L^2}^2 \, \|f\|_{H^{k_\alpha}}.
\end{align*}
This shows $T_h \in \LL(H^{k_\alpha} ,L^2)$ with $\|T_h\|_{\LL} \leq \sum_{|\alpha|\leq k_\alpha} \| \partial^{(\alpha,0)} h \|_{L^2}$.
Henceforth, $T_h$ is referred to as a \emph{kernel operator of order $(k_\alpha,k_\beta)$} associated to $h$. 
To prevent ambiguity in cases where the order may not be immediately clear from the context, we use the explicit notation $T_h = T_{h,k_\alpha,k_\beta}$ 
to indicate a specific order. 
Note that within this framework, the classic $L^2$-integral operator $T_{h,0}=T_{h,0,0}$ of the form \eqref{Eq:AssociatedOperator_Intro} associated to 
$h\in L^2(\Theta\times\Theta)$ emerges as special case. A more detailed analysis reveals the following basic properties:

\begin{Prop}\label{Prop:AssociatedOperator}
Let $k_\alpha, k_\beta \in \N_0$ and $h\in H^{k_\alpha,k_\beta}(\Theta \times \Theta)$.  Then the associated kernel operator $T_h$ of order $(k_\alpha,k_\beta)$
defined via \eqref{Eq:AssociatedOperator} is Hilbert-Schmidt, that is, $T_h \in HS(H^{k_\alpha}, H^{k_\beta})$ such that 
\begin{itemize}
	\item[(a)]  $\displaystyle [T_hf]^{(\beta)}(y) 
	= \int_\Theta \sum_{|\alpha| \leq k_{\alpha}} \partial^{(\alpha,\beta)} h(x,y) \, f^{(\alpha)}(x) \,dx ~\text{a.e. for} 
	~ |\beta| \leq k_{\beta}$;\refstepcounter{equation}\hfill $(\theequation)$ \label{Eq:DerivativeThf}\\
	\item[(b)] $\displaystyle \|T_h\|_{HS(H^{k_\alpha}, H^{k_\beta})}^2 
	= \sum_{|\alpha|\leq k_\alpha} \sum_{|\beta|\leq k_\beta} \big\| \partial^{(\alpha,\beta)}h \big\|_{L^2}^2 < \infty$. 
	\refstepcounter{equation}\hfill $(\theequation)$ \label{Eq:HilbertSchmidtNorm}
\end{itemize}
\end{Prop}

\begin{proof}
(a) Let $f\in H^{k_\alpha}(\Theta)$. The weak derivatives of $T_hf$, if they exist, are determined by the requirement \eqref{Eq:WeakDerivativeFunction}, 
that is, 
\begin{align}\label{Eq:DerivativeOfThf_Characterization}
	\int_\Theta [T_hf]^{(\beta)}(y)\, \psi(y) \,dy
	&=\int_\Theta \bigg( \int_\Theta \sum_{|\alpha|\leq k_\alpha} \partial^{(\alpha,0)} h(x,y) f^{(\alpha)}(x) \,dx \bigg)^{(\beta)}\,\psi(y) \,dy \nonumber\\
	&= (-1)^{|\beta|} \int_\Theta \int_\Theta \sum_{|\alpha|\leq k_\alpha} \partial^{(\alpha,0)} h(x,y) f^{(\alpha)}(x) \,dx \, \psi^{[\beta]}(y) \,dy
\end{align}
for all $\psi \in C_c^{\infty}(\Theta)$. Moreover, since $h \in H^{k_\alpha,k_\beta}(\Theta \times \Theta)$, for $\phi,\psi \in C_c^{\infty}(\Theta)$, 
\begin{align*}
	\int_\Theta \int_\Theta \partial^{(\alpha,\beta)} h(x,y)\, \psi(y) \,dy \,\phi(x) \,dx
	= \int_\Theta (-1)^{|\beta|} \int_\Theta \partial^{(\alpha,0)} h(x,y)\, \psi^{[\beta]}(y) \,dy \,\phi(x) \,dx,
\end{align*}
which implies, for almost all $x\in \Theta$, 
\begin{align*}
	\int_\Theta \partial^{(\alpha, \beta)} h(x,y)\, \psi(y) \,dy
	= (-1)^{|\beta|} \int_\Theta \partial^{(\alpha,0)} h(x,y)\, \psi^{[\beta]}(y) \,dy.
\end{align*}
In view of \eqref{Eq:DerivativeOfThf_Characterization}, employing Fubini's theorem once more, this in turn yields
\begin{align*}
	(-1)^{|\beta|} \int_\Theta \int_\Theta \sum_{|\alpha|\leq k_\alpha} \partial^{(\alpha,0)} h(x,y) \, f^{(\alpha)}(x) \,dx \, \psi^{[\beta]}(y) \,dy
	= \int_\Theta \int_\Theta \sum_{|\alpha|\leq k_\alpha} \partial^{(\alpha,\beta)} h(x,y) \, f^{(\alpha)}(x) \,dx\, \psi(y) \,dy
\end{align*}
from which \eqref{Eq:DerivativeThf} follows immediately.\\
(b) Due to Fubini's theorem, the restriction $h_{y,\beta}\colon \Theta \to \R$ defined by $h_{y,\beta}(x):= \partial^{(0,\beta)}h(x,y)$,
$|\beta| \leq k_{\beta}$, is square-integrable for almost all $y\in\Theta$ and by assumption, for $\phi,\psi \in C_c^\infty(\Theta)$ and $|\alpha| \leq k_\alpha$, 
\begin{align*}
	\int_\Theta \int_\Theta h_{y,\beta}(x) \varphi^{[\alpha]}(x) \,dx \, \psi(y) \,dy
	= \int_\Theta (-1)^{|\alpha|} \int_\Theta \partial^{(\alpha,\beta)} h(x,y) \varphi(x) \,dx \, \psi(y) \,dy.
\end{align*} 
This shows $h_{y,\beta} \in H^{k_\alpha}(\Theta)$ with weak derivatives 
$h_{y,\beta}^{(\alpha)} = \partial^{(\alpha,\beta)}h(\cdot,y)$ for almost all $y\in \Theta$.
Therefore, using Parseval's identity, for any orthonormal basis $\{e_j\mid j\in\N\}$ of $H^{k_{\alpha}}(\Theta)$, 
\begin{align*}
	\sum_j \|T_h e_j\|_{H^{k_\beta}}^2 
	&= \sum_j \int_\Theta \sum_{|\beta| \leq k_\beta} \bigg| \int_\Theta \sum_{|\alpha|\leq k_\alpha} \partial^{(\alpha,\beta)}h(x,y)\, 
		e_j^{(\alpha)}(x) \,dx \bigg|^2 \,dy\\
	&= \int_\Theta \sum_{|\beta|\leq k_\beta} \sum_j \big| \< \partial^{(0,\beta)}h(\cdot, y), e_j \>_{H^{k_\alpha}} \big|^2 \,dy\\
	&= \int_\Theta \sum_{|\beta|\leq k_\beta} \big\| \partial^{(0,\beta)} h(\cdot, y) \big\|_{H^{k_\alpha}}^2 \,dy\\
	&= \sum_{|\alpha|\leq k_\alpha} \sum_{|\beta|\leq k_\beta} \big\| \partial^{(\alpha,\beta)}h \big\|_{L^2}^2.
\end{align*}
Consequently, $T_h\colon H^{k_\alpha}(\Theta) \to H^{k_\beta}(\Theta)$ is Hilbert-Schmidt satisfying \eqref{Eq:HilbertSchmidtNorm}.
\end{proof}

\paragraph{Characterisation of Hilbert-Schmidt operators} Notably, the converse of Proposition \ref{Prop:AssociatedOperator} is also true: 
Every $T\in HS(H^{k_\alpha}, H^{k_\beta})$ corresponds via \eqref{Eq:AssociatedOperator} to a kernel $h \in H^{k_\alpha,k_\beta}(\Theta\times\Theta)$.
Since this elementary relation has not been explicitly recorded in this form, we also include a short canonical proof. 

\begin{Lem}\label{Lem:CharacterizationHS-Opertors}
Let $k_{\alpha},k_{\beta}\in \N_0$ and $T \in \LL(H^{k_\alpha}, H^{k_\beta})$. The following statements are equivalent: 
\begin{enumerate}
	\item[$\mathrm{(i)}$] $\displaystyle T \in HS(H^{k_\alpha}, H^{k_\beta})$
	\item[$\mathrm{(ii)}$] There exists a kernel $h \in H^{k_\alpha,k_\beta}(\Theta \times \Theta)$ such that $T = T_h$ via \eqref{Eq:AssociatedOperator}.
\end{enumerate}
In this case, \eqref{Eq:HilbertSchmidtNorm} holds true.
\end{Lem}

\begin{proof}
In view of Proposition \ref{Prop:AssociatedOperator} it remains to verify the implication (i) $\Longrightarrow$ (ii): Note, that 
$U^{k_\alpha,k_\beta}(\Theta\times \Theta)$ is a Hilbert space by Lemma \ref{Lem:U^kIsAhilbertSpace} and, therefore, due to \eqref{Eq:HilbertSchmidtNorm} 
the mapping 
\begin{align*}
	\Phi\colon U^{k_\alpha,k_\beta}(\Theta \times \Theta) \to  HS(H^{k_\alpha}, H^{k_\beta}), \qquad h\mapsto T_h
\end{align*}
defines a linear isometry.  
Let $T \in HS(H^{k_\alpha}, H^{k_\beta})$ have finite rank. 
Then $T$ is in particular compact and, thus, has a singular system $\{(s_j,u_j,v_j)\mid j=1,\ldots,n\}$ such that
\begin{align*}
	T = \sum_{j\leq n} s_j \, v_j \otimes u_j = \sum_{j\leq n} s_j \< \cdot, u_j \>_{H^{k_\alpha}} \, v_j.  
\end{align*}
It follows that, for $f \in H^{k_\alpha}(\Theta)$, almost everywhere
\begin{align*}
	[Tf](y) 
	= \sum_{j\leq n} s_j \< f, u_j \>_{H_{k^\alpha}} \, v_j(y)
	= \int_\Theta \sum_{|\alpha|\leq k_\alpha} \sum_{j\leq n} s_j \,u_j^{(\alpha)}(x)\, v_j(y)\, f^{(\alpha)}(x) \,dx,
\end{align*} 
which shows $T$ is the kernel operator of order $(k_\alpha,k_\beta)$ associated to the kernel $h_n(\cdot,\cdot) = \sum_{j\leq n} s_j\, u_j(\cdot) v_j(\cdot)$. 
Moreover, 
\begin{align*}
	\sum_{|\alpha|\leq k_\alpha} \sum_{|\beta|\leq k_\beta} \| \partial^{(\alpha,\beta)}h_n\|_{L^2}^2
	= \sum_{i,j\leq n} s_i s_j \, \< u_i,u_j \>_{H^{k_\alpha}} \< v_i,v_j \>_{H^{k_\beta}}
	= \sum_{j\leq n} s_j^2 
\end{align*}
and, therefore, $T=T_{h_n} \in  HS(H^{k_\alpha}, H^{k_\beta})$ with 
\begin{align*}
	\|T\|_{ HS(H^{k_\alpha}, H^{k_\beta})}^2 
	= \sum_{|\alpha|\leq k_\alpha} \sum_{|\beta|\leq k_\beta} \| \partial^{(\alpha,\beta)}h_n\|_{L^2}^2.
\end{align*} 
Now, since the finite rank operators are dense in $HS(H^{k_\alpha}, H^{k_\beta})$ with respect to $\|\cdot\|_{HS(H^{k_\alpha}, H^{k_\beta})}$ 
(e.g., \cite{Werner2018Funktionalanalysis}, Theorem VI.6.2), this implies 
$\mathrm{ran}(\Phi)\subseteq HS(H^{k_\alpha}, H^{k_\beta})$ 
is dense. But the image of a linear isometry is also closed and hence $\mathrm{ran}(\Phi) = \overline{\mathrm{ran}}(\Phi) = HS(H^{k_\alpha}, H^{k_\beta})$, 
which proves the assertion.  
\end{proof}

\begin{Rem}[Relation to $L^2$-integral operators]\label{Rem:HSNorm}
Recall that $T_{h,0}\colon L^2(\Theta) \to L^2(\Theta)$ denotes the $L^2$-integral operator associated to a kernel $h\in L^2(\Theta\times \Theta)$, that is, 
$T_{h,0}f = \int_{\Theta} h(x,\cdot) f(x) \,dx$ almost everywhere. In particular, $T_{\partial^{(\alpha,\beta)}h,0} \in HS(L^2)$  
for a kernel $h \in U^{k_\alpha,k_\beta}(\Theta\times \Theta)$ and equation \eqref{Eq:HilbertSchmidtNorm} can be restated as
\begin{align}\label{Eq:HilbertSchmidtNorm2}
	\|T_h\|_{HS(H^{k_\alpha}, H^{k_\beta})}^2 = \sum_{|\alpha|\leq k_\alpha} \sum_{|\beta|\leq k_\beta} \| T_{\partial^{(\alpha,\beta)}h,0} \|_{HS(L^2)}^2.
\end{align}
\end{Rem}

\begin{Rem}[Hilbert-Schmidt property of embedding maps]\label{Rem:EmbeddingsHSProperty}
Suppose $\Theta$ has the cone property and $k_{\alpha} > k_{\beta}$, such that $k_\alpha - k_\beta > d/2$. Then each $f \in H^{k_\alpha}(\Theta)$ 
has a representative in $C_B^{k_\beta}(\Theta)$ (see Remark \ref{Rem:ContinuousRepresentatives}) and it was first shown by Maurin
\cite{Maurin1961AbbildungenVomHilbert-SchmidtschenTyp} and later 
generalized by Clark \cite{Clark1966HilbertSchmidtPropertyForEmbeddings} that the embedding 
$\J\colon H^{k_\alpha}(\Theta) \to H^{k_\beta}(\Theta)$ is Hilbert-Schmidt in this case.
That is, for any orthonormal basis $\{e_j\mid j\in\N\}$ of $H^{k_\alpha}(\Theta)$,
\begin{align*}
	\|\J\|_{HS(H^{k_\alpha},H^{k_\beta})}^2
	= \sum_j \| \J e_j \|_{H^{k_\beta}}^2
	= \sum_j \| e_j \|_{H^{k_\beta}}^2 < \infty.
\end{align*}
Accordingly, Lemma \ref{Lem:CharacterizationHS-Opertors} implies existence 
of a kernel $h_\J \in U^{k_\alpha,k_\beta}(\Theta \times \Theta)$ such that, for $f \in H^{k_\alpha}(\Theta)$,
\begin{align*}
	f(y) 
	= [\J f](y) 
	=[T_{h_\J} f](y)
	= \int_\Theta \sum_{|\alpha|\leq k_\alpha} \partial^{(\alpha,0)}h_{\J}(x,y)f^{(\alpha)}(x) \,dx
\end{align*}
and, moreover, 
\begin{align*}
	\sum_{|\alpha|\leq k_\alpha} \sum_{|\beta|\leq k_\beta} \big\| \partial^{(\alpha,\beta)}h_\J \big\|_{L^2}^2
	= \| T_{h_\J} \|_{HS}^2
	= \| \J \|_{HS}^2
	= \sum_j \| e_j \|_{H^{k_\beta}}^2.
\end{align*}
This allows to conclude $h_\J(x,y) = \sum_j e_j(x)\,e_j(y)$, where the series on the right hand side converges in 
$H^{k_\alpha,k_\beta}(\Theta\times\Theta)$, because for $m<n$,
\begin{align*}
	\big\| \sum_{j=m}^n e_j(\cdot)\,e_j(\cdot) \big\|_{U^{k_\alpha,k_\beta}}^2
	= \sum_{i,j=m}^n \< e_i,e_j \>_{H^{k_\alpha}} \< e_i,e_j \>_{H^{k_\beta}}
	= \sum_{j=m}^n \| e_j \|_{H^{k_\beta}}^2.
\end{align*}
\end{Rem}

\paragraph{Adjoint operator}
According to Lemma \ref{Lem:CharacterizationHS-Opertors} every $T \in HS(H^{k_\alpha}, H^{k_\beta})$ corresponds to a kernel 
$h \in  H^{k_\alpha,k_\beta}(\Theta\times\Theta)$ via \eqref{Eq:AssociatedOperator}. 
The kernel $h^*$ corresponding to the adjoint operator $T^* \in HS(H^{k_\beta}, H^{k_\alpha})$ is then related to $h$ as follows:

\begin{Lem}\label{Lem:AdjointOperator}
Let $k_\alpha,k_\beta \in \N_0$ and $h \in H^{k_\alpha,k_\beta}(\Theta\times \Theta)$. Let $T_h \in HS(H^{k_\alpha}, H^{k_\beta})$ denote the 
associated kernel operator of order $(k_\alpha,k_\beta)$ defined via \eqref{Eq:AssociatedOperator}. 
Then the adjoint operator $T_h^* \in HS(H^{k_\beta}, H^{k_\alpha})$ is the kernel operator of order $(k_\beta,k_\alpha)$ associated to the kernel 
$h^*\in H^{k_\beta,k_\alpha}(\Theta\times \Theta)$ given by $h^*(x,y) := h(y,x)$ almost everywhere; that is, for $g\in H^{k_\beta}(\Theta)$,  
\begin{align}\label{Eq:AdjointOfT_h_Version2}
	T_h^* g(x) 
	= T_{h^*}g(x)
	= \int_\Theta \sum_{|\beta| \leq k_\beta} \partial^{(\beta,0)}h^*(y,x) \, g^{(\beta)}(y) \,dy \quad\text{a.e.}
\end{align}
\end{Lem}

\begin{proof}
First note that due to \eqref{Eq:DerivativeThf}, for 
$f\in H^{k_\alpha}(\Theta)$ and $g\in H^{k_\beta}(\Theta)$,  
\begin{align*}
	\<T_hf,g \>_{H^{k_\beta}}
	&= \int_\Theta \sum_{\alpha\leq k_\alpha} \bigg[ \int_\Theta \sum_{|\beta \leq k_\beta} \partial^{(0,\beta)}h(x,y) g^{(\beta)}(y) \,dy \bigg]^{(\alpha)}
	\,f^{(\alpha)}(x) \,dx.
\end{align*} 
This shows the adjoint operator $T_h^*$ is also determined by the kernel $h$ via
\begin{align}\label{Eq:AdjointOfT_h_Version1}
	T_h^* g(x) = \int_\Theta \sum_{\beta\leq k_\beta} \partial^{(0,\beta)} h(x,y) \, g^{(\beta)}(y) \,dy \quad\text{a.e.}
\end{align}
Now, let $h^* \in L^2(\Theta\times\Theta)$ denote the kernel determined by setting $h^*(x,y) := h(y,x)$ almost everywhere and observe that for 
$|\alpha| \leq k_{\alpha}$, $|\beta| \leq k_{\beta}$ and $\varphi \in C_c^{\infty}(\Theta\times \Theta)$,
\begin{align*}
	(-1)^{|\alpha+\beta|} \int_\Theta \int_\Theta h^*(x,y)\, \partial^{[\beta,\alpha]} \varphi(x,y) \,dx\,dy
	= \int_\Theta \int_\Theta \partial^{(\alpha,\beta)} h(y,x) \, \varphi(x,y) \,dx\,dy.
\end{align*}
This yields $h^* \in U^{k_\beta, k_\alpha}(\Theta\times\Theta)$ with weak derivatives $\partial^{(\beta,\alpha)}h^*\in L^2$ given by 
\begin{align}\label{Eq:AdjointKernelDerivatives}
	\partial^{(\beta,\alpha)}h^*(x,y) = \partial^{(\alpha,\beta)}h(y,x) \qquad\text{a.e.}
\end{align}
for $|\beta| \leq k_\beta$ and $|\alpha| \leq k_\alpha$. Substituting \eqref{Eq:AdjointKernelDerivatives} into \eqref{Eq:AdjointOfT_h_Version1} then yields 
\eqref{Eq:AdjointOfT_h_Version2}.
\end{proof}

\subsection{Spectral Theory of Higher Order Kernel Operators}\label{SubSec:ExtensionsOfMercersTheorem}
The connection established between (weakly) differentiable kernels and Hilbert-Schmidt operators now enables the derivation of series representations
for such kernels.   
Specifically, the transition from the classical $L^2$-approach to a more refined analysis in the Sobolev spaces $H^k(\Theta)$ is accomplished as follows:\\
Let $k=k_\alpha=k_\beta \in \N_0$ and let $T_{h} = T_{h,k}$ denote the kernel operator of order $k$ associated with
$h \in H^{k,k}(\Theta\times\Theta)$, as defined in \eqref{Eq:AssociatedOperator}.
It follows from Lemma \ref{Lem:AdjointOperator} that if $h$ is symmetric, then $T_h \in HS(H^k)$ is self-adjoint.
In this case, the spectral theorem for self-adjoint compact operators (e.g., \cite{DunfordSchwarz1988Linear Operators}, \cite{Werner2018Funktionalanalysis})
ensures that, with respect to $\|\cdot\|_{H^k}$:
\begin{align}\label{Eq:SpectralRepresentation}
	T_h f = \sum_j \lambda_{j} \< f,e_{j} \>_{H^k} \, e_{j} \qquad \forall f\in H^k(\Theta),
\end{align}
where $\lambda_j=\lambda_{j,k} \in \R$ and $e_j=e_{j,k} \in H^k(\Theta)$ are the eigenvalues (counted by multiplicity) and corresponding
normalized eigenfunctions of $T_h=T_{h,k}$. That is,
\begin{align}\label{Eq:EigenEquation}
	[T_{h,k}e_{j,k}](y) 
	= \int_\Theta \sum_{|\alpha| \leq k} \partial^{(\alpha,0)} h(x,y) e_{j,k}^{(\alpha)}(x) \,dx
	= \lambda_{j,k} e_{j,k}(y) \quad\text{a.e.}
\end{align}
subject to $\|e_{j,k}\|_{H^k} = 1$.
The eigen equation \eqref{Eq:EigenEquation} serves as the foundation for deriving series representations of $h$ in terms of the
$H^k$-\emph{eigensystem} $\{ (\lambda_{j,k}, e_{j,k})\mid j\in\N \}$.
As this approach naturally suggests to leverage the fundamental embedding characteristics of Sobolev spaces (see Remark \ref{Rem:ContinuousRepresentatives}),
the following theorem demonstrates that both the eigenfunctions and the resulting kernel expansions exhibit significant regularity. 
In this regard, accounting for the derivatives within equation \eqref{Eq:EigenEquation} can be interpreted as an intrinsic regularisation of the eigenfunctions.  

\begin{Thrm}\label{Thrm:KernelExpansion}
Let $k\in \N_0$ and suppose $h\in H^{k,k}(\Theta\times\Theta)$ is symmetric.
Furthermore, let $\{ (\lambda_{j,k}, e_{j,k})\mid j\in\N \}$ denote the eigensystem 
of the associated kernel operator $T_{h,k} \in HS(H^k)$ defined via \eqref{Eq:AssociatedOperator}. Then 
\begin{align}\label{Eq:KernelExpansionHk}
	\min_{w\in H^{k,k}(\Theta\times\Theta),\atop \rank(T_{w,k})=m} \,
	\sum_{|\alpha|,|\beta|\leq k} \int_\Theta \int_\Theta | \partial^{(\alpha,\beta)}h(x,y) - \partial^{(\alpha,\beta)}w(x,y) |^2 \,dx dy
	= \sum_{j>m} \lambda_{j,k}^2,
\end{align}
where the minimum is achieved by $w(x,y) = \sum_{j\leq m} \lambda_{j,k} e_{j,k}(x) e_{j,k}(y)$.\\
If $\Theta$ has the cone property and $k-d/2 > s$ for some $s\in\N_0$, then $e_{j,k}\in C_B^s(\Theta)$ and 
there exists a constant $C=C(\Theta,k,d)>0$ such that 
\begin{align}\label{Eq:UniformBoundEigenfunctions}
	\max_{|\alpha|\leq s}\, \sup_{x\in\Theta} |e_{j,k}^{[\alpha]}(x)| \leq C \quad \forall j\in \N.  
\end{align}
Moreover, suppose $\Theta\times\Theta$ has the cone property and $k - d > \ell$ for some $\ell \in \N_0$. 
Then there exists a constant $C=C(\Theta\times\Theta,k,d)>0$ such that 
\begin{align}\label{Eq:KernelExpansionUniformConv}
	\begin{split}
		\max_{|\alpha| + |\beta| \leq \ell}\, \sup_{x,y\in\Theta} \, \Big| \partial^{[\alpha,\beta]}h(x,y) 
		- \sum_{j\leq m} \lambda_{j,k}\, e_{j,k}^{[\alpha]}(x)\,e_{j,k}^{[\beta]}(y) \Big| 
		&\leq C \Big( \sum_{j>m} \lambda_{j,k}^2 \Big)^{1/2}.
	\end{split}
\end{align}
\end{Thrm}

\begin{proof}
It follows immediately from \eqref{Eq:DerivativeThf}, \eqref{Eq:AdjointKernelDerivatives} and \eqref{Eq:EigenEquation} that, for $|\alpha| \leq k$,
\begin{align}\label{Eq:DerivativeT_hfAsInnerProduct}
	[T_hf]^{(\alpha)}(y) 
	= \< \partial^{(\alpha,0)}h(y,\cdot), f \>_{H^k} \qquad \text{a.e.}
\end{align}
and 
\begin{align}\label{Eq:DerivativeT_hfEigenfunction}
	[T_he_j]^{(\alpha)}(y) = \lambda_j\,e_j^{(\alpha)}(y) \qquad \text{a.e.}
\end{align}
Complementing $\{ e_j, j\in\N \}$ by an orthonormal system $S$ for $\ker(T_h)\subset H^k(\Theta)$ and recalling
$\partial^{(\alpha,0)} h(x,\cdot) = h_{x,\alpha} \in H^k(\Theta)$ for almost all $x\in \Theta$ then allows to conclude, that with respect to 
$\|\cdot\|_{H^k}$,
\begin{align*}
	\partial^{(\alpha,0)}h(x,\cdot)
	&= \sum_{j} \< \partial^{(\alpha,0)}h(x,\cdot), e_j \>_{H^k}\, e_j(\cdot) 
	+ \sum_{e \in S} \< \partial^{(\alpha)}h(x,\cdot), e\>_{H^k}\, e(\cdot)\\
	&= \sum_j [T_he_j]^{(\alpha)}(x)\, e_j(\cdot) 
	+ \sum_{e \in S} [T_h e]^{(\alpha)}(x)\, e(\cdot) \\
	&= \sum_j \lambda_j \, e_j^{(\alpha)}(x)\, e_j(\cdot).
\end{align*}
Now, this means precisely
\begin{align*}
	0 &= \big\| \partial^{(\alpha,0)}h(x,\cdot) - \sum_j \lambda_j \, e_j^{(\alpha)}(x)\, e_j(\cdot) \big\|_{H^k}^2\\
	&= \sum_{|\beta|\leq k} \int_\Theta \Big| \partial^{(\alpha,\beta)}h(x,y) 
	- \sum_j \lambda_j \, e_j^{(\alpha)}(x)\, e_j^{(\beta)}(y) \Big|^2 \,dy
\end{align*}
and, therefore, almost everywhere, for $|\alpha|,|\beta| \leq k$,
\begin{align}\label{Eq:PointwiseConvergence}
	\partial^{(\alpha,\beta)}h(x,y)
	= \sum_j \lambda_j\, e_j^{(\alpha)}(x)\, e_j^{(\beta)}(y).
\end{align}
Moreover, due to the orthogonality of the eigenfunctions,
\begin{align}\label{Eq:ApproxErrorInHk}
	&\sum_{|\alpha|,|\beta|\leq k} \int_\Theta\int_\Theta 
	\Big| \partial^{(\alpha,\beta)}h(x,y) 
	- \sum_{j\leq m} \lambda_j e_j^{(\alpha)}(x)e_j^{(\beta)}(y)  \Big|^2 \,dx\,dy\nonumber\\
	&= \int_\Theta \int_\Theta \sum_{|\alpha|,|\beta|\leq k}
	\Big| \sum_{j>m} \lambda_j e_j^{(\alpha)}(x)e_j^{(\beta)}(y)  \Big|^2 \,dx\,dy\nonumber\\
	&= \sum_{i,j>m} \lambda_i\,\lambda_j \,\< e_i,e_j \>_{H^k}^2 
	=\sum_{j>m} \lambda_j^2.
\end{align}
The optimality \eqref{Eq:KernelExpansionHk} of the expansion \eqref{Eq:PointwiseConvergence} thus follows from \eqref{Eq:ApproxErrorInHk}, 
Lemma \ref{Lem:CharacterizationHS-Opertors} and the Schmidt-Mirsky Theorem (e.g., \cite{Hsing2015FoundationsOfFDA}, Theorem 4.4.7), according to which for 
any $w\in H^{k,k}(\Theta\times\Theta)$ such that $\rank(T_{w,k})=m$,
\begin{align*}
	\| T_{h,k} - T_{w,k} \|_{HS(H^k)}^2
	\geq \| T_{h,k} - \sum_{j\leq m} \lambda_j \,e_j\otimes e_j \|_{HS(H^k)}.
\end{align*}
If $\Theta$ has the cone property and $k-d/2>s$ then by the Sobolev 
embeddings (see Remark \ref{Rem:ContinuousRepresentatives}) there exists a constant $C=C(\Theta,k,d)>0$ such that 
\begin{align*}
	\| e_{j} \|_{C_B^s} = \max_{|\alpha|\leq s} \| e_{j}^{[\alpha]} \|_{\infty} \leq C \|e_{j}\|_{H^k} = C,
\end{align*}
which establishes \eqref{Eq:UniformBoundEigenfunctions}. 
Finally, suppose $\Theta\times \Theta$ has the cone property and
$k-d>\ell$ for some $\ell\geq 0$. It then follows again from the Sobolev embeddings, that for some constant $C=C(\Theta\times\Theta,k,d)$,
\begin{align*}
	\max_{|\alpha| + |\beta| \leq \ell} \sup_{x,y\in\Theta} \, \big| \partial^{[\alpha,\beta]}h(x,y) 
	- \sum_{j\leq m} \lambda_{j}\, e_{j}^{[\alpha]}(x)\,e_{j}^{[\beta]}(y) \big|
	\leq C \Big(\sum_{|\alpha|,|\beta|\leq k} \big\| \partial^{(\alpha,\beta)}h 
	- \sum_{j\leq m} \lambda_{j}\, e_{j}^{(\alpha)}(\cdot)\,e_{j}^{(\beta)}(\cdot) \big\|_{L^2}^2 \Big)^{1/2}.
\end{align*}
The estimate \eqref{Eq:KernelExpansionUniformConv} thus follows from \eqref{Eq:KernelExpansionHk}.
\end{proof}

\begin{Rem}[Universal uniform bound on eigenfunctions and rate of convergence]\label{Rem:UniformBoundEigenfunction}
It is worth noting that the universal uniform boundedness \eqref{Eq:UniformBoundEigenfunctions} of eigenfunctions (and their derivatives) is often a 
desirable property.
While the $L^2$-eigenfunctions $e_{j,0}$ of several common kernels - such as the Gaussian or Brownian motion kernel - do satisfy
$\sup_j \|e_{j,0}\|_\infty < \infty$,
the classic $L^2$-approach lacks general conditions to guarantee this. 
Consequently, many applications must treat universal boundedness as an additional, unverified assumption. 
In this regard, the transition to $H^k$-eigenfunctions $e_{j,k}$ offers an alternative that essentially renders this assumption superfluous.\\
The constant $C>0$ in \eqref{Eq:UniformBoundEigenfunctions} corresponds to the embedding constant of $H^k(\Theta)$ into $C_B^s(\Theta)$ 
and an optimal value can in many cases be derived explicitly via a Green's function approach. 
For instance, for the open interval $\Theta = (0,b)$ and $k=1$, Marti \cite{Marti1983EvaluationOfTheLeastConstantInSobolevInequality} 
showed that $C=\sqrt{\coth(b)}$ is optimal (e.g., $C \approx 1.1459 ~\text{for}~ b=1$).
Similarly, the constant $C>0$ in \eqref{Eq:KernelExpansionUniformConv} represents the embedding constant of $H^k(\Theta \times \Theta)$ into 
$C_B^\ell(\Theta\times \Theta)$.\\
It may also be remarked that the error of approximation can be expressed explicitly using only the first $m$ eigenvalues, as
\eqref{Eq:HilbertSchmidtNorm} yields: 
\begin{align}\label{Eq:ApproxErrorL2}
\sum_{j>m} \lambda_{j,k}^2 
= \sum_{|\alpha|,|\beta|\leq k} \| \partial^{(\alpha,\beta)}h \|_{L^2}^2 - \sum_{j\leq m} \lambda_{j,k}^2.	
\end{align}
\end{Rem}

\paragraph{Mercer's theorem for differentiable kernels}
According to \eqref{Eq:KernelExpansionUniformConv}, the dimension-dependent regularity condition
$k-d > \ell \in \N_0$ already ensures that the expansion in terms of the $H^k$-eigensystem $\{ (\lambda_{j,k}, e_{j,k})\mid j\in\N \}$
converge uniformly to $h$ and it's (strong) derivatives $\partial^{[\alpha,\beta]}h$ up to order $\ell$ granted $\Theta\times \Theta$ has the cone property.  
Moreover, the approximation error is uniform across all admissible derivatives and can be quantified explicitly via the eigenvalues $\lambda_{j,k}$. 
In other words, provided $\Theta \times \Theta$ possesses the cone property, \emph{any} symmetric kernel $h$ can be 
expanded uniformly in terms of the $H^k$-eigensystem of the associated operator $T_{h,k}$ as long as $k\geq d+1$.
This stands in contrast to the classical Mercer's theorem, which requires the additional assumption of positive definiteness.
On the other hand, \eqref{Eq:KernelExpansionHk} implies that for all $|\alpha|,|\beta| \leq k$ the expansion 
\begin{align}\label{Eq:KernelDerivativeExpansionL2}
	\partial^{(\alpha,\beta)}h = \sum_j \lambda_{j,k} \,e_{j,k}^{(\alpha)}(\cdot)\,e_{j,k}^{(\beta)}(\cdot),
\end{align}
converges in $L^2(\Theta \times \Theta)$, regardless of the dimension and whether $\Theta\times \Theta$ has the cone property or not. 
Our following main theorem shows that \eqref{Eq:KernelDerivativeExpansionL2} will also converge uniformly for $|\alpha|,|\beta| \leq k$, 
if we additionally assume the kernel $h$ is strongly differentiable and \emph{positive definite}. 
In this context, Mercer's classic result emerges 
as a special case ($k=0$) and the theorem may be viewed as its natural refinement. In the interest of clarity, let 
\begin{align*}
	C^{k,k}(\ol\Theta \times \ol\Theta)
	:= \big\{ h\in C(\overline\Theta \times \overline\Theta) \mid \partial^{[\alpha,\beta]}h \in C(\overline\Theta \times \overline\Theta) 
	~\text{for}~|\alpha| ,|\beta|\leq k \big\},
\end{align*}
representing the class of uniformly continuous kernels $h\colon \Theta \times \Theta \to \R$ whose mixed partial 
derivatives $\partial^{[\alpha,\beta]}h$ exist and are uniformly continuous on $\Theta \times \Theta$ for $|\alpha|,|\beta|\leq k$. 
The proof of our refinement relies on the following lemma, which can either be established using the classic Mercer expansion or, more elegantly, 
via a probabilistic argument, see Section \ref{Sec:ApplicatonStochasticProcesses}.

\begin{Lem}\label{Lem:NuclearityOfT_h}
Let $k\in \N_0$ and suppose $h\in C^{k,k}(\ol\Theta \times \ol\Theta)$ is positive definite. Then the associated kernel operator $T_{h,k}$ is non-negative 
and nuclear.
\end{Lem}
\begin{proof}
Without relying on the classic Mercer expansion, it follows directly from Example \ref{Ex:DPKernels} in Section \ref{Sec:ApplicatonStochasticProcesses} that
there exists a centred $H^k(\Theta)$-valued Gaussian element $X$ having $T_h$ as its covariance operator. This immediately implies the assertion.\\
Alternatively, we can also use the $L^2$-eigensystem $\{ (\lambda_{j,0}, e_{j,0}) \mid j\in\N \}$. We then have $\lambda_{j,0}\geq 0$,
$e_{j,0} \in C^k(\ol\Theta)$ and, for $|\alpha|,|\beta|\leq k$, the expansion 
$\partial^{[\alpha,\beta]}h(x,y) = \sum_{j} \lambda_{j,0} e_{j,0}^{[\alpha]}(x)e_{j,0}^{[\beta]}(y)$, where the right hand side converges absolutely and uniformly
(e.g., \cite{Zhou2007DerivativeReproducingPropertiesForKernelMethods} or \cite{Takhanov2023SpeedOfConvergenceMercer}, Lemma 3). 
It follows that for $f\in H^k(\Theta)$,
\begin{align*}
	\< T_{h}f,f \>_{H^k}
	&= \sum_{|\alpha|,|\beta|\leq k} \int_{\Theta}\int_{\Theta} f^{(\alpha)}(x) \partial^{[\alpha,\beta]}h(x,y) \partial^{(\beta)}f(y) \,dx\,dy\\
	&= \sum_j \lambda_{j,0} \Big| \sum_{|\alpha|\leq k} \int_{\Theta} f^{(\alpha)}(x) e_{j,0}^{[\alpha]}(x) \,dx \Big|^2\\
	&= \sum_j \lambda_{j,0} |\< f,e_{j,0} \>_{H^k}|^2 \geq 0,
\end{align*} 
where interchanging summation and integration is allowed due to the uniform convergence of the series. 
Moreover, define the map $A\colon H^k(\Theta) \to \ell^2$, $f\mapsto (\sqrt{\lambda_{j,0}}\< f, e_{j,0} \>_{H^k})_{j\in\N}$ and observe that for any 
orthonormal basis $\{ u_n\mid n\in\N \}$, by Parseval's identity,
\begin{align}\label{Eq:HSNormOfA}
	\sum_n \| A u_n \|_{\ell^2}^2 
	&= \sum_n \sum_j \lambda_{j,0} |\< u_n, e_{j,0} \>_{H^k}|^2\nonumber\\
	&= \sum_j \lambda_{j,0} \| e_{j,0} \|_{H^k}^2\nonumber\\
	&= \sum_{|\alpha|\leq k} \int_{\Theta} \sum_j \lambda_{j,0} |e_{j,0}^{[\alpha]}(x)|^2 \,dx\nonumber\\
	&= \sum_{|\alpha|\leq k} \int_{\Theta} \partial^{[\alpha,\alpha]}h(x,x) \,dx < \infty.
\end{align}  
This shows $A \in HS(H^k,\ell^2)$ and it follows from the identity $\<Af,g\>_{\ell^2} = \<f,A^*g\>_{H^k}$ that the adjoint operator $A^* \in HS(\ell^2, H^k)$ 
is given by $A^*g = \sum_j \sqrt{\lambda_{j,0}}\, g_j e_{j,0}$ for $g=(g_j)_j \in \ell^2$. Specifically, by Cauchy-Schwarz inequality, we have for $x\in\ol\Theta$,
\begin{align*}
	\sum_{j\geq N} \sqrt{\lambda_{j,0}} \,|g_j\, e_{j,0}^{[\alpha]}(x)| 
	\leq \Big(\sum_{j \geq N} g_j^2\Big)^{1/2} \Big(\sum_{j \geq N} \lambda_{j,0}|e_{j,0}^{[\alpha]}(x)|^2 \Big)^{1/2}
	\leq \Big(\sum_{j \geq N} g_j^2\Big)^{1/2} \| \partial^{[\alpha,\alpha]}h \|_{\infty}^{1/2}.
\end{align*}
This implies $A^*g \in C^k(\ol\Theta) \subset H^k(\Theta)$ with $[A^*g]^{[\alpha]}(x) = \sum_j \sqrt{\lambda_{j,0}}\, g_j e_{j,0}^{[\alpha]}(x)$. 
It is then straight forward to verify $T_h = A^* A$, which yields $T_h$ is nuclear. 
\end{proof}

\begin{Thrm}[Refined Mercer Theorem]\label{Thrm:Mercer}
Let $k\in \N_0$ and suppose $h\in C^{k,k}(\ol\Theta \times \ol\Theta)$ is positive definite. Furthermore, let $\{ (\lambda_{j,k}, e_{j,k})\mid j\in\N \}$ 
denote the eigensystem of the associated kernel operator $T_{h,k} \in HS(H^k)$ defined via \eqref{Eq:AssociatedOperator}. 
Then, $T_{h,k}$ is also compact on $H_0^k(\Theta) := \{ f\in C^k(\Theta)\,|\, \|f\|_{H^k} < \infty \}$, $e_{j,k}\in C^k(\ol\Theta)$ for all 
$j\in \N$ and for $|\alpha|,|\beta| \leq k$,
\begin{align*}
	\lim_{m\to \infty} \sup_{x,y \in \ol\Theta} \Big| \partial^{[\alpha,\beta]}h(x,y)
	- \sum_{j\leq m} \lambda_{j,k} \,e_{j,k}^{[\alpha]}(x)\, e_{j,k}^{[\beta]}(y) \Big| 
	= 0.
\end{align*}
\end{Thrm} 

\begin{proof}
\emph{Step 1 (Compactness of $T_h$ on $H_0^k(\Theta)$)}\\
First note that for $f\in H_0^k(\Theta)$ by assumption 
\begin{align}\label{Eq:UniformBoundMercer}
	\| T_h f \|_{C^k(\ol\Theta)} 
	&= \sum_{|\beta|\leq k} \| (T_h f)^{[\beta]} \|_{\infty} \nonumber\\
	&\leq \sum_{|\alpha|,|\beta|\leq k} \sup_{y\in\Theta} \Big( \int_\Theta \big| \partial^{[\alpha,\beta]}h(x,y) \big|^2 \,dx \Big)^{1/2} 
	\Big( \int_\Theta \big|  f^{[\alpha]}(x)  \big|^2 \,dx \Big)^{1/2} 
	\nonumber\\
	&\leq \sum_{|\alpha|,|\beta|\leq k} \sqrt{\mathrm{vol}(\Theta)} \, \| \partial^{[\alpha, \beta]}h \|_{\infty} \, \| f^{[\alpha]} \|_{L^2} 
	\nonumber\\
	&= \sqrt{\kappa\,\mathrm{vol}(\Theta)} \, \Big( \sum_{|\alpha|,|\beta|\leq k} \| \partial^{[\alpha, \beta]}h \|_{\infty}^2 \Big)^{1/2}\,\|f\|_{H^k},
\end{align}
where $\kappa=\kappa(k,d)=\binom{k+d}{k}$ is the cardinality of the set $\{ \beta \in \N_0^d\mid |\beta| \leq k \}$. 
Moreover, since $\partial^{[\alpha,\beta]}h$ has a unique continuous extension 
on $\overline{\Theta} \times \overline{\Theta}$ for all $|\alpha|,|\beta| \leq k$, for any $\varepsilon > 0$ there exists a $\delta > 0$ such that
\begin{align*}
	\| (x,y) - (x',y') \|_2 \leq \delta \quad\Longrightarrow\quad 
	|\partial^{[\alpha,\beta]} h(x,y) - \partial^{[\alpha,\beta]} h(x',y')| \leq \varepsilon.
\end{align*}
Therefore, if $\|y-y'\|_2 = \| (x,y) - (x,y') \|_2 \leq \delta$, then for $f\in H_0^k(\Theta)$ and $|\beta|\leq k$,
\begin{align}\label{Eq:UniformContinuityT_hf}
	&\big| [T_hf]^{[\beta]}(y) - [T_hf]^{[\beta]}(y') \big|
	\nonumber\\
	&\leq \sum_{|\alpha|\leq k} \Big( \int_\Theta \big| \partial^{[\alpha,\beta]}h(x,y) - \partial^{[\alpha,\beta]}h(x,y')\big|^2 \,dx \Big)^{1/2}
	\, \Big( \int_\Theta \big| f^{[\alpha]}(x) \big|^2 \,dx \Big)^{1/2}
	\nonumber\\
	&\leq \Big( \sum_{|\alpha|\leq k} \int_\Theta \big| \partial^{[\alpha,\beta]}h(x,y) - \partial^{[\alpha,\beta]}h(x,y')\big|^2 \,dx \Big)^{1/2}
	\, \Big( \sum_{|\alpha|\leq k} \|f^{[\alpha]}\|_{L^2}^2 \Big)^{1/2}
	\nonumber\\
	&\leq \varepsilon\, \sqrt{\kappa\,\mathrm{vol}(\Theta)} \, \|f\|_{H^k}.
\end{align}
This shows $[T_hf]^{[\beta]}$ is also uniformly continuous on $\Theta$. Hence,   
the restriction $T_{h}\colon H_0^k(\Theta) \to C^k(\ol\Theta)$ is well-defined and \eqref{Eq:UniformBoundMercer}
implies $T_{h} \in \LL(H_0^k, C^k(\ol\Theta))$ with 
\begin{align}\label{Eq:OperatorNormT_{0,h}}
	\|T_{h}\|_{\LL(H_0^k,  C^k(\ol\Theta))} 
	\leq \sqrt{\kappa\,\mathrm{vol}(\Theta)} \, \Big( \sum_{|\alpha|, |\beta|\leq k} \| \partial^{[\alpha,\beta]}h \|_{\infty}^2 \Big)^{1/2}.
\end{align} 
Compactness of $T_h\colon H_0^k(\Theta) \to C^k(\ol\Theta)$ can be established as follows: Let $(g_n)_n$ be a sequence in 
$K:= \{ T_hf \mid f\in H_0^k(\Theta)~\text{with}~ \|f\|_{H^k} \leq 1 \}$, 
that is, $g_n = T_hf_n$ for some $f_n \in H_0^k(\Theta)$ with $\|f_n\|_{H^k} \leq 1$. Then, as a consequence of \eqref{Eq:OperatorNormT_{0,h}},
\begin{align*}
	\sum_{|\beta|\leq k} \| g_n^{[\beta]} \|_{\infty}
	=\|g_n\|_{C^k(\ol\Theta)}  
	\leq \sqrt{\kappa\,\mathrm{vol}(\Theta)} \, \Big( \sum_{|\alpha|, |\beta| \leq k} \| \partial^{[\alpha, \beta]}h \|_{\infty}^2 \Big)^{1/2},
\end{align*}
which implies $(g_n^{[\beta]})_n$ is uniformly bounded for all $|\beta| \leq k$. Moreover, due to \eqref{Eq:UniformContinuityT_hf},
\begin{align*}
	|g_n^{[\beta]}(y) - g_n^{[\beta]}(y')|
	= | (T_hf_n)^{[\beta]}(y) - (T_hf_n)^{[\beta]}(y') |
	\leq \varepsilon\, \sqrt{\kappa\,\mathrm{vol}(\Theta)}.
\end{align*}
Hence, $(g_n^{[\beta]})_n$ is also uniformly equicontinuous for all $|\beta| \leq k$ and, thus, by the Arzela-Ascoli theorem, 
there exists a subsequence $(g_{0,n})_n$ 
of $(g_n)_n$ that converges uniformly to some $g \in C(\ol\Theta)$. Since $(g_{0,n}^{[\beta]})_n$ is again uniformly bounded and 
equicontinuous for all $|\beta|=1$ another application of the Arzela-Ascoli theorem yields existence of a further subsequence $(g_{1,n})_n$ 
such that $(g_{1,n}^{[\beta]})_n$ converges 
uniformly and the limit agrees with $g^{[\beta]}$ for all $|\beta|=1$, that is $g\in C^1(\ol\Theta)$. 
Repeating this argument $k$-times then leads to a subsequence $(g_{k,n})_n$ and a $g\in C^k(\ol\Theta)$ such that 
$g_{k,n}^{[\beta]} \to g^{[\beta]}$ uniformly for 
all $|\beta|\leq k$ as $n \to \infty$. This in turn establishes relative compactness of $K$ and, therefore, 
$T_h\colon H_0^k(\Theta) \to C^k(\ol\Theta)$ is compact. 
Since the identity $\mathrm{Id}\colon C^k(\ol\Theta) \to H_0^k(\Theta)$ is continuous, the map $T_h\colon H_0^k(\Theta) \to H_0^k(\Theta)$ is 
compact as well. 
By the Meyers-Serrin theorem \cite{MeyerSerrin1964H=W} the completion of $H_0^k(\Theta)$ with respect to 
$\|\cdot\|_{H^k}$ is $H^k(\Theta)$.
Therefore, it follows from the spectral theorem for compact operators on inner product spaces (e.g., \cite{Werner2018Funktionalanalysis}, 
Theorem VI.4.1) that $e_j \in H_0^k(\Theta)$ in the representation \eqref{Eq:SpectralRepresentation}. Moreover, $e_j \in C^k(\ol\Theta)$ as
$e_j^{[\beta]}=[T_he_j]^{[\beta]}$ is uniformly continuous for $|\beta|\leq k$ by \eqref{Eq:UniformContinuityT_hf}.\\

\noindent\emph{Step 2 (Pointwise absolute convergence)}\\
By Step 1 the eigenfunctions $e_j$ are $k$-times continuously differentiable. Hence, the symmetric kernels 
$h_n\colon \ol\Theta\times \ol\Theta \to \R$, $n\in \N$,
defined by 
\begin{align*}
	h_n(x,y) := \sum_{j\leq n} \lambda_j\, e_j(x)\,e_j(y)
\end{align*}
are differentiable as well with 
\begin{align*}
	\partial^{[\alpha,\beta]}h_n(x,y)
	= \sum_{j\leq n} \lambda_j\, e_j^{[\alpha]}(x)\,e_j^{[\beta]}(y)
\end{align*}
for $|\alpha|,|\beta| \leq k$. Let $T_{\partial^{[\alpha,\alpha]}h,0}\colon L^2(\Theta) \to L^2(\Theta)$ denote the $L^2$-integral operator associated to the 
kernel $\partial^{[\alpha,\alpha]}h$ (cf. Remark \ref{Rem:HSNorm}), then by \eqref{Eq:KernelDerivativeExpansionL2}, for $f\in L^2(\Theta)$,
\begin{align}\label{Eq:Non-negativityThL2}
	\< T_{\partial^{[\alpha,\alpha]}h,0}f, f \>_{L^2}
	&= \int_{\Theta} \int_{\Theta} f(x) \,\partial^{[\alpha,\alpha]}h(x,y)\, f(y) \,dx \,dy\nonumber\\
	&= \int_{\Theta} \int_{\Theta} f(x) \, \sum_j \lambda_j \,e_j^{[\alpha]}(x)\,e_j^{[\alpha]}(y) \, f(y) \,dx \,dy\nonumber\\
	&= \sum_j \lambda_j |\< f, e_j^{[\alpha]} \>_{L^2}|^2.
\end{align}
The use of Fubini's theorem for the last equality is a crucial step in the proof and can be justified via Lemma \ref{Lem:NuclearityOfT_h}: 
because $h\in C^{k,k}(\ol\Theta\times\ol\Theta)$, the associated operator $T_h\colon H^k(\Theta) \to H^k(\Theta)$ is non-negative and nuclear. 
Accordingly, the eigenvalues $\lambda_j$ are non-negative and summable, which yields
\begin{align*}
	\sum_j \int_{\Theta} \int_{\Theta} | \lambda_j \,e_j^{[\alpha]}(x)\,e_j^{[\alpha]}(y) \, f(x)f(y)| \,dx\,dy
	&\leq \sum_j \lambda_j\, \|e_j^{[\alpha]}\|_{L^2}^2 \,\|f\|_{L^2}^2 \\ 
	&\leq \|f\|_{L^2} \sum_j \lambda_j < \infty.
\end{align*} 
Similarly, let $T_{\partial^{[\alpha,\alpha]}h_n,0}\colon L^2(\Theta) \to L^2(\Theta)$ denote the $L^2$-integral operator associated to the kernel 
$\partial^{[\alpha,\alpha]}h_n$. Then,
\begin{align}\label{Eq:Non-negativityTh_nL2}
	\< T_{\partial^{[\alpha,\alpha]}h_n,0} f,f \>_{L^2}
	= \sum_{j\leq n} \lambda_j |\< f, e_j^{[\alpha]} \>_{L^2}|^2.
\end{align}
It follows from \eqref{Eq:Non-negativityThL2} and \eqref{Eq:Non-negativityTh_nL2} that the $L^2$-integral operator 
$T_{\partial^{[\alpha,\alpha]}(h-h_n),0}\colon$ $L^2(\Theta) \to L^2(\Theta)$ associated to the kernel 
$\partial^{[\alpha,\alpha]}(h-h_n) \colon \ol\Theta \times \ol\Theta \to \R$ satisfies 
\begin{align*}
	\< T_{\partial^{[\alpha,\alpha]}(h-h_n),0} f,f \>_{L^2}
	&= \< T_{\partial^{[\alpha,\alpha]}h,0}f, f \>_{L^2} - \< T_{\partial^{[\alpha,\alpha]}h_n,0} f,f \>_{L^2}\\
	&= \sum_{j>n} \lambda_j |\< f, e_j^{[\alpha]} \>_{L^2}|^2 \geq 0,
\end{align*}
since all $\lambda_j \geq 0$. 
This shows $\partial^{[\alpha,\alpha]}(h-h_n)\colon \ol\Theta\times\ol\Theta \to \R$ is again a positive-definite kernel for $|\alpha|\leq k$
(e.g., \cite{Hsing2015FoundationsOfFDA}, Theorem 4.6.4), which implies for each $x\in \ol\Theta$,
\begin{align*}
	\partial^{[\alpha,\alpha]} h(x,x) - \partial^{[\alpha,\alpha]} h_n(x,x) 
	= \partial^{[\alpha,\alpha]} (h-h_n)(x,x)
	\geq 0.
\end{align*}
Now, this in turn yields
\begin{align}\label{Eq:MercerProof_PointwiseConvergence1}
	\sum_{j\leq n} \lambda_j |e_j^{[\alpha]}(x)|^2 
	= \partial^{[\alpha, \alpha]} h_n(x,x)
	\leq \partial^{[\alpha,\alpha]} h(x,x)
	\leq \| \partial^{[\alpha,\alpha]} h \|_{\infty}
\end{align}
and, hence, $\sum_{j} \lambda_j |e_j^{[\alpha]}(x)|^2$ converges for every $x\in\ol\Theta$ and $|\alpha|\leq k$. Furthermore, by
\eqref{Eq:MercerProof_PointwiseConvergence1}, for $|\alpha|,|\beta|\leq k$ and $x,y\in\ol\Theta$,
\begin{align*}
	\sum_{j\leq n}  \lambda_j |e_j^{[\alpha]}(x) e_j^{[\beta]}(y)|
	\leq \Big( \sum_{j\leq n} \lambda_j |e_j^{[\alpha]}(x)|^2 \Big)^{1/2} \Big( \sum_{j\leq n} \lambda_j |e_j^{[\beta]}(y)|^2 \Big)^{1/2}
	\leq \| \partial^{[\alpha,\alpha]} h \|_{\infty}^{1/2}\, \| \partial^{[\beta,\beta]} h \|_{\infty}^{1/2},
\end{align*}
which shows the series $\sum_j \lambda_j e_j^{[\alpha]}(x) e_j^{[\beta]}(y)$ converges absolutely.\\

\noindent\emph{Step 3 (Uniform convergence)}\\
Note that due to \eqref{Eq:HilbertSchmidtNorm}, for $n> m$,
\begin{align*}
	\| T_{h_n - h_m} \|_{HS(H^k)}^2
	&= \sum_{|\alpha|,|\beta|\leq k} \big\|  \partial^{[\alpha,\beta]} h_n - \partial^{[\alpha,\beta]} h_m  \big\|_{L^2}^2\\
	&= \sum_{|\alpha|,|\beta|\leq k} \int_\Theta \int_\Theta \Big|\sum_{m<j\leq n} \lambda_j \, e_j^{[\alpha]}(x)\, e_j^{[\beta]}(y) \Big|^2 \,dx\,dy\\
	&= \sum_{m <i,j\leq n} \lambda_i \lambda_j \, \< e_i, e_j \>_{H^k}^2
	= \sum_{m<j\leq n} \lambda_j^2.
\end{align*}
This shows $(T_{h_n})_n$ is a Cauchy sequence in $HS(H^k)$. Since this space is isometrically isomorphic to $H^{k,k}(\Theta \times \Theta)$ by Lemma 
\ref{Lem:CharacterizationHS-Opertors}, the limit operator $\widehat T_h := \lim_{n\to \infty} T_{h_n} \in HS(H^k)$ is associated to the kernel 
\begin{align*}
	\widehat{h}(x,y) := \sum_{j} \lambda_j e_j(x) \, e_j(y),
\end{align*}
where the right hand side converges in $H^{k,k}(\Theta\times\Theta)$. It remains to establish that the convergence is also uniform and
$\partial^{[\alpha, \beta]} \widehat{h}(x,y) = \partial^{[\alpha, \beta]} h(x,y)$. To that end, recall that 
$\sum_{j} \lambda_j |e_j^{[\beta]}(y)|^2$ converges for each 
fixed $y\in\ol\Theta$ and $|\beta|\leq k$ by \eqref{Eq:MercerProof_PointwiseConvergence1}. That is, for each $\epsilon>0$, there exists a 
$N=N(\epsilon,y)\in \N$, such that 
\begin{align}\label{Eq:MercerProof1}
	\sum_{j\geq N} \lambda_j |e_j^{[\beta]}(y)|^2 \leq \epsilon^2.
\end{align} 
Therefore, using the Cauchy-Schwarz inequality and \eqref{Eq:MercerProof_PointwiseConvergence1}, for all $x\in\ol\Theta$,
\begin{align}\label{Eq:MercerProof2}
	\sum_{j\geq N} |\lambda_j e_j^{[\alpha]}(x) e_j^{[\beta]}(y)|
	\leq \Big( \sum_{j\geq N} \lambda_j |e_j^{[\alpha]}(x)|^2 \Big)^{1/2} \Big( \sum_{j \geq N} \lambda_j |e_j^{[\beta]}(y)|^2 \Big)^{1/2}
	\leq \| \partial^{[\alpha,\alpha]}h \|_{\infty}^{1/2} \, \epsilon.
\end{align}
This shows $\sum_{j}|\lambda_j e_j^{[\alpha]}(x) e_j^{[\beta]}(y)|$ is uniformly convergent in $x$ for all $|\alpha|,|\beta| \leq k$ if $y$ is fixed. 
It follows that for any fixed $y\in \ol\Theta$ and $|\beta| \leq k$, the function 
$q_{\beta,y}(x) := \partial^{[0,\beta]}h(x,y) - \partial^{[0,\beta]}\wh{h}(x,y)$, $x\in\ol\Theta$, is $k$-times continuously differentiable and by Lemma 
\ref{Lem:UniformConvergence}, for any $f\in C^k(\ol\Theta)$,
\begin{align*}
	\< q_{\beta,y}, f \>_{H^k}
	&= \int_\Theta \sum_{|\alpha|\leq k} \partial^{[\alpha,\beta]}h(x,y) f^{[\alpha]}(x) \,dx
	- \sum_{j=1}^\infty \lambda_j e_j^{[\beta]}(y) \int_\Theta \sum_{|\alpha|\leq k} e_j^{[\alpha]}(x) f^{[\alpha]}(x)\, dx\\
	&= [T_hf]^{[\beta]}(y) - \sum_{j=1}^\infty \lambda_j \< f,e_j \>_{H^k}\, e_j^{[\beta]}(y) = 0.
\end{align*}
Taking $f = q_{\beta,y}$ then yields  
\begin{align*}
	\| q_{\beta,y} \|_{H^k}^2 
	= \sum_{|\alpha|\leq k} \int_\Theta | \partial^{[\alpha,\beta]} h(x,y) - \partial^{[\alpha,\beta]}\wh{h}(x,y) |^2 \,dx
	= 0,
\end{align*}
which implies $\partial^{[\alpha,\beta]} \widehat{h}(x,y) = \partial^{[\alpha,\beta]} h(x,y)$ for all $x,y \in \ol\Theta$ and 
$|\alpha|,|\beta|\leq k$ by continuity. In particular,
\begin{align*}
	\partial^{[\alpha,\alpha]} h(x,x)
	= \partial^{[\alpha,\alpha]} \wh{h}(x,x)
	= \sum_j \lambda_j |e_j^{[\alpha]}(x)|^2,
\end{align*}
where the right hand side converges for each  $x\in\ol\Theta$. Now, according to Dini's theorem (e.g., \cite{Werner2018Funktionalanalysis}), the convergence 
is also uniform, wherefore $N$ in \eqref{Eq:MercerProof1} can be chosen independently of $y$. By \eqref{Eq:MercerProof2}, this finally implies 
$\sum_j \lambda_j e_j^{[\alpha]}(\cdot)e_j^{[\beta]}(\cdot)$ converges absolutely and uniformly to $\partial^{[\alpha,\beta]}h$.
\end{proof}

The following corollaries are immediate:

\begin{Cor}\label{Cor:NuclearNorm}
Granted the conditions of Theorem \ref{Thrm:Mercer} are satisfied, then $T_{h,k} \in N(H^k)$ is non-negative with
\begin{align}
	\|T_{h,k}\|_{N(H^k)} 
	= \sum_j \lambda_{j,k}
	= \sum_{|\alpha|\leq k} \int_{\Theta} \partial^{[\alpha,\alpha]}h(x,x) \,dx. \label{Eq:NuclearNorm}
\end{align}
Therefore, $\lambda_{j,k} \leq j^{-1}\, \mathrm{vol}(\Theta) \sum_{|\alpha|\leq k} \sup_{x\in \ol\Theta} |\partial^{[\alpha,\alpha]}h(x,x)|$\, for $j\geq 1$.
\end{Cor}

\begin{proof}
It has already been established in Lemma \ref{Lem:NuclearityOfT_h} that $T_h$ is non-negative and nuclear. 
Since the expansion \eqref{Eq:KernelDerivativeExpansionL2} converges uniformly by Theorem \ref{Thrm:Mercer}, it follows that 
\begin{align*}
	\|T_h\|_{N(H^k)}
	&= \sum_{j} \lambda_j \sum_{|\alpha|\leq k} \| e_j^{[\alpha]} \|_{L^2}^2\\
	&= \sum_{|\alpha|\leq k} \int_{\Theta}  \sum_{j} \lambda_j \,e_j^{[\alpha]}(x)\,e_j^{[\alpha]}(x) \,dx\\
	&= \sum_{|\alpha|\leq k} \int_{\Theta} \partial^{[\alpha,\alpha]}h(x,x) \,dx.
\end{align*}
Alternatively, the trace formula also already follows from Lemma \ref{Lem:NuclearityOfT_h} due to the factorization $T=A^*A$ and \eqref{Eq:HSNormOfA}.
The amendment is a consequence of the descending order of eigenvalues, which implies $j\lambda_{j,k} \leq \sum_i \lambda_{i,k}$.  
\end{proof}

\begin{Rem}[Relation to $L^2$-integral operators]\label{Rem:NuclearNorm}
For a positive definite kernel $h\in C^{k,k}(\ol\Theta\times \ol\Theta)$ let $T_{\partial^{[\alpha,\alpha]}h,0}\colon L^2(\Theta) \to L^2(\Theta)$ denote 
the $L^2$-integral operator associated to $\partial^{[\alpha,\alpha]}h \in C(\ol{\Theta}\times \ol{\Theta})$ for $|\alpha|\leq k$, 
cf. Remark \ref{Rem:HSNorm}. 
Note that \eqref{Eq:Non-negativityThL2} implies that $\partial^{[\alpha,\alpha]}h$ is itself positive definite (cf. Theorem 3.6 in
\cite{Buescu2004PositiveDefinitenessMaricesAndDifferentiableKernelInEq}). Consequently, $T_{\partial^{[\alpha,\alpha]}h,0}$ belongs to the trace-class $N(L^2)$ 
due to continuity of $\partial^{[\alpha,\alpha]}h$, and equation \eqref{Eq:NuclearNorm} can be restated as: 
\begin{align}\label{Eq:NuclearNormsRemark}
	\|T_{h,k}\|_{N(H^k)}
	= \sum_{|\alpha|\leq k} \|T_{\partial^{[\alpha,\alpha]}h,0}\|_{N(L^2)}.
\end{align}
Specifically, if $\{\lambda_{j,0}^{(\alpha)}\mid j\in\N \}$ denote the eigenvalues of $T_{\partial^{[\alpha,\alpha]}h,0} \in N(L^2)$, we have:
\begin{align*}
	\sum_{j} \lambda_{j,0}^{(\alpha)}
	=\|T_{\partial^{[\alpha,\alpha]}h,0}\|_{N(L^2)}
	= \int_\Theta \partial^{[\alpha,\alpha]}h(x,x) \,dx
	= \sum_j \lambda_{j,k} \|e_{j,k}^{[\alpha]}\|_{L^2}^2. 
\end{align*} 
This identity allows for the interpretation of $\| e_{j,k}^{[\alpha]} \|_{L^2}^2 \in (0,1)$ as a weight on $\lambda_{j,k}$ such that 
$\|T_{\partial^{[\alpha,\alpha]}h,0}\|_{N(L^2)} = \sum_j \lambda_{j,k} \|e_{j,k}^{[\alpha]}\|_{L^2}^2$ for $|\alpha| \leq k$.
However, simulations (see Example \ref{Ex:EigenExpansionIntBM}) suggest that the term wise relation 
$\lambda_{j,0}^{(\alpha)} = \lambda_{j,k} \|e_{j,k}^{[\alpha]}\|_{L^2}^2$ does not hold in general. Establishing a more precise relationship 
between these distinct eigenvalues - if one exists - will eventually require a more profound spectral analysis of $T_{h,k}$ and 
$T_{\partial^{[\alpha,\alpha]}h,0}$ and is thus left for future research. 
\end{Rem}

\begin{Cor}\label{Cor:MercerRateOfConvergence}
Granted the conditions of Theorem \ref{Thrm:Mercer} are satisfied. Suppose $\Theta$ has the cone property and $k-d/2 > s$ for some $s\in\N_0$. 
Then, 
\begin{align*}
	\max_{|\alpha|,|\beta|\leq s}\, \sup_{x,y \in \ol\Theta} \Big| \partial^{[\alpha,\beta]}h(x,y)
	- \sum_{j\leq m} \lambda_{j,k} \,e_{j,k}^{[\alpha]}(x)\, e_{j,k}^{[\beta]}(y) \Big|
	\leq C^2\, \sum_{j>m} \lambda_{j,k},
\end{align*}
where $C=C(\Theta,k,d)>0$ is the constant in \eqref{Eq:UniformBoundEigenfunctions}.
\end{Cor}

\begin{proof}
The uniform convergence in Theorem \ref{Thrm:Mercer} allows to write  
\begin{align*}
	\partial^{[\alpha,\beta]}h(x,y)
	- \sum_{j\leq m} \lambda_{j,k} \,e_{j,k}^{[\alpha]}(x)\, e_{j,k}^{[\beta]}(y)
	= \sum_{j>m} \lambda_{j,k} \,e_{j,k}^{[\alpha]}(x)\, e_{j,k}^{[\beta]}(y)
	= \Delta_{m,k}^{[\alpha,\beta]}(x,y).
\end{align*}
Moreover, due to the assumptions, the universal uniform boundedness \eqref{Eq:UniformBoundEigenfunctions} of the eigenfunctions $e_{j,k}$ applies for some 
constant $C>0$. Thus, for $|\alpha|,|\beta| \leq s$,  
\begin{align*}
	\big| \Delta_{m,k}^{[\alpha,\beta]}(x,y) \big|^2
	\leq  \sum_{j>m} \lambda_{j,k} |e_{j,k}^{[\alpha]}(x)|^2 
	 \sum_{j>m} \lambda_{j,k} |e_{j,k}^{[\beta]}(y)|^2 
	&\leq \big(C^2 \sum_{j>m} \lambda_{j,k}\big)^2. \qedhere
\end{align*}
\end{proof}

\begin{Rem}[Rate of convergence]
Similar to \eqref{Eq:ApproxErrorL2} the error of approximation can be estimated only knowing the first $m$ eigenvalues since 
\eqref{Eq:NuclearNorm} yields 
\begin{align}\label{Eq:ApproxErrorMercer}
	\sum_{j>m} \lambda_{j,k} 
	= \sum_{|\alpha|\leq k} \int_\Theta \partial^{[\alpha,\alpha]}h(x,x) \,dx - \sum_{j\leq m} \lambda_{j,k}.
\end{align}
\end{Rem}

\subsection{Reproducing Kernel Hilbert Spaces}
A primary consequence of Theorem \ref{Thrm:Mercer} is that it yields novel spectral representations of the
Reproducing kernel Hilbert space (RKHS) associated with a positive definite kernel $h\colon \Theta\times \Theta \to \R$. 
Recall that the RKHS, denoted by $H_h$, is the unique Hilbert space of functions $f\colon \Theta\to\R$ in which point evaluations are continuous 
linear functionals. Equivalently, $H_h$ is characterized by the reproducing property 
\begin{align}\label{Eq:RKHS_RepProperty}
	\< f, h(x,\cdot) \>_{H_h} = f(x)
\end{align}
for all $f\in H_h$ and $x\in\Theta$. Formally, this space can be defined as the closure of the span of kernel functions $\{ h(x,\cdot)\mid x\in \Theta\}$
with respect to the inner product $\< f,g \>_{H_h} = \sum_{i\leq n}\sum_{j\leq m} a_ih(x_i,x_j)b_j$, where $f = \sum_{i\leq n} a_i h(x_i,\cdot)$ and 
$g=\sum_{j\leq m} b_j h(x_j,\cdot)$ (e.g., \cite{Aronszajn1950TheoryOfReproducingKernels}). 
For a continuous kernel $h$, Mercer's theorem provides a well-known 
spectral representation of $H_h$ in terms of the eigensystem $\{ (\lambda_{j,0}, e_{j,0})\mid j\in \N \}$ of the associated $L^2$-integral
operator $T_{h,0}$ (e.g., \cite{Cucker2001OnTheMathematicalFoundationsOfLearning}). 
Accordingly, additional representations for a differentiable kernel can now be derived from Theorem \ref{Thrm:Mercer}:

To this end, suppose $h\in C^{k,k}(\ol\Theta \times \ol\Theta)$ for $k\in\N_0$ is positive definite, and let $\{ (\lambda_{j,k}, e_{j,k})\mid j\in \N \}$ 
denote the eigensystem of the associated kernel operator $T_{h,k} \in HS(H^k)$. Furthermore, define the space
\begin{align}\label{Eq:RKHS_SpectralRep}
	\HH_{h,k} = \big\{ f\in H^k(\Theta) \mid f = \sum_j a_j\, e_{j,k} ~\text{s.t.}~(\lambda_{j,k}^{-1/2}\, a_j)_j \in \ell^2 \big\},
\end{align} 
where $\ell^2 = \{ (a_j)_j\mid \sum_j |a_j|^2 <\infty\}$ is the space of square summable sequences and $a_j = \< f,e_{j,k} \>_{H^k}$.
We equip $\HH_{h,k}$ with the inner product 
\begin{align}\label{Eq:RKHS_SpectralInnerProduct}
	\< f,g \>_{\HH_{h,k}} = \sum_j \lambda_{j,k}^{-1}\, a_j\, b_j
\end{align}
for $f = \sum_j a_j\,e_{j,k}$ and $g=\sum_j b_j\, e_{j,k}$. To prove that $H_h = \HH_{h,k}$ as function spaces under these 
conditions, we first establish a series of auxiliary results that are also of independent interest.  

\begin{Lem}\label{Lem:RKHS_FeatureMap} 
For $|\alpha| \leq k$ the feature map, 
\begin{align*}
	E_k^\alpha\colon \ol\Theta \to \ell^2, \quad x\mapsto \big(\sqrt{\lambda_{j,k}}\,e_{j,k}^{[\alpha]}(x)\big)_j,
\end{align*}
is well-defined, continuous, and satisfies 
\begin{align*}
	\< E_k^\alpha(x), E_k^\beta(y) \>_{\ell^2} = \partial^{[\alpha,\beta]}h(x,y).
\end{align*}
\end{Lem}

\begin{proof}
The series $\sum_j \lambda_{j,k}\, |e_{j,k}^{[\alpha]}(x)|^2$ converges uniformly to $\partial^{[\alpha,\alpha]}h(x,x)$ by Theorem \ref{Thrm:Mercer}, thus 
$E_k^{\alpha}(x) \in \ell^2$. Moreover, by the same theorem, 
\begin{align*}
	\< E_k^\alpha(x), E_k^\beta(y) \>_{\ell^2} 
	= \sum_j \lambda_{j,k}\, e_{j,k}^{[\alpha]}(x)\,e_{j,k}^{[\beta]}(y) = \partial^{[\alpha,\beta]}h(x,y)
\end{align*}
and 
\begin{align*}
	\| E_k^\alpha(x) - E_k^\alpha(y) \|_{\ell^2}
	= \partial^{[\alpha,\alpha]}h(x,x) + \partial^{[\alpha,\alpha]}h(y,y) - 2\partial^{[\alpha,\alpha]}h(x,y),
\end{align*}
which implies continuity of $E_k^\alpha$ since $\partial^{[\alpha,\alpha]}h \in C(\ol\Theta \times\ol\Theta)$ by assumption. 
\end{proof}

Regarding the regularity of functions in $\HH_{h,k}$ the following applies: 

\begin{Prop}\label{Prop:RKHSisSubsetOfC^k}
The functions in $\HH_{h,k}$ are $k$-times continuously differentiable, i.e. $\HH_{h,k} \subset C^k(\ol\Theta)$. 
Moreover, if $f=\sum_j a_j \, e_{j,k} \in \HH_{h,k}$, then for $|\alpha| \leq k$ and $x\in\ol\Theta$,
\begin{align}\label{Eq:DerivativeInRKHS}
	f^{[\alpha]}(x) = \sum_j a_j \, e_{j,k}^{[\alpha]}(x)
\end{align} 
and the series converges uniformly and absolutely. 
\end{Prop}

\begin{proof}
Let $f=\sum_j a_j \, e_{j,k} \in \HH_{h,k}$ and $x\in \ol\Theta$. Then, according to Theorem \ref{Thrm:Mercer} and \eqref{Eq:RKHS_SpectralInnerProduct},
\begin{align*}
	|f(x)|^2 
	= \big| \sum_j \frac{a_j}{\sqrt{\lambda_{j,k}}} \sqrt{\lambda_{j,k}}\, e_{j,k}(x) \big|^2
	\leq \sum_j \frac{a_j^2}{\lambda_{j,k}} \, \sum_j \lambda_{j,k} |e_{j,k}(x)|^2
	= \| f \|_{\HH_{h,k}}^2 \, h(x,x) \leq \| f \|_{\HH_{h,k}}^2 \|h\|_\infty.
\end{align*}
Therefore, convergence in $\HH_{h,k}$ implies uniform convergence and applied to 
$f_n = f - \sum_{j\leq n} a_j e_{j,k}$ this yields $f(x) = \sum_j a_j \, e_{j,k}(x)$ uniformly in $x\in \ol\Theta$. Similarly, for $|\alpha| \leq k$, 
\begin{align*}
	\sum_{j\geq n} | a_j \, e_{j,k}^{[\alpha]}(x) |
	\leq \Big( \sum_{j\geq n} \frac{a_j^2}{\lambda_{j,k}} \Big)^{1/2} \| \partial^{[\alpha,\alpha]}h \|_{\infty}^{1/2},
\end{align*} 
which shows absolute and uniform convergence of $\sum_{j} a_j \, e_{j,k}^{[\alpha]}$. Since $e_{j,k} \in C^k(\ol\Theta)$, this in turn establishes 
$f\in C^k(\ol\Theta)$ and \eqref{Eq:DerivativeInRKHS}.
\end{proof}

The reproducing property \eqref{Eq:RKHS_RepProperty} is verified in the next lemma. 

\begin{Lem}\label{Lem:RKHS_RepProperty}
For all  $f\in \HH_{h,k}$, $|\alpha| \leq k$ and $x\in\ol\Theta$,
\begin{enumerate}
	\item[$\mathrm{(a)}$] $\displaystyle \partial^{[\alpha,0]}h(x,\cdot) \in \HH_{h,k}$
	\item[$\mathrm{(b)}$] $\displaystyle \< f, \partial^{[\alpha,0]}h(x,\cdot) \>_{\HH_{h,k}} = f^{[\alpha]}(x)$
\end{enumerate}
\end{Lem}

\begin{proof}
(a) Let $b_j = \lambda_{j,k}\,e_{j,k}^{[\alpha]}(x)$, then $\sum_j b_j\, e_{j,k} = \partial^{[\alpha,0]}h(x,\cdot)$ 
and $\sum_j b_j^2 \,\lambda_{j,k}^{-1} = \partial^{[\alpha,\alpha]}h(x,x)$ according to Theorem \ref{Thrm:Mercer}.\\
(b) Let $f=\sum_j a_j e_{j,k}$, then by (a) and Proposition \ref{Prop:RKHSisSubsetOfC^k}, 
\begin{align*}
	\< f, \partial^{[\alpha,0]}h(x,\cdot) \>_{\HH_{h,k}}
	= \sum_j \lambda_{j,k}^{-1}\, a_j\, b_j
	= \sum_j a_j \,e_{j,k}^{[\alpha]}(x)
	&= f^{[\alpha]}(x).\qedhere
\end{align*}
\end{proof}

\begin{Thrm}\label{Thrm:RKHS_SpectralRep}
As Hilbert spaces of functions, $(\HH_{h,k}, \<\cdot,\cdot\>_{\HH_{h,k}}) = (H_h, \<\cdot,\cdot\>_{H_h})$.
\end{Thrm}

\begin{proof}	
Let $f\in \HH_{h,k}$ and suppose $\< f, h(x,\cdot) \>_{\HH_{h,k}}= 0$ for all $x\in\ol\Theta$. Then $f=0$ by Lemma \ref{Lem:RKHS_RepProperty}, which implies 
$\{ h(x,\cdot)\mid x\in\ol\Theta \}$ is dense in $\HH_{h,k}$. Combined with property (a) in Lemma \ref{Lem:RKHS_RepProperty} this already proves 
the assertion due to the unique characterization of the RKHS $H_h$ (e.g., \cite{Cucker2001OnTheMathematicalFoundationsOfLearning}, Theorem III.2).
\end{proof}

Specifically, Theorem \ref{Thrm:RKHS_SpectralRep} provides a link between the eigensystems corresponding to kernel operators of varying orders, all associated 
with a fixed positive definite kernel $h \in C^{k,k}(\ol\Theta\times\ol\Theta)$ .

\begin{Cor}
For $\ell\leq k$, let $\{ (\lambda_{j,\ell}, e_{j,\ell})\mid j\in\N \}$ denote the eigensystem of the associated operator $T_{h,\ell}\in HS(H^\ell)$. 
Then, for $n\in\N$,
\begin{align*}
	\frac{1}{\lambda_{n,k}}
	= \sum_j \frac{|\< e_{n,k}, e_{j,\ell} \>_{H^\ell}|^2}{\lambda_{j,\ell}}
	\quad\text{and}\quad
	\frac{1}{\lambda_{n,\ell}}
	= \sum_j \frac{|\< e_{n,\ell}, e_{j,k} \>_{H^k}|^2}{\lambda_{j,k}}.
\end{align*}
\end{Cor}
\begin{proof}
According to Theorem \ref{Thrm:RKHS_SpectralRep}, we have $\|f\|_{\HH_{h,k}} = \| f \|_{H_h} = \|f\|_{\HH_{h,\ell}}$ for any $f\in H_h$. 
By \eqref{Eq:RKHS_SpectralInnerProduct}, this implies 
\begin{align}\label{Eq:CorRKHSNorm}
	\sum_j \frac{|\<f,e_{j,k}\>_{H^k}|^2}{\lambda_{j,k}}
	= \sum_j \frac{|\<f,e_{j,\ell}\>_{H^\ell}|^2}{\lambda_{j,\ell}}.
\end{align}
Moreover, $e_{n,k}, e_{n,\ell}\in H_h \subset C^k(\ol\Theta)$ by Proposition \ref{Prop:RKHSisSubsetOfC^k}. Substituting $e_{n,k}$, respectively, $e_{n,\ell}$
into \eqref{Eq:CorRKHSNorm} yields the assertion.
\end{proof}

\section{Applications to Stochastic Processes}\label{Sec:ApplicatonStochasticProcesses}
Before applying the results of the previous section to stochastic processes with paths in a Sobolev space, a brief review of some fundamentals 
of probability theory on Banach spaces is in order. For a comprehensive exposition of this topic, the reader may wish to consult 
\cite{Bogachev1998GaussianMeasures}, 
\cite{Hsing2015FoundationsOfFDA},
\cite{Hytonen2016AnalysisInBanachSpaces},
\cite{LedouxTalagrand1991ProbabilityBanachSpaces},  
and \cite{Vakhania1987ProbabilityDistributionsOnBanachSpaces}.

\subsection{Random Elements in Sobolev Spaces}\label{SubSec:Random Elements in Sobolev Spaces}
Throughout this section, $(\Omega,\A,\P)$ denotes a complete probability space, meaning that all $\P$-negligible sets are contained in $\A$. 
Given a real Banach space $E$, a map $X\colon \Omega \to E$ is called an $E$-valued random element, if it is $\A$-$\B(E)$-measurable, where $\B(E)$ 
denotes the Borel-$\sigma$-algebra of $E$. If $E$ is separable, this is equivalent to the $\A$-measurability of $\varphi(X)\colon \Omega \to \R$ 
for all $\varphi \in E_0'$, where $E_0'$ is a separating subspace of $E'$.  
Moreover, in this case, the distribution or law of an $E$-valued random element $X$ - defined as the push-forward measure $\P_X := \P \circ X^{-1}$ 
on $(E, \B(E)$ - is uniquely determined by the joint distribution of the marginals $\varphi(X)$ for $\varphi \in E_0'$.
A measure $\gamma$ on $(E, \B(E))$ is called Gaussian, if $\gamma\circ \varphi^{-1}$ is a Gaussian measure on $(\R,\B(\R))$ for all $\varphi \in E'$.
Similarly, an $E$-valued random element $X$ is called Gaussian, if its law $\P_X$ is a Gaussian measure; that is, if $\varphi(X)$ is a normally distributed 
for all $\varphi \in E'$.

\begin{Lem}\label{Lem:SobolevRandomElement}
Let $X\colon \Omega \to H^k(\Theta)$ be a map. The following statements are equivalent:
\begin{enumerate}
	\item[$\mathrm{(i)}$] $X$ is a $H^k(\Theta)$-valued random element; that is $\A$-$\B(H^k)$-measurable.
	\item[$\mathrm{(ii)}$] $X^{(\alpha)}$ is a $L^2(\Theta)$-valued random element for $|\alpha| \leq k$; that is $\A$-$\B(L^2)$-measurable.
\end{enumerate}
\end{Lem}

\begin{proof}
Since $X(\omega) \in H^k(\Theta)$ for all $\omega \in \Omega$ by assumption, the map $X^{(\alpha)}\colon \Omega \to L^2(\Theta)$ is well-defined 
for $|\alpha| \leq k$.\\
(i) $\Longrightarrow$ (ii):
It suffices to verify that for $\varphi \in C_c^{\infty}(\Theta)$ (which is separating for $L^2(\Theta)$) the map 
\begin{align*}
	\omega \mapsto \< X^{(\alpha)}(\omega), \varphi \>_{L^2} = (-1)^{|\alpha|} \< X(\omega), \varphi^{[\alpha]} \>_{L^2}
\end{align*}
is measurable. But this follows immediately from the $\A$-$\B(H^k)$-measurability of $X$ and the continuity of the functional 
$\<\,\cdot\,, \varphi^{[\alpha]}\>_{L^2}\colon H^k(\Theta) \to \R$.\\
(ii) $\Longrightarrow$ (i): Since each $X^{(\alpha)}$ is $\A$-$\B(L^2)$-measurable, it follows that
\begin{align*}
	\omega \mapsto \< X(\omega), f \>_{H^k} = \sum_{|\alpha| \leq k} \< X^{(\alpha)}(\omega), f^{(\alpha)} \>_{L^2}
\end{align*}
is measurable for any $f\in H^k(\Theta)$. Hence, $X$ is $\A$-$\B(H^k)$-measurable.
\end{proof}

\paragraph{Expectation}
The separability of $H^k(\Theta)$ implies that $\A$-$\B(H^k)$-measurability of a map $X:\Omega \to H^k(\Theta)$ is equivalent to the notion of strong 
measurability (or Bochner measurability). Consequently, for any $H^k(\Theta)$-valued random element $X$, the Bochner-integral 
$\E[X]:= \int_{\Omega} X \,d\P \in H^k(\Theta)$ exists if and only if 
$\E\|X\|_{H^k}<\infty$, for which $\E\|X^{(\alpha)}\|_{L^2} < \infty$ for all $|\alpha| \leq k$ is sufficient. In this case, 
the expectation $\E[X] \in H^k(\Theta)$ is uniquely determined by the relation $\E \< X, f \>_{H^k} = \< \E[X], f \>_{H^k}$ for all $f\in H^k(\Theta)$, 
that is
\begin{align}\label{Eq:ExpectationRelation}
	\sum_{|\alpha| \leq k} \E \big[\int_\Theta X^{(\alpha)}(t) f^{(\alpha)}(t) \,dt \big]
	= \sum_{|\alpha| \leq k} \int_\Theta \E[X]^{(\alpha)}(t) f^{(\alpha)}(t) \,dt.
\end{align}
Moreover, $\partial^{(\alpha)}\in \LL(H^k, L^2)$, so we also have $\E[X]^{(\alpha)} = \E[X^{(\alpha)}]$ as elements in $L^2(\Theta)$, that is 
\begin{align*}
	\int_{\Theta} ( \E[X]^{(\alpha)}(t) - \E[X^{(\alpha)}](t) )^2 \,dt = 0. 
\end{align*}
A $H^k(\Theta)$-valued random element $X$ with $\E[X] \equiv 0$ is called centred.

\paragraph{Covariance operator}  
Let $X$ and $Y$ be two $H^k(\Theta)$-valued random elements such that $\E \|X\|_{H^k}^2 < \infty$ and $\E \|Y\|_{H^k}^2 < \infty$. 
Then their cross-covariance operator
\begin{align}\label{Eq:CovarianceOperatorDefinition}
	\Cov_k(X,Y) := \int_{\Omega} (X-\E[X]) \otimes (Y-\E[Y]) \,d\P \in N(H^k)
\end{align} 
is well-defined - here, as a Bochner integral in the space $N(H^k)$ (e.g., \cite{Rademacher2024AsymNormalitySpectralMeans} 
or \cite{Vakhania1987ProbabilityDistributionsOnBanachSpaces}). Specifically, $\Cov_k(X) := \Cov_k(X,X)$ is called the covariance operator of $X$.  
Recall that $N(H^k)$ denotes the space of nuclear operators on $H^k(\Theta)$ and $\otimes\colon H^k(\Theta) \times H^k(\Theta) \to N(H^k)$ is defined by 
$f\otimes g(u) = \<u,g\>_{H^k}f$ for $f,g,u \in H^k(\Theta)$.
The following Lemma establishes a crucial link to the kernel operators of higher order discussed in the previous Section \ref{Sec:KernelExpansions}.

\begin{Lem}\label{Lem:CovarianceOperator}
Let $X$ be a $H^k(\Theta)$-valued random element such that $\E\|X\|_{H^k}^2<\infty$.
\begin{enumerate}
\item[(a)] $\Cov_k(X) \in N(H^k)$ is self-adjoint, non-negative and satisfies $\|\Cov_k(X)\|_{N(H^k)} = \E\|X\|_{H^k}^2$.
\item[(b)] There exists a unique kernel $c_X \in H^{k,k}(\Theta\times\Theta)$, such that $\Cov_k(X) = T_{c_X,k}$; that is, for $f\in H^k(\Theta)$,
\begin{align}\label{Eq:CovarianceOperator}
	[\Cov_k(X)f](t) 
	=[T_{c_X,k}f](t)
	= \int_{\Theta} \sum_{|\alpha| \leq k} \partial^{(\alpha,0)}c_X(s,t) f^{(\alpha)}(s) \,ds \quad\text{a.e.}
\end{align}
Moreover, $\Cov_0(X^{(\alpha)}, X^{(\beta)}) = T_{\partial^{(\beta,\alpha)}c_X,0} \in N(L^2)$, for $|\alpha|, |\beta| \leq k$; that is, for $f\in L^2(\Theta)$,
\begin{align}\label{Eq:CovarianceOperatorDerivatives}
	[\Cov_0(X^{(\alpha)}, X^{(\beta)})f](t) 
	=[T_{\partial^{(\beta,\alpha)}c_X,0}f](t)
	= \int_{\Theta} \partial^{(\beta,\alpha)}c_X(s,t) f(s) \,ds \quad\text{a.e.}
\end{align}
and 
\begin{align}\label{Eq:CovarianceNuclearNorm}
	\| \Cov_k(X) \|_{N(H^k)}
	=\sum_{|\alpha|\leq k} \| \Cov_0(X^{(\alpha)}) \|_{N(L^2)}
\end{align}
\end{enumerate}
\end{Lem}

\begin{proof}
(a) These properties of a covariance operator are well-known (e.g., \cite{Vakhania1987ProbabilityDistributionsOnBanachSpaces}, \cite{Bosq2000LinearProcesses} 
or \cite{Rademacher2024AsymNormalitySpectralMeans}, Theorem A.1).\\
(b) Because $N(H^k) \subset HS(H^k)$, Lemma \ref{Lem:CharacterizationHS-Opertors} ensures the existence of a unique kernel $c_X \in H^{k,k}(\Theta\times\Theta)$
such that $\Cov_k(X) = T_{c_X,k}$. To establish \eqref{Eq:CovarianceOperatorDerivatives}, assume without loss that $X$ is centred. Moreover, note that
$\partial^{(\alpha)} \in \LL(H^k,L^2)$, so there exists an adjoint operator $\partial^{(\alpha)*} \in \LL(L^2,H^k)$ such that 
$\< \partial^{(\alpha)}f,g \>_{L^2} = \< f, \partial^{(\alpha)*}g \>_{H^k}$ for $f\in H^k(\Theta)$ and $g\in L^2(\Theta)$.
This implies $\partial^{(\alpha)*}g \in H^k(\Theta)$ is the unique element satisfying for all $f\in H^k(\Theta)$,
\begin{align}\label{Eq:AdjointWeakDerivative}
	\int_\Theta f^{(\alpha)}(t)\,g(t)\,dt
	= \int_\Theta \sum_{|\beta|\leq k} f^{(\beta)}(t)\, [\partial^{(\alpha)*}g]^{(\beta)}(t) \,dt. 
\end{align}
Let $\partial^{(\alpha)}\otimes\partial^{(\beta)} \in \LL(HS(H^k), HS(L^2))$ denote the operator tensor product, defined via
$\partial^{(\alpha)}\otimes\partial^{(\beta)}(A) = \partial^{(\alpha)} \circ A \circ \partial^{(\beta)*}$ for $A\in HS(H^k)$ 
(e.g., \cite{Kadison1991FundamentalsTheoryOfOperatorAlgebras}). Here, we used that $HS(H^k,L^2)\cong H^{k,0}(\Theta \times \Theta) \supset H^{k,k}(\Theta \times \Theta)\cong HS(H^k)$ and 
$HS(L^2,H^k) \cong H^{0,k}(\Theta\times\Theta) \subset L^2(\Theta\times\Theta) \cong HS(L^2)$ by Lemma \ref{Lem:CharacterizationHS-Opertors}.
It follows by commutativity of the Bochner-integral with linear operators, that 
\begin{align*}
	\Cov_0(X^{(\alpha)}, X^{(\beta)})
	&= \int_{\Omega} X^{(\alpha)} \otimes X^{(\beta)} \,d\P\\
	&= \int_{\Omega} \partial^{(\alpha)}\otimes\partial^{(\beta)} (X \otimes X) \, d\P\\
	&= \partial^{(\alpha)}\otimes\partial^{(\beta)}(\Cov_k(X)).
\end{align*}
Applying this operator to an element $f\in L^2(\Theta)$ and using \eqref{Eq:AdjointWeakDerivative} therefore yields 
\begin{align*}
	[\Cov_0(X^{(\alpha)}, X^{(\beta)})f](t)
	&= \partial^{(\alpha)}\, \Cov_k(X) \,\partial^{(\beta)*}f](t)\\
	&= \partial^{(\alpha)} \int_{\Theta} \sum_{|\gamma|\leq k} \partial^{(\gamma,0)}c_X(s,t)\,[\partial^{(\beta)*}f(s)]^{(\gamma)} \,ds\\
	&= \partial^{(\alpha)} \int_{\Theta} \partial^{(\beta,0)}c_X(s,t)\,f(s)\,ds\\
	&= \int_{\Theta} \partial^{(\beta,\alpha)}c_X(s,t)\,f(s)\,ds. \qedhere
\end{align*}
Finally, note that (e.g., \cite{Rademacher2024AsymNormalitySpectralMeans}, Lemma A.2)
\begin{align*}
	\sum_{|\alpha|\leq k} \| \E[X^{(\alpha)}] - f^{(\alpha)} \|_{L^2}^2 + \| \Cov_0(X^{(\alpha)}) \|_{N(L^2)}
	&= \sum_{|\alpha|\leq k} \E \|X^{(\alpha)} - f^{(\alpha)}  \|_{L^2}^2\\
	&= \E \| X-f\|_{H^k}^2\\
	&= \| \E[X] - f \|_{H^k}^2 + \|\Cov_k(X)\|_{N(H^k)}.
\end{align*}
Minimizing over all $f\in H^k(\Theta)$, i.e. substituting $f=\E[X]$, then yields \eqref{Eq:CovarianceNuclearNorm}.
\end{proof}

\begin{Rem}[$H_{\kappa}(\Theta)$-valued random elements]\label{Rem:ExAndCovInH_kappa}
Recall from Remark \ref{Rem:H_Kappa} that there exists an isometric isomorphism $\J\colon H^k(\Theta) \to H_{\kappa}(\Theta)$. 
Accordingly, a $H^k(\Theta)$-valued random element $X$ can also be viewed as a $H_{\kappa}(\Theta)$-valued random element 
$\J X = (X^{(\alpha)})_{|\alpha|\leq k}$ with distribution $\P_{\J(X)} = \P_X \circ \J^{\ast}$. 
If $\E\|X\|_{H^k} < \infty$, the commutativity of the Bochner integral 
with bounded linear operators implies that the expectation of $\J X$ satisfies
\begin{align}\label{Eq:ExpectationPX}
	\E[\J X] = \J(\E [X]) = (\E[X]^{(\alpha)})_{|\alpha|\leq k} \in H_\kappa(\Theta).
\end{align}
Similarly, let $g=\J f\in H_\kappa(\Theta)$. Then, granted $\E\|X\|_{H^k}^2 < \infty$, the covariance operator of $\J X$ satisfies, 
\begin{align}\label{Eq:CovariancePX2}
	\Cov(\J X)g 
	&= \J \,\Cov_k(X) \J^*\J f
	= \J (\Cov_k(X)f)\nonumber\\
	&= \Big( \sum_{|\alpha|\leq k} \int_{\Theta} \partial^{(\alpha,\beta)}c_X(s,\cdot)\,f^{(\alpha)}(s) \,ds \Big)_{|\beta|\leq k}\nonumber\\
	&= \Big( \sum_{|\alpha|\leq k} \Cov_0(X^{(\beta)}, X^{(\alpha)}) f^{(\alpha)} \Big)_{|\beta|\leq k}
\end{align}
with respect to $\|\cdot\|_{L^2_\kappa}$. 
Consequently, 
\begin{align}\label{Eq:CovariancePX}
	\Cov(\J X) = (\Cov_0(X^{(\beta)},X^{(\alpha)}))_{|\alpha|,|\beta| \leq k} \in N(H_\kappa),
\end{align}
where the operator matrix on the right-hand side acts according to \eqref{Eq:CovariancePX2}.
\end{Rem}

It is well-known that the law of a Gaussian random element is fully characterized by its expectation and covariance operator (e.g.,
\cite{Vakhania1987ProbabilityDistributionsOnBanachSpaces}, Proposition 2.8). In the context of Sobolev spaces, this characterization takes the 
following form:

\begin{Lem}[Gaussian elements]\label{Lem:GaussianElements}
Let $X\colon \Omega \to H^k(\Theta)$ be a map. The following statements are equivalent: 
\begin{enumerate}
	\item[$\mathrm{(i)}$] $X$ is a Gaussian $H^k(\Theta)$-valued random element with expectation $\E[X] \in H^k(\Theta)$ 
	and covariance operator $\Cov_k(X) \in N(H^k)$ given by \eqref{Eq:CovarianceOperator}.
	
	\item[$\mathrm{(ii)}$] $\<X,f\>_{H^k} \sim \NN(\< \E[X], f \>_{H^k}, \< \Cov_k(X)f,f \>_{H^k})$ for all $f\in H^k(\Theta)$.
	
	\item[$\mathrm{(iii)}$]  $PX=(X^{(\alpha)})_{|\alpha|\leq k}$ is a Gaussian $H_{\kappa}(\Theta)$-valued random element with expectation 
	$\E[PX] = (\E[X]^{(\alpha)})_{|\alpha|\leq k}\in H_\kappa(\Theta)$ and covariance operator 
	$\Cov(PX) = (\Cov_0(X^{(\beta)},X^{(\alpha)}))_{|\alpha|,|\beta|\leq k} \in N(H_{\kappa})$ 
	given by \eqref{Eq:CovariancePX}.
\end{enumerate}
In the affirmative case, this is denoted by $X \sim \NN_{H^k}(\E[X], \Cov_k(X))$.
\end{Lem}

The assumption $\E\|X\|_{H^k}^2<\infty$ is quite useful as it implies existence of a well-behaved representative for $X$.

\begin{Lem}\label{Lem:MeasurableRepresentative}
Let $X$ be a $H^k(\Theta)$-valued random element such that $\E\|X\|_{H^k}^2 < \infty$.
Then there exists a $\A\otimes \B(\Theta)$-measurable function 
$Y\colon \Omega \times \Theta \to \R$, unique only up to a $\P\otimes \lambda^d$-null set, such that
\begin{enumerate}
	\item[(a)] $\Theta \ni t \mapsto Y(\omega,t)$ is a representative of $X(\omega) \in H^k(\Theta)$ for almost every $\omega \in\Omega$.
	\item[(b)] $\Theta\ni t\mapsto \E[Y(t)] = \mu_Y(t)$ is a representative of $\E[X] \in H^k(\Theta)$ and $\mu_Y^{(\alpha)}(t) = \E[Y^{(\alpha)}(t)]$ a.e.
	\item[(c)] $\Theta\times\Theta\ni (s,t) \mapsto \Cov(Y(s),Y(t)) = \gamma_{Y}(s,t)$ is a representative of $c_X \in H^{k,k}(\Theta \times \Theta)$ and 
		$\partial^{(\alpha,\beta)}\gamma_{Y}(s,t) = \Cov[ Y^{(\alpha)}(s), Y^{(\beta)}(t) ]$ a.e.
\end{enumerate}
\end{Lem}

\begin{proof}
(a) It follows from Lemma \ref{Lem:SobolevRandomElement} that $X^{(\alpha)}\colon \Omega \to L^2(\Theta)$ is $\A$-$\B(L^2)$-measurable for all $|\alpha|\leq k$.
Moreover, by assumption, $\E \|X^{(\alpha)}\|_{L^2}^2 < \infty$. By Proposition 1.2.24 in \cite{Hytonen2016AnalysisInBanachSpaces}, this implies 
existence of a
$Y^{(\alpha)} \in L^2(\Omega \times \Theta) = \{ f\colon \Omega \times \Theta \to \R\mid f~ \text{is}~ \A\otimes\B(\Theta)\text{-measurable and} 
\int_{\Omega}\int_{\Theta} |f(\omega,t)|^2 \,dt\,d\P(\omega) <\infty \}$ such that $Y^{(\alpha)}(\omega,\cdot)$ is a representative for 
$X^{(\alpha)}(\omega) \in L^2(\Theta)$ for all $\omega \in \Omega\backslash N_{\alpha}$, where $\P(N_\alpha) = 0$.
Let $N= \bigcup_{|\alpha|\leq k} N_\alpha$, then  
$\< Y^{(\alpha)}(\omega,\cdot), \varphi \>_{L^2} = (-1)^{|\alpha|} \< X(\omega),\varphi^{[\alpha]} \>_{L^2}$ for 
$\omega \in \Omega\backslash N$ and $|\alpha|\leq k$.
This shows, $(Y^{(\alpha)}(\omega,\cdot))_{|\alpha|\leq k}$ is a representative for $\J X(\omega) = (X^{(\alpha)}(\omega))_{|\alpha|\leq k} \in H_{\kappa}(\Theta)$ 
for all $\omega \in \Omega\backslash N$, see Remark \ref{Rem:H_Kappa}.  
Accordingly, $Y(\omega,\cdot) = \J^*( (Y^{(\alpha)}(\omega,\cdot))_{|\alpha|\leq k} )$
is a representative for $X(\omega) \in H^k(\Theta)$ for all $\omega \in \Omega\backslash N$.\\
(b) Substituting $Y$ into the left hand side of \eqref{Eq:ExpectationRelation} and applying Fubini shows that $t\mapsto \E[Y^{(\alpha)}(t)]$ is a representative 
for $\E[X]^{(\alpha)}$. Therefore, $\E[Y^{(\alpha)}(t)] = \E[X]^{(\alpha)}(t) = \E[X^{(\alpha)}](t) = \E[Y^{(\alpha)}](t)$ for almost every $t\in\Theta$.\\
(c) We can assume, without loss, that $X$ is centred. It follows from \eqref{Eq:DerivativeThf}, \eqref{Eq:CovarianceOperatorDefinition} 
and \eqref{Eq:CovarianceOperator}, that for any $f,g\in H^k(\Theta)$, 
\begin{align*}
	\sum_{|\alpha|,|\beta|\leq k} \int_{\Theta}\int_{\Theta} \partial^{(\alpha,\beta)}c_X(s,t) f^{(\alpha)}(s) g^{(\beta)}(t) \,ds\,dt
	&=\sum_{|\beta|\leq k} \int_\Theta [\Cov_k(X)f]^{(\beta)}(t)\, g^{(\beta)}(t) \,dt\nonumber\\
	&= \< \Cov_k(X)f,g \>_{H^k}\nonumber\\
	&= \E \< \<f,X\>_{H^k} X, g \>_{H^k}\nonumber\\
	&= \E \sum_{|\beta|\leq k} \int_\Theta \Bigg( \sum_{|\alpha|\leq k} \int_\Theta X^{(\alpha)}(s) X^{(\beta)}(t) f^{(\alpha)}(s) \,ds \Bigg) 
	g^{(\beta)}(t) \,dt. 
\end{align*}
Substituting $Y$ into the right hand side and applying Fubini yields
\begin{align}\label{Eq:SobolevCovarianceEquation}
	\sum_{|\alpha|,|\beta|\leq k} \int_{\Theta}\int_{\Theta} \partial^{(\alpha,\beta)}c_X(s,t) f^{(\alpha)}(s) g^{(\beta)}(t) \,ds\,dt
	= \sum_{|\alpha|,|\beta|\leq k} \int_{\Theta}\int_{\Theta} \E[Y^{(\alpha)}(s) Y^{(\beta)}(t)] f^{(\alpha)}(s) g^{(\beta)}(t) \,ds\,dt.
\end{align}
Moreover, also by Fubini, for $\varphi,\psi \in C_c^{\infty}(\Theta)$, 
\begin{align*}
	 \int_{\Theta}\int_{\Theta} \E[Y^{(\alpha)}(s) Y^{(\beta)}(t)] \varphi(s) \psi(t) \,ds\,dt
	 = (-1)^{|\alpha|+|\beta|} \int_{\Theta}\int_{\Theta} \E[ Y(s) Y(t) ] \varphi^{[\alpha]}(s) \psi^{[\beta]}(t) \,ds\,dt,
\end{align*}
which shows, that the weak derivatives of $\gamma_Y(s,t)=\E[Y(s) Y(t)]$ have representatives $\E[Y^{(\alpha)}(s) Y^{(\beta)}(t)]$, i.e.  
$\partial^{(\alpha,\beta)} c_Y(s,t) = \E[Y^{(\alpha)}(s) Y^{(\beta)}(t)]$ almost everywhere. The assertion thus follows from \eqref{Eq:SobolevCovarianceEquation}
\end{proof}

To avoid a overly pedantic notation, we will adhere to a common practice and identify a $H^k(\Theta)$-valued random element $X$ with its
representative $Y$ from Lemma \ref{Lem:MeasurableRepresentative}, as $\E\|X\|_{H^k}^2 < \infty$ will be our minimum requirement in the following.

\subsection{Weakly Differentiable Random Fields}
The abstract notion of a $H^k(\Theta)$-valued random element does not, by itself, illustrate how such an object can be constructed in practice.
However, Lemma \ref{Lem:MeasurableRepresentative} suggest that a random field with finite variance and sample paths in $H^k(\Theta)$ may act as a representative for such a random element. 
It is therefore desirable to identify sufficient conditions for random fields that guarantee
the required sample path properties, specifically weak differentiability.
Previous investigations in this direction suggest that the regularity of the sample paths is primarily, tough not exclusively, determined by the 
second order dynamics of the process. 
Accordingly, given the scope of this paper, only conditions that are directly related 
to the covariance kernel are discussed in the following. 
Regarding weak differentiability, extensive studies have been conducted: Scheuerer 
\cite{Scheuerer2010RegularitySecondOrderFields} investigated general second-order random fields, while more recently, Henderson 
\cite{Henderson2024SobolevRegularityGaussianFields} refined these results for Gaussian fields by 
sharpening them from sufficient to necessary.\\ 
In this context, we recall that a measurable random field $(X_t)_{t\in\Theta}$ is a $\B(\Theta)\otimes\A$-measurable map 
$\Theta \times \Omega \ni (t,\omega) \mapsto X_t(\omega) \in \R$.
If $\E[|X_t|]<\infty$ for all $t\in\Theta$, the mean function $\mu_X(t) := \E[X_t]$, $t\in \Theta$, is well-defined 
and measurable; the field is called centred if $\mu_X = 0$. 
Furthermore, if $\Var(X_t) < \infty$ for all $t\in\Theta$, the covariance kernel $\gamma_X(s,t) := \Cov(X_s,X_t)$, $s,t \in \Theta$,
is also well-defined and measurable. Specifically, $\sigma_X(t) = \sqrt{\gamma_X(t,t)}$ is called standard deviation function. 
In this case, $(X_t)_{t\in\Theta}$ is referred to as a second-order random field. A random field is called 
Gaussian, if for any $n\in \N$ and $t_1,\ldots,t_n \in \Theta$, the vector $(X_{t_1},\ldots,X_{t_n})$ follows a multivariate normal distribution.
We then write $(X_t)_{t\in\Theta} \sim \mathcal{GP}(\mu_X,\gamma_X)$. 

\begin{Prop}[\cite{Scheuerer2010RegularitySecondOrderFields}, \cite{Henderson2024SobolevRegularityGaussianFields}]\label{Prop:ScheuererHenderson}
Let $(X_t)_{t\in\Theta}$ be a measurable centred second order random field with covariance kernel $\gamma_X$. 
The sample paths of $(X_t)_{t\in\Theta}$ are in $H^k(\Theta)$ a.s.,
\begin{enumerate}
	\item[(a)] if for all $|\alpha| \leq k$, the partial derivative $\partial^{[\alpha,\alpha]}\gamma_X$ exists and
	is continuous on the diagonal of $\Theta \hspace{-0.05cm} \times \hspace{-0.05cm} \Theta$ such that
	\begin{align*}
		\int_{\Theta} \partial^{[\alpha,\alpha]}\gamma_X(t,t) \,dt < \infty.
	\end{align*}
	
	\item[(b)] if and only if: $(X_t)_{t\in\Theta} \sim \mathcal{GP}(\mu_X,\gamma_X)$ with $\sigma_X \in L^1_{loc}(\Theta)$. 
	Moreover, for all $|\alpha| \leq k$ the weak derivative $\partial^{(\alpha,\alpha)}\gamma_X$ exists and is in 
	$L^2(\Theta \times \Theta)$ such that the associated $L^2$-integral operator
	\begin{align}\label{Eq:NuclearityConditionL2}
		[T_{\partial^{(\alpha,\alpha)}\gamma_X,0}f](t) = \int_\Theta \partial^{(\alpha,\alpha)}\gamma_X(s,t) f(s) \,ds \quad \text{a.e.}
	\end{align}
	is nuclear.
\end{enumerate}
\end{Prop} 

At this point, we may emphasize what almost sure $L^2$-weak differentiability of the paths of $(X_t)_{t\in\Theta}$ means precisely:
for each $|\alpha| \leq k$ there exists a measurable process $(X_t^{(\alpha)})_{t\in\Theta}$ with sample paths in $L^2(\Theta)$ almost surely, 
such that,
\begin{align}\label{Eq:AlmostSureWeakDiffPaths}
	\int_{\Theta} X_t\, \varphi^{[\alpha]}(t) \,dt 
	= (-1)^{|\alpha|} \int_{\Theta} X_t^{(\alpha)} \, \varphi(t)\,dt \quad\text{a.s.}
\end{align} 
for all $\varphi \in C_c^{\infty}(\Theta)$. 
Note, that we can always redefine $(X_t^{(\alpha)})_{t\in\Theta}$ on the null set, where the sample paths fail square-integrability, to
obtain an indistinguishable measurable modification whose sample paths are exclusively in $L^2(\Theta)$. For convenience, this modification is 
also denoted by $(X_t^{(\alpha)})_{t\in\Theta}$ as from now on, we will only consider these modifications. 
The following lemma confirms that random fields with almost surely weakly 
differentiable sample paths allow for the interpretation as random elements in $H^k(\Theta)$.

\begin{Lem}\label{Lem:RandomFieldsAsSobolevElements}
Let $(X_t)_{t\in\Theta}$ be a measurable centred random field.
Suppose the sample paths of $(X_t)_{t\in\Theta}$ are in $H^k(\Theta)$ a.s., 
then $\Omega \ni \omega \mapsto X(\omega) = (X_t(\omega))_{t\in\Theta}$ defines a $H^k(\Theta)$-valued random element. 
Moreover, if $\E[X_t^2] <\infty$ for all $t\in\Theta$ and 
$\gamma_X \in H^{k,k}(\Theta\times\Theta)$, then $\E\|X\|_{H^k}^2 <\infty$ and $\Cov_k(X) = T_{\gamma_X,k} \in N(H^k)$.
\end{Lem}

\begin{proof}
By assumption, there exist measurable processes $(X_t^{(\alpha)})_{t\in\Theta}$ with sample paths in $L^2(\Theta)$ satisfying 
\eqref{Eq:AlmostSureWeakDiffPaths}. It then follows from Tonelli's theorem that $\omega \mapsto \int_\Theta X_t^{(\alpha)}(\omega)f(t)\,dt$ is 
$\A$-measurable for any $f \in L^2(\Theta)$ and, hence, $\omega \mapsto (X_t^{(\alpha)}(\omega))_{t\in\Theta}$ defines a $L^2(\Theta)$-valued random 
element for each $|\alpha| \leq k$ (c.f., \cite{Hsing2015FoundationsOfFDA}, Theorem 7.4.1). Observe, that for 
$\alpha=(\alpha_1,\ldots,\alpha_d) \in \N_0^d$ with $|\alpha| < k$ and $e_j=(\delta_{ij}, i=1,\ldots,d) \in\{0,1\}^d$ for $j\in\{1,\ldots,d\}$, the 
relation \eqref{Eq:AlmostSureWeakDiffPaths} implies 
\begin{align*}
	\int_{\Theta} X_t^{(\alpha)}\,\varphi^{[e_j]}(t)\, dt 
	= - \int_{\Theta} X_t^{(\alpha+e_j)}\, \varphi(t) \,dt.
\end{align*}
Successive iteration of this identity and the $\A$-$\B(L^2)$-measurability of each $(X_t^{(\alpha)})_{t\in\Theta}$ then allow to conclude that
$\omega \mapsto ( (X_t^{(\alpha)}(\omega))_{t\in\Theta} )_{|\alpha|\leq k}$ is a $H_{\kappa}(\Theta)$-valued random element. 
As $\J^*\colon H_{\kappa}(\Theta) \to H^k(\Theta)$ is continuous, this in turn yields $X=(X_t)_{t\in\Theta}$ is a $H^k(\Theta)$-valued random element.
Now, suppose $\E|X_t|^2 < \infty$ and $\gamma_X \in H^{k,k}(\Theta\times\Theta)$. Then, for $|\alpha|,|\beta|\leq k$ and $\phi,\psi \in C_c^{\infty}(\Theta)$, 
\begin{align*}
	\E[ \< X^{(\alpha)}_\bullet, \phi \>_{L^2}\,\< X^{(\beta)}_\bullet, \psi \>_{L^2} ]
	&= (-1)^{|\alpha|+|\beta|} \int_{\Theta}\int_{\Theta} \E[X_sX_t] \varphi^{[\alpha]}(s)\psi^{[\beta]}(t) \,ds\,dt\\
	&= \int_{\Theta}\int_{\Theta} \partial^{(\alpha,\beta)}\gamma_X(s,t) \varphi(s)\psi(t) \,ds\,d,
\end{align*}
by an application of Fubini.
Because $C_c^{\infty}(\Theta)$ is dense in $L^2(\Theta)$, this shows that the cross-covariance operators 
$\Cov_0(X^{(\beta)}_\bullet, X^{(\alpha)}_\bullet) \in N(L^2)$ are associated to $\partial^{(\alpha,\beta)}\gamma_X \in L^2(\Theta\times\Theta)$.
We can thus conclude from Remark \ref{Rem:ExAndCovInH_kappa}, that
\begin{align*}
	[\Cov_k(X)f](t)
	&= [\J^*\,\Cov(\J X)\J f](t)
	= \sum_{|\alpha|\leq k} \int_\Theta \partial^{(\alpha,0)}\gamma_X(s,t)\,f^{(\alpha)}(s) \,ds
\end{align*}
for $f\in H^k(\Theta)$, which establishes the second assertion.
\end{proof}
 
Gaussian random fields with weakly differentiable sample paths admit a natural identification as Gaussian elements in a 
Sobolev space, cf. \cite{Henderson2024SobolevRegularityGaussianFields}.

\begin{Cor}[Gaussian Fields] \label{Cor:GaussianProcess}
Let $(X_t)_{t\in\Theta} \sim \mathcal{GP}(0,\gamma_X)$ be measurable with $\gamma_X \in H^{k,k}(\Theta\times\Theta)$. 
Suppose the sample paths of $(X_t)_{t\in\Theta}$ are  in $H^k(\Theta)$ almost surely.
Then $X:=(X_t)_{t\in\Theta} \sim \NN_{H^k}(0,\Cov_k(X))$ with $\Cov_k(X) = T_{\gamma_X,k}\in N(H^k)$.
\end{Cor}

\begin{proof}	
In view of Lemma \ref{Lem:RandomFieldsAsSobolevElements} it remains to verify that $X=(X_t)_{t\in\Theta}$ is also Gaussian.
According to Lemmas \ref{Lem:TestfunctionsAreSeparating2} and \ref{Lem:GaussianInBanchSpace} it suffices
to establish that $\< X_\bullet^{(\alpha)},\varphi \>_{L^2}$ is normally distributed for all $\varphi \in C_c^{\infty}(\Theta)$ and $|\alpha| \leq k$. 
To that end, first note that by \eqref{Eq:AlmostSureWeakDiffPaths} the map 
$\omega \mapsto Z(\omega) :=\< X_\bullet^{(\alpha)}(\omega), \varphi \>_{L^2}$ is a well-defined square integrable random variable. 
Moreover, let $G:= \ol\Span\{ X_t\mid t\in\Theta \}$. Then $G$ is a closed linear subspace in $L^2(\P)$ and it suffices to establish 
$Z \in G$, since $G$ consists of Gaussian random variables exclusively. However, $Z = Z_1 + Z_2$ for $Z_1 \in G$ and $Z_2 \perp G$ by 
the projection theorem and, therefore, 
\begin{align*}
	\E[ Z Z_2 ]
	= \E \Big[ \int_{\Theta}  X_t^{(\alpha)}\varphi(t)\,dt\, Z_2 \Big]
	= (-1)^{|\alpha|} \int_\Theta \E [ X_t Z_2] \,\varphi^{[\alpha]}(t) \,dt = 0,
\end{align*}
due to Fubini's theorem. This shows $Z \perp Z_2$ in $L^2(\P)$, that is $Z \in G$. 
\end{proof}

\begin{Ex}[Positive Definite Kernels]\label{Ex:DPKernels}
Suppose $h\in C^{k,k}
(\ol\Theta\times\ol\Theta)$ is positive definite. Then there exists a $X\sim \NN_{H^k}(0,\Cov_k(X))$ with $\Cov_k(X)=T_{h,k} \in N(H^k)$,
where $T_{h,k}$ is defined via \eqref{Eq:AssociatedOperator}.
\begin{proof}
Since $h$ is assumed positive-definite, it is well-known that as a consequence of Kolmogorov's consistency theorem, there exists a centred 
Gaussian field $(X_t)_{t\in \Theta}$ with $\E[X_sX_t] = h(s,t)$, i.e. $h$ is the covariance kernel of $(X_t)_{t\in \Theta}$. 
Moreover, $h\in V^k(\ol\Theta\times \ol\Theta)$ implies that there exists a measurable modification with sample paths in $H^k(\Theta)$ a.s. by Proposition 
\ref{Prop:ScheuererHenderson}. 
Therefore, Corollary \ref{Cor:GaussianProcess} ensures existence of
a centred $H^k(\Theta)$-valued Gaussian element $X$ with $\Cov_k(X)=T_h$. 
\end{proof}
\end{Ex}

\begin{Ex}[Integrated Brownian motion]\label{Ex:IntegratedBM}
A prototypical $H^1((0,1))$-valued Gaussian element is obtained by simply integrating a Brownian motion (BM) to enforce the desired differentiability. 
To elaborate, let $(B_t, 0\leq t \leq 1)$ be a BM and for $t\in[0,1]$ define
\begin{align}\label{Eq:IntegratedBM}
	X_t := \int_0^t B_u \,du.
\end{align}
Then $(X_t, 0\leq t\leq 1)$ is called an \emph{integrated Brownian motion} (IBM). Note, that both processes are Gaussian and also measurable due to their continuity. 
Moreover, the sample paths of $(X_t)_t$ are (weakly) differentiable by construction such that $(X_t^{(1)})_t=(X_t^{[1]})_t = (B_t)_t$, which may also be verified directly, 
cf. \eqref{Eq:AlmostSureWeakDiffPaths}: For $\varphi \in C_c^{\infty}((0,1))$,
\begin{align*}
	-\int_0^1 X_t \,\varphi^{[1]}(t) \,dt 
	= -\int_0^1 \int_0^t B_u \,du \varphi^{[1]}(t) \,dt
	= \int_0^1 B_u\varphi(u) \,du.
\end{align*}
This shows $(X_t)_t$ has sample paths in $H^k(\Theta)$ with mean function $\mu_X$ and covariance kernel $\gamma_X$ given by:
\begin{align*}
	\mu_X(t) 
	&= \E[ X_t] = \int_0^t \E [B_u] \,du = 0,\\
	\gamma_X(s,t) 
	&= \E[X_s X_t] 
	= \int_0^t \int_0^s \min(u,v) \,dv\,du
	= \frac{1}{2}\min(s,t)^2\max(s,t) - \frac{1}{6}\min(s,t)^3.
\end{align*}
Differentiating $\gamma_X$ with respect to $s$ yields
\begin{align*}
	\partial^{[1,0]} \gamma_X(s,t) = 
	\begin{cases}
		st - \frac{1}{2}s^2 &, s\leq t; \\
		\frac{1}{2}t^2 &, s> t.
	\end{cases}
\end{align*}
On the other hand, we also have
\begin{align*}
	\E[X_s^{[1]} X_t] = \E[ B_s \hspace{-0.1cm}\int_0^t B_u \,du ]
	= \int_0^t \hspace{-0.1cm}\min(s,u) \,du
	= \begin{cases}
		st - \frac{1}{2}s^2 \hspace{-0.3cm}&, s\leq t; \\
		\frac{1}{2}t^2 \hspace{-0.3cm} &, s> t,
	\end{cases}
\end{align*}
which confirms $\partial^{[1,0]} \gamma_X(s,t) = \E[X_s^{[1]} X_t]$, cf. Lemma \ref{Lem:MeasurableRepresentative} (c). Similarly, 
\begin{align*}
	\partial^{[0,1]} \gamma_X(s,t) =
	\left.
	\begin{cases}
		st - \frac{1}{2}t^2 \hspace{-0.3cm}&, t\leq s \\
		\frac{1}{2}s^2 \hspace{-0.3cm} &, t> s
	\end{cases}
	\right\}
	= \E[X_s X_t^{[1]}] 
\end{align*}
and 
\begin{align*}
	\partial^{[1,1]} \gamma_X(s,t) = 
	\left.
	\begin{cases}
		s &, s\leq t \\
		t &, s> t
	\end{cases}
	\right\} 
	= \min(s,t)
	= \E[X_s^{[1]} X_t^{[1]}].
\end{align*}
Therefore, $\gamma_X \in C^{1,1}1([0,1]^2)$ and it follows from Corollary \ref{Cor:GaussianProcess}, that $X=(X_t)_t \sim \NN_{H^1}(0, \Cov_1(X))$ with 
$\Cov_1(X) = T_{\gamma_X,1}$. 
Specifically, the equivalence stated in Proposition \ref{Prop:ScheuererHenderson} (b) is demonstrated as follows: 
On the one hand, it is apparent that $\gamma_X,\, \partial^{[1,1]}\gamma_X \in L^2((0,1)^2)$ and $\partial^{[1,1]}\gamma_X$ is again positive definite by \eqref{Eq:Non-negativityThL2}. 
Since both kernels are continuous, the associated $L^2$-integral operators 
$T_{\gamma_x,0}$ and  $T_{\partial^{[1,1]}\gamma_X,0}$ are nuclear (cf. Remark \ref{Rem:NuclearNorm}), which in turn ensures $(X_t)_t$ has (weakly) 
differentiable sample paths. 
On the other hand, (weak) differentiability of $(X_t)_t$ yields that 
$\partial^{[1,1]}\gamma_X$ is the covariance kernel of $X^{[1]}$, which immediately implies $T_{\partial^{[1,1]}\gamma_X,0}$ is nuclear.\\
Furthermore, the trace-identity in Corollary \ref{Cor:NuclearNorm}, respectively, \eqref{Eq:CovarianceNuclearNorm} can be applied to obtain
\begin{align*}
	\|\Cov_1(X)\|_{N(H^1)}
	&= \|\Cov_0(X)\|_{N(L^2)} + \|\Cov_0(X^{[1]})\|_{N(L^2)}\\
	&= \int_0^1 \gamma_X(t,t) \,dt + \int_0^1 \partial^{[1,1]} \gamma_X(t,t)\,dt\\
	&= \frac{1}{3} \int_0^1 t^3 \,dt + \int_0^1 t\,dt 
	= \frac{1}{12} + \frac{1}{2} = \frac{7}{12}.
\end{align*} 
This illustrates how
the $H^1$-covariance operator incorporates the total variance of both, the original process (IBM) 
and of its (weak) derivative (BM) in an additive way.
\end{Ex}

\subsection{Mean Square Optimal Approximations}
In this final section, the results from Section \ref{SubSec:ExtensionsOfMercersTheorem} regarding kernel expansions in Sobolev spaces are applied 
to the covariance kernel of a weakly differentiable random field, i.e. a $H^k(\Theta)$-valued random element. This allows for the derivation of 
approximations that are not only mean square optimal in $H^k(\Theta)$ but also uniformly convergent.\\

Let $X$ be a $H^k(\Theta)$-valued random element. Then $X$ can be expanded in any orthonormal basis $\{ u_j\mid j\in\N \}$ of $H^k(\Theta)$,
that is $X = \sum_j \< X,u_j \>_{H^k}u_j$ in $H^k(\Theta)$ almost surely. Summing only a specific 
number $m\in\N$ of terms then creates a $m$-dimensional approximation for $X$. It is a long-studied problem to find a basis such that the corresponding
$m$-dimensional approximation is optimal with respect to a given error-measure. Suppose $X$ is centred and $\E\|X\|_{H^k}^2<\infty$, 
then a common error-measure in this regard is the $H^k$-mean square error given by 
\begin{align}\label{Eq:MeanSquareError}
	\E \big\| X - \sum_{|j|\leq m} \< X,u_j \>_{H^k} u_j  \big\|_{H^k}^2
	= \E \|X\|_{H^k}^2 - \sum_{|j|\leq m} \< \Cov_k(X)u_j, u_j \>_{H^k}.
\end{align} 
It is well-known (e.g., \cite{Hsing2015FoundationsOfFDA}, Theorem 7.28) that \eqref{Eq:MeanSquareError} is minimized by choosing the normalized 
eigenvectors of the covariance operator $\Cov_k(X) \in N(H^k)$. That is, the optimal $u_j = e_{j,k}$ are determined as solutions of
\begin{align}\label{Eq:EigenEquation_Sobolev}
	[\Cov_k(X)e_{j,k}](t) 
	= \int_\Theta \sum_{|\alpha|\leq k} \partial^{(\alpha,0)}c_X(s,t)\,e_{j,k}^{(\alpha)}(s) \,ds 
	= \lambda_j e_{j,k}(t) \quad \text{a.e.}
\end{align} 
subject to $\|e_{j,k}\|_{H^k}=1$ with corresponding eigenvalues $\lambda_{j,k}$ counted by multiplicity. 
Recall that $\sum_j \lambda_{j,k}=\|\Cov_k(X)\|_{N(H^k)}=\E\|X\|_{H^k}^2$ and let 
\begin{align}\label{Eq:SobolevApprox}
	X_{m,k} = \sum_{j\leq m} \<X,e_{j,k}\>_{H^k} \,e_{j,k}.
\end{align}
By construction, the scores $\<X,e_{j,k}\>_{H^k}$ in terms of the $H^k$-eigenbasis satisfy
\begin{align}\label{Eq:ScoresUncorrelated}
	\Cov( \<X,e_{i,k}\>_{H^k}, \<X,e_{j,k}\>_{H^k} ) = \<\Cov_k(X)e_{i,k}, e_{j,k}\>_{H^k} = \lambda_{j,k} \delta_{ij},
\end{align}
showing that the crucial biorthogonality property of the original KL-expansion is - of course - retained. 
Substituting \eqref{Eq:SobolevApprox} into \eqref{Eq:MeanSquareError} then yields
\begin{align}\label{Eq:SobolevApproxError}
	\E\| X - X_{m,k}\|_{H^k}^2 
	= \sum_{|\alpha|\leq k} \E \| X^{(\alpha)}-X_{m,k}^{(\alpha)} \|_{L^2}^2
	= \sum_{j>m} \lambda_{j,k}
\end{align}
and the right hand side can not be reduced any further by choosing another basis. 
This shows, the expansion \eqref{Eq:SobolevApprox}
approximates $X$ as well as all it's weak derivatives $X^{(\alpha)}$, $|\alpha| \leq k$, \emph{simultaneously} in a $L^2$-mean square optimal sense.

\paragraph{KL-Representations in Sobolev Spaces} 
Transitioning to (weakly) differentiable random fields, it is natural to ask whether $X_{m,k}$ and its (weak) derivatives also converge in $L^2(\P)$ uniformly in $t$, 
allowing for consistent control of the truncation error across the domain $\Theta$.  
The subsequent refinements of the original KL-expansion are immediate consequences of Theorem \ref{Thrm:Mercer} and 
Corollary \ref{Cor:MercerRateOfConvergence}.

\begin{Thrm}[Refined KL-Representation]\label{Thrm:KLExpansionSobolev}
Let $(X_t)_{t\in\Theta}$ be a measurable centred second order random field with covariance kernel
$\gamma_X \in C^{k,k}(\ol\Theta\times \ol\Theta)$ and let $\{ (\lambda_{j,k}, e_{j,k})\mid j\in\N \}$ 
denote the eigensystem of the associated kernel operator $T_{\gamma_X,k} \in N(H^k)$.  
Then $X=(X_t)_{t\in\Theta}$ is a $H^k(\Theta)$-valued random element and, for $|\alpha| \leq k$,
\begin{align}\label{Eq:UniformConvergence}
	\lim_{m\to\infty} \,\sup_{t\in \Theta} \,\E \Big| X_t^{(\alpha)} - \sum_{j\leq m} \< X, e_{j,k} \>_{H^k} \,e_{j,k}^{[\alpha]}(t) \Big|^2 = 0.
\end{align}
\end{Thrm}

\begin{proof}
First, note that Proposition \ref{Prop:ScheuererHenderson} (a) assesses that the sample paths of $(X_t)_{t\in\Theta}$ are in $H^k(\Theta)$. 
Therefore, by Lemma \ref{Lem:RandomFieldsAsSobolevElements}, $X=(X_t)_{t\in\Theta}$ is a $H^k(\Theta)$-valued random element with covariance operator
$\Cov_k(X) = T_{\gamma_X,k}$. 
It follows from Theorem \ref{Thrm:Mercer} that $e_j \in C^k(\ol\Theta)$, so the pointwise evaluation in \eqref{Eq:UniformConvergence} is well-defined 
and we have 
\begin{align}\label{Eq:PointwiseMSError}
	&\E \Big| X_t^{(\alpha)} - \sum_{j\leq m} \< X, e_j \>_{H^k}\, e_j^{[\alpha]}(t) \Big|^2\nonumber\\
	&= \partial^{[\alpha,\alpha]}\gamma_X(t,t) 
	- 2 \sum_{j\leq m} \int_\Theta \sum_{|\beta|\leq k} \E[ X_s^{(\beta)}X_t^{(\alpha)}]\, e_j^{[\beta]}(s) \,ds\, e_j^{[\alpha]}(t)
	+ \sum_{i,j\leq m} \< \E\< e_i,X \>_{H^k}X, e_j \>_{H^k} e_i^{[\alpha]}(t) \,e_j^{[\alpha]}(t)\nonumber\\
	&= \partial^{[\alpha,\alpha]}\gamma_X(t,t) 
	- 2 \sum_{j\leq m} \int_\Theta \sum_{|\beta|\leq k} \partial^{[\beta,\alpha]}\gamma_X(s,t)\, e_j^{[\beta]}(s) \,ds\, e_j^{[\alpha]}(t)
	+ \sum_{i,j\leq m} \<\Cov_k(X)e_j, e_j \>_{H^k} e_i^{[\alpha]}(t) \,e_j^{[\alpha]}(t)\nonumber\\
	&= \partial^{[\alpha,\alpha]}\gamma_X(t,t) 
	- 2 \sum_{j\leq m} [\Cov_k(X)e_j]^{[\alpha]}(t)\, e_j^{[\alpha]}(t)
	 + \sum_{j\leq m} \lambda_j\, e_j^{[\alpha]}(t) \, 
	e_j^{[\alpha]}(t)\nonumber\\
	&= \partial^{[\alpha,\alpha]}c_X(t,t) - \sum_{j\leq m} \lambda_j \, e_j^{[\alpha]}(t) \, e_j^{[\alpha]}(t).
\end{align}
By Theorem \ref{Thrm:Mercer} the right hand side vanishes uniformly over $\ol\Theta$ as $m$ increases to infinity.
\end{proof}

\begin{Cor}\label{Cor:KLExpansionSobolev}
Granted the conditions of Theorem \ref{Thrm:KLExpansionSobolev} are satisfied. Suppose $\Theta$ has the cone property and $k-d/2 > s$ for some $s\in\N_0$.
Then, for $|\alpha|\leq s$,
\begin{align}\label{Eq:UniformConvergenceRate}
	\sup_{t\in \ol\Theta} \,\E \Big| X_t^{(\alpha)} - \sum_{j\leq m} \< X, e_{j,k} \>_{H^k} \,e_{j,k}^{[\alpha]}(t) \Big|^2
	\leq C^2 \sum_{j>m} \lambda_{j,k},
\end{align}
where $C>0$ is the constant in \eqref{Eq:UniformBoundEigenfunctions}.
\end{Cor}

\begin{proof}
This follows immediately from \eqref{Eq:PointwiseMSError} and Corollary \ref{Cor:MercerRateOfConvergence}.
\end{proof}

\begin{Ex}\label{Ex:KLExpansionSobolev}
Let $k\geq 1$ and $\Theta=(0,b) \subset \R$. 
Suppose $(X_t)_{t\in\Theta}$ is a measurable centred second order stochastic process with covariance kernel
$\gamma_X \in C^{k,k}([0,b]^2)$. 
Then by Corollary \ref{Cor:KLExpansionSobolev}, substituting $C=\sqrt{\coth(b)}$ from Remark \ref{Rem:UniformBoundEigenfunction} 
into \eqref{Eq:UniformConvergenceRate} yields,
\begin{align*}
	\sup_{t\in [0,b]} \,\E \Big| X_t^{(\alpha)} - \sum_{j\leq m} \< X, e_{j,k} \>_{H^k} \,e_{j,k}^{[\alpha]}(t) \Big|^2
	\leq \coth(b) \sum_{j>m} \lambda_{j,k}.
\end{align*}   
\end{Ex} 

In the Gaussian case, the expansion \eqref{Eq:SobolevApprox} also converges almost surely uniformly - that is, pathwise. 

\begin{Cor}\label{Cor:KLGaussian}
Let $(X_t)_{t\in\Theta}\sim \mathcal{GP}(0, \gamma_X)$ be measurable with $\gamma_X \in C^{k,k}(\ol\Theta \times \ol\Theta)$ and
let $\{ (\lambda_{j,k}, e_{j,k})\mid j\in\N \}$ denote the eigensystem of the associated kernel operator $T_{\gamma_X,k} \in N(H^k)$. 
Then $X=(X_t)_{t\in\Theta} \sim \NN_{H^k}(0, T_{\gamma_X,k})$ and if $X^{(\alpha)}=(X_t^{(\alpha)})_{t\in\Theta}$ has uniformly continuous sample paths  
\begin{align}\label{Eq:KLGaussian}
	\P \Big( \lim_{m\to \infty} \,\sup_{t\in \ol\Theta}\Big|X_t^{(\alpha)} - \sum_{j\leq m} Z_{j,k}\, e_{j,k}^{[\alpha]}(t) \Big| =0\Big) =1,
\end{align}
where $Z_{j,k} = \<X, e_{j,k}\>_{H^k} \sim \NN(0,\lambda_{j,k})$ are independent.
\end{Cor} 

\begin{proof}
By Proposition \ref{Prop:ScheuererHenderson} the sample paths of $X=(X_t)_{t\in\Theta}$ are in $H^k(\Theta)$ and, hence, 
$X\sim \NN_{H^k}(0, T_{\gamma_X,k})$ by Corollary \ref{Cor:GaussianProcess}. 
This implies, the vector $(\<X, e_j\>_{H^k}\mid j=1,\ldots,m)$ has a joint normal distribution for any $m\in\N$. 
Thus, the scores $\< X,e_j \>_{H^k}$ are not only uncorrelated but also independent and satisfy $\<X,e_j\>_{H^k} = \,Z_j \sim \NN(0,\lambda_j)$ 
by \eqref{Eq:ScoresUncorrelated}. Specifically, since $e_j \in C^k(\ol\Theta)$ by Theorem \ref{Thrm:Mercer}, $Z_j e_{j}^{[\alpha]}$
are Gaussian $C(\ol\Theta)$-valued random elements and for $t\in \Theta$ the sum $\sum_j Z_j e_j^{[\alpha]}(t)$ converges in probability to 
$X_t^{(\alpha)}$ by Theorem \ref{Thrm:KLExpansionSobolev}. Now consider $|\alpha| \leq k$ for which $(X_t^{(\alpha)})_{t\in\Theta}$ has by assumption 
uniformly continuous sample paths. Then $(X_t^{(\alpha)})_{t\in\Theta}$ is also a $C(\ol\Theta)$-valued random element and because the 
set of point-evaluation functionals is separating for the Banach space $C(\ol\Theta)$, \eqref{Eq:KLGaussian} follows from an application of the Ito-Nisio 
theorem in Banach spaces (e.g., \cite{Vakhania1987ProbabilityDistributionsOnBanachSpaces}, Chapter V, Theorem 2.4).
\end{proof}

\begin{Ex}\label{Ex:EigenExpansionIntBM}
To give an illustrative example, that also allows for an immediate comparison to the classical Karhunen-Loève theory ($k=0$), we consider again the 
integrated Brownian motion (IBM) with covariance kernel 
\begin{align}\label{Eq:IntegratedBM_kernel}
	\gamma_X(s,t) =  \frac{1}{2}\min(s,t)^2\max(s,t) - \frac{1}{6}\min(s,t)^3
\end{align}
already discussed in Example \ref{Ex:IntegratedBM}. The respective eigensystems $\{ (\lambda_{j,k}, e_{j,k})\mid j\in\N \}$ in $L^2\,(k=0)$ and 
$H^1\,(k=1)$ can be derived as follows:\\

(i) The $L^2$-eigensystem $\{ (\lambda_{j,0}, e_{j,0})\mid j\in \N \}$ is obtained by solving \eqref{Eq:EigenEquation_Sobolev} for $k=0$,
\begin{align}\label{Eq:KLforIntegratedBM_L2_EigenvectorEq}
	\lambda \, e(t)
	&=\int_0^1 \gamma_X(s,t)\, e(s) \,ds \nonumber\\
	&= \int_0^t \left(\frac{1}{2}s^2t - \frac{1}{6}s^3 \right)e(s) \,ds + \int_t^1 \left(\frac{1}{2}t^2s - \frac{1}{6}t^3 \right)e(s) \,ds.
\end{align}
Successive differentiation transforms this equation into an equivalent ordinary differential equation
\begin{align}\label{Eq:KLforIntegratedBM_L2_ODE}
	\lambda \, e^{[4]}(t) = e(t), \quad 0\leq t \leq 1
\end{align}
with boundary conditions 
\begin{align}\label{Eq:KLforIntegratedBM_L2_BoundaryConditions}
	e(0) = e^{[1]}(0) = 0\quad\text{and}\quad  e^{[2]}(1) = e^{[3]}(1) = 0.
\end{align}
The general solution of \eqref{Eq:KLforIntegratedBM_L2_ODE} is then given by 
\begin{align*}
	e(t) = c_1 \sinh\Big(\frac{t}{\sqrt[4]{\lambda}}\Big) +c_2 \cosh\Big(\frac{t}{\sqrt[4]{\lambda}}\Big)
	+c_3 \sin\Big(\frac{t}{\sqrt[4]{\lambda}}\Big)+c_4 \cos\Big(\frac{t}{\sqrt[4]{\lambda}}\Big)
\end{align*}
and the boundary conditions \eqref{Eq:KLforIntegratedBM_L2_BoundaryConditions} imply $c_3=-c_1$, $c_4=-c_2$ as well as 
\begin{align}\label{Eq:KLforIntegratedBM_L2_CoeffRatio}
	\frac{c_1}{c_2} 
	= \frac{ \sinh\big( \sqrt[4]{1/\lambda} \big) - \sin\big( \sqrt[4]{1/\lambda}\big)}{ \cosh\big( \sqrt[4]{1/\lambda} \big) 
		+ \cos\big( \sqrt[4]{1/\lambda} \big)}
\end{align}
and
\begin{align}\label{Eq:KLforIntegratedBM_L2_EigenvalueEq}
	1 + \cosh\big( \sqrt[4]{1/\lambda} \big) \,\cos\big( \sqrt[4]{1/\lambda} \big) = 0.
\end{align}
Thus, a $L^2$-normalized solution of \eqref{Eq:KLforIntegratedBM_L2_EigenvectorEq}, respectively, \eqref{Eq:KLforIntegratedBM_L2_ODE} with 
boundary conditions \eqref{Eq:KLforIntegratedBM_L2_BoundaryConditions}, can be found as follows: Compute the eigenvalues $\{ \lambda_{j,0} \}$ by 
solving the eigenvalue equation \eqref{Eq:KLforIntegratedBM_L2_EigenvalueEq} and determine the corresponding eigenfunctions of the form 
\begin{align}\label{Eq:EigenfunctionIntBMinL2}
	e_{j,0}(t) 
	= c_{1,j}\Big[ \sinh\Big( \frac{t}{\sqrt[4]{\lambda_{j,0}}} \Big) - \sin\Big( \frac{t}{\sqrt[4]{\lambda_{j,0}}} \Big)  \Big]
	+ c_{2,j}\Big[ \cosh\Big( \frac{t}{\sqrt[4]{\lambda_{j,0}}} \Big) - \cos\Big( \frac{t}{\sqrt[4]{\lambda_{j,0}}} \Big) \Big] 
\end{align}
by choosing $c_{1,j},c_{2,j}\in \R$ such that \eqref{Eq:KLforIntegratedBM_L2_CoeffRatio} and $\|e_{j,0}\|_{L^2} = 1$ are satisfied.\\

(ii) The $H^1$-eigensystem $\{(\lambda_{j,1}, e_{j,1})\mid j\in\N\}$ is obtained by solving \eqref{Eq:EigenEquation_Sobolev} for $k=1$,
\begin{align}\label{Eq:KLforIntegratedBM_H1_EigenvectorEq}
	\lambda \, e(t)
	&= \int_0^1 \gamma_X(s,t) e(s) + \partial^{[1,0]}\gamma_X(s,t) e^{(1)}(s)  \,ds \nonumber\\
	\begin{split}
		&= \int_0^t \left(\frac{1}{2}s^2t - \frac{1}{6}s^3 \right)e(s) + \left(st-\frac{1}{2}s^2\right)e^{(1)}(s) \,ds
		+ \int_t^1 \left(\frac{1}{2}t^2s - \frac{1}{6}t^3 \right)e(s) + \frac{1}{2} t^2 e^{(1)}(s) \,ds.
	\end{split}
\end{align}
Again, successive differentiation transforms this equation into an equivalent ordinary differential equation
\begin{align}\label{Eq:KLforIntegratedBM_H1_ODE}
	\lambda \, e^{[4]}(t) = e(t) - e^{[2]}(t), \quad 0\leq t \leq 1
\end{align}
with boundary conditions 
\begin{align}\label{Eq:KLforIntegratedBM_H1_BoundaryConditions} 
	e(0) = e^{[1]}(0) = 0 \quad\text{and}\quad e^{[2]}(1)= \lambda\, e^{[3]}(1) + e^{[1]}(1) = 0. 
\end{align}
The general solution of \eqref{Eq:KLforIntegratedBM_H1_ODE} is then given by 
\begin{align*}
	e(t) &= c_1\sinh\Big( \sqrt{\frac{2}{\theta+1}}\,t \Big) + c_2\cosh\Big( \sqrt{\frac{2}{\theta+1}}\,t \Big) + c_3\sin\Big( \sqrt{\frac{2}{\theta-1}}\,t \Big)
	+ c_4\cos\Big( \sqrt{\frac{2}{\theta-1}}\,t \Big)
\end{align*}
with $\theta = \sqrt{4\lambda+1}$ for $\lambda>0$. Moreover, the boundary conditions \eqref{Eq:KLforIntegratedBM_H1_BoundaryConditions} imply, 
\begin{align*}
	c_4=-c_2, \qquad c_3 = - \sqrt{\frac{\theta -1}{\theta +1}} \, c_1,
\end{align*}
as well as, 
\begin{align}\label{Eq:KLforIntegratedBM_H1_CoeffRatio}
	\frac{c_1}{c_2}
	= - \frac{(\theta-1)\cosh\left( \sqrt{\frac{2}{\theta+1}} \right) + (\theta+1)\cos\left( \sqrt{\frac{2}{\theta-1}}\right)}
	{(\theta-1)\sinh\left( \sqrt{\frac{2}{\theta+1}} \right) +\sqrt{\theta^2-1}\sin\left( \sqrt{\frac{2}{\theta-1}} \right)}
\end{align}
and 
\begin{align}\label{Eq:KLforIntegratedBM_H1_EigenvalueEq}
	\begin{split}
		&8\lambda + (8\lambda +4)\cosh\left( \sqrt{\frac{2}{\sqrt{4\lambda+1}+1}} \right)\,\cos\left( \sqrt{\frac{2}{\sqrt{4\lambda+1}-1}} \right)\\
		&-4\sqrt{\lambda} \,\sinh\left( \sqrt{\frac{2}{\sqrt{4\lambda+1}+1}} \right)\,\sin\left( \sqrt{\frac{2}{\sqrt{4\lambda+1}-1}} \right) = 0.
	\end{split}
\end{align}
Thus, a $H^1$-normalized solution of \eqref{Eq:KLforIntegratedBM_H1_EigenvectorEq}, respectively, \eqref{Eq:KLforIntegratedBM_H1_ODE} with boundary conditions 
\eqref{Eq:KLforIntegratedBM_H1_BoundaryConditions} can be found as follows: First compute the eigenvalues $\{\lambda_{j,1}\}$ by solving 
the eigenvalue equation \eqref{Eq:KLforIntegratedBM_H1_EigenvalueEq} 
and set $\theta_j = \sqrt{4\lambda_{j,1}+1}$. Then, determine the corresponding eigenfunctions of the form 
\begin{align}\label{Eq:EigenfunctionIntBMinH1}
	\begin{split}
		e_{j,1}(t) &= c_{1,j} \Big[ \sinh\Big( \sqrt{\frac{2}{\theta_j+1}} \,t\Big) 
		- \sqrt{\frac{\theta_j-1}{\theta_j+1}} \sin\Big( \sqrt{\frac{2}{\theta_j-1}} \,t \Big) \Big]
		+ c_{2,j}\Big[ \cosh\Big( \sqrt{\frac{2}{\theta_j+1}}\, t\Big) -\cos\Big( \sqrt{\frac{2}{\theta_j-1}}\, t\Big) \Big]
	\end{split}
\end{align}
by choosing $c_{1,j},c_{2,j}\in \R$ such that \eqref{Eq:KLforIntegratedBM_H1_CoeffRatio} and $\|e_{j,1}\|_{H^1} = 1$ are satisfied.\\

\emph{Evaluation}\quad The eigenvalues $\lambda_{j,0}$ and $\lambda_{j,1}$ are computed for $j=1,\ldots,27$ by solving the respective eigenvalue equations \eqref{Eq:EigenEquation} 
and \eqref{Eq:EigenEquation_Sobolev} numerically. 
Together with the eigenvalues $\lambda_{j,0}^{(1)} = (j-1/2)^{-2}\pi^{-2}$ corresponding to the KL-expansion of Brownian motion (i.e., eigenvalues of 
$T_{\partial^{[1,1]}\gamma_X,0}$ with $\partial^{[1,1]}\gamma_x(s,t)=\min(s,t)$) they are plotted in Figure \ref{Fig:EigenvaluesPlot} and also presented in Table 
\ref{Tab:Eigenvalues}.
Notably, the $H^1$-eigenvalues $\lambda_{j,1}$ are almost identical to those of Brownian motion 
$\lambda_{j,0}^{(1)} \sim (j\pi)^{-2}$ for $j=2,\ldots,21$, before their rate of decay seems to catch up with the $L^2$-eigenvalues
$\lambda_{j,0} \sim (j\pi)^{-4}$.

\begin{figure}[H]
	\raggedleft
	\begin{minipage}{.5\textwidth}
		\raggedright
		\includegraphics[width=8cm,height=5cm]{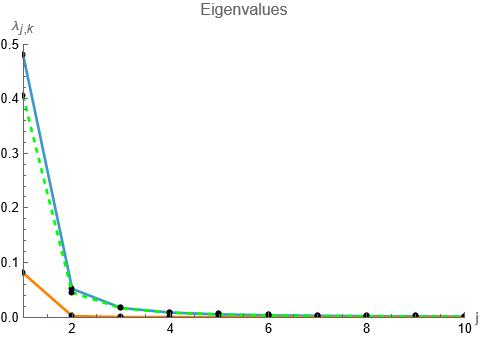}
	\end{minipage}%
	\begin{minipage}{.5\textwidth}
		\raggedright
		\includegraphics[width=8cm,height=5cm]{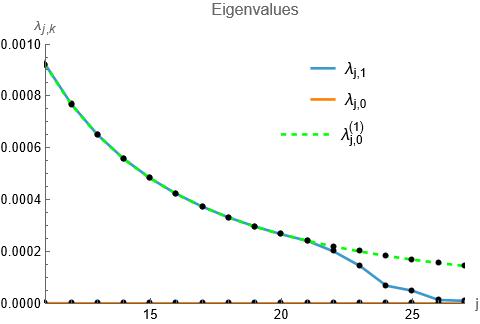}
	\end{minipage}%
	\caption{Comparison of the decrease of the eigenvalues $\lambda_{j,1}$, $\lambda_{j,0}$ and $\lambda_{j,0}^{(1)}$}
	\label{Fig:EigenvaluesPlot}
\end{figure}

The coefficients $(c_{j,1}, c_{j,2})$ of the respective eigenfunctions \eqref{Eq:EigenfunctionIntBMinL2} and \eqref{Eq:EigenfunctionIntBMinH1}
are also computed numerically and presented in Table \ref{Tab:Coefficients}. 
Their respective ratios \eqref{Eq:KLforIntegratedBM_L2_CoeffRatio} and \eqref{Eq:KLforIntegratedBM_H1_CoeffRatio} as functions of $\lambda$ are
plotted in Figure \ref{Fig:CoefficientsRatio}.
For the $L^2$-eigenfunctions $e_{j,0}$ the coefficients converge towards $(-1,1)$ as their ratio \eqref{Eq:KLforIntegratedBM_L2_CoeffRatio} approaches 
negative one when $\lambda$ goes to zero. In contrast, for the $H^1$-eigenfunctions $e_{j,1}$ the coefficients seem to decrease only slowly and their 
limit remains infeasible this way because their ratio begins to oscillate between $+\infty$ and $-\infty$ as $\lambda$ goes to zero.

\begin{figure}[H]
	\raggedleft
	\begin{minipage}{.5\textwidth}
		\raggedright
		\includegraphics[width=8cm,height=5cm]{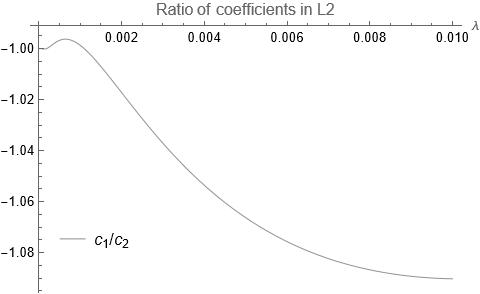}
	\end{minipage}%
	\begin{minipage}{.5\textwidth}
		\raggedright
		\includegraphics[width=8cm,height=5cm]{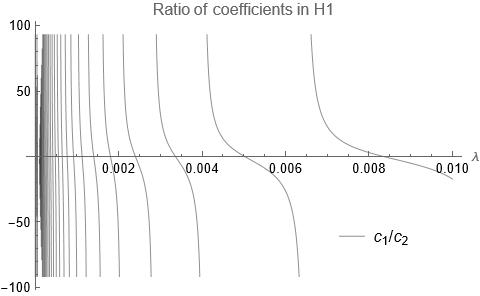}
	\end{minipage}%
	\caption{Ratios \eqref{Eq:KLforIntegratedBM_L2_CoeffRatio} and \eqref{Eq:KLforIntegratedBM_H1_CoeffRatio} of coefficients in the eigenfunctions}
	\label{Fig:CoefficientsRatio}
\end{figure}

The $L^2$- and $H^1$-eigenfunctions $e_{j,0}$ and $e_{j,1}$ themselves are compared in Figure \ref{Fig:Eigenfunctions}: 
the behaviour of the $H^1$-eigenvalues $\lambda_{j,1}$ is closely matched by that of the derivatives $e_{j,1}^{[1]}$ of the corresponding 
$H^1$-eigenfunctions. 
For $j=1,\ldots,21$, these derivatives almost coincide with 
the eigenfunctions of Brownian motion $u_{j,0}(t) = \sqrt{2}\,\sin((j-1/2)\pi \,t)$ 
(i.e., the eigenfunctions of $T_{\partial^{[1,1]}\gamma_X,0}$ with $\partial^{[1,1]}\gamma_X=\min(s,t)$), 
before their frequencies begin to diverge. 
The higher frequency of $e_{j,1}^{[1]}$ from $j=22$ onwards appears to compensate the increased decay of the corresponding eigenvalues $\lambda_{j,1}$. 
This suggests that higher-order expansions tend to emphasize approximating 
the potentially more complex derivatives first to a suitable degree to ensure the optimality conditions \eqref{Eq:KernelExpansionUniformConv} and \eqref{Eq:SobolevApproxError}.
Moreover, the amplitude of the $H^1$-eigenfunctions $e_{j,1}$ decreases inherently, 
and multiplication by the 
$H^1$-eigenvalues $\lambda_{j,1}$ primarily serves to dampen the derivatives $e_{j,1}^{[1]}$. 
Conversely, the amplitude of the $L^2$-eigenfunctions $e_{j,0}$ is bounded but does not decrease inherently, whereas the amplitude of the derivatives $e_{j,0}^{[1]}$ 
explodes. 
\begin{figure}[H]
	\raggedleft
	\begin{minipage}{.33\textwidth}
		\raggedright
		\includegraphics[width=5.5cm,height=3cm]{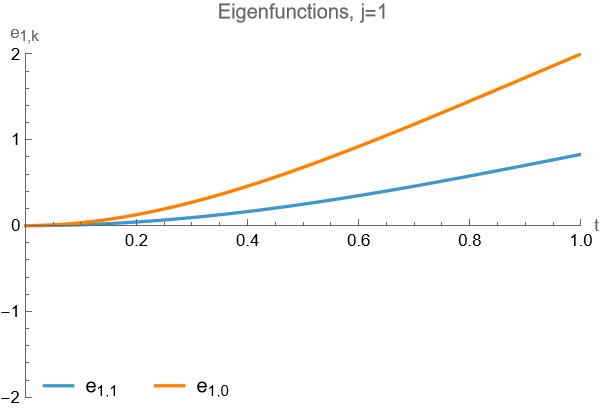}
	\end{minipage}%
	\begin{minipage}{.33\textwidth}
		\raggedright
		\includegraphics[width=5.5cm,height=3cm]{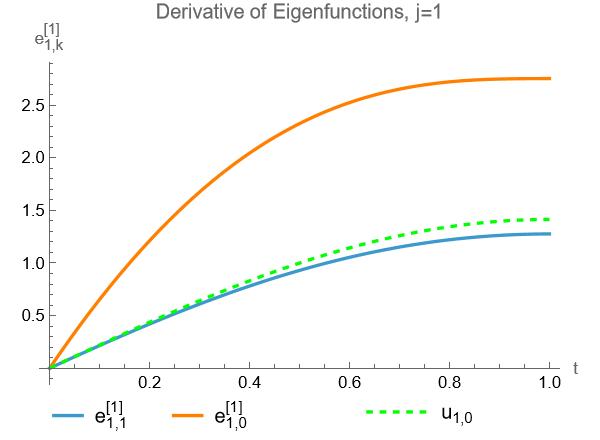}
	\end{minipage}%
	\begin{minipage}{.33\textwidth}
		\raggedright
		\includegraphics[width=5.5cm,height=3cm]{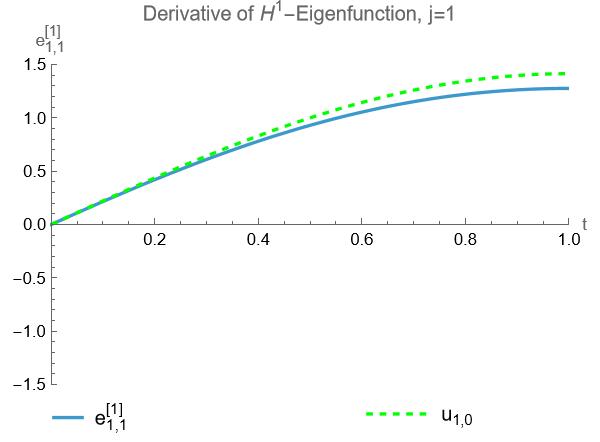}
	\end{minipage}%
\end{figure}

\begin{figure}[H]
	\raggedleft
	\begin{minipage}{.33\textwidth}
		\raggedright
		\includegraphics[width=5.5cm,height=3cm]{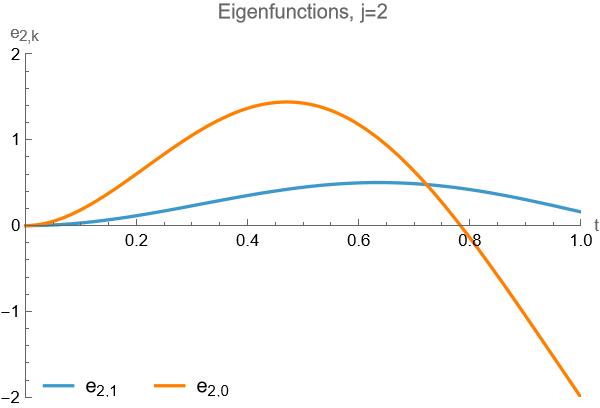}
	\end{minipage}%
	\begin{minipage}{.33\textwidth}
		\raggedright
		\includegraphics[width=5.5cm,height=3cm]{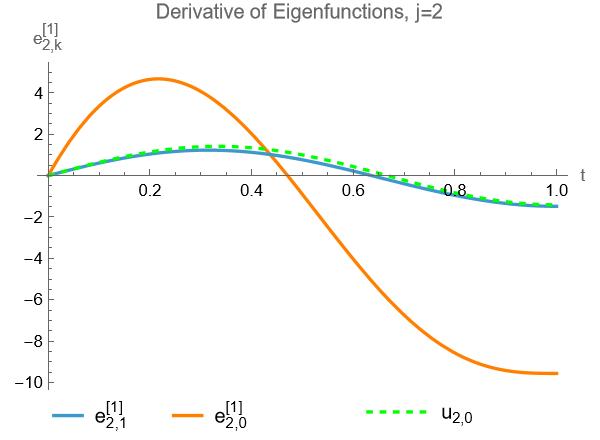}
	\end{minipage}%
	\begin{minipage}{.33\textwidth}
		\raggedright
		\includegraphics[width=5.5cm,height=3cm]{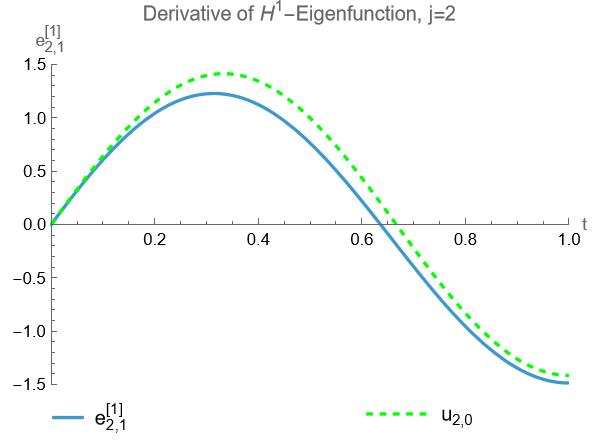}
	\end{minipage}%
\end{figure}

\begin{figure}[H]
	\raggedleft
	\begin{minipage}{.33\textwidth}
		\raggedright
		\includegraphics[width=5.5cm,height=3cm]{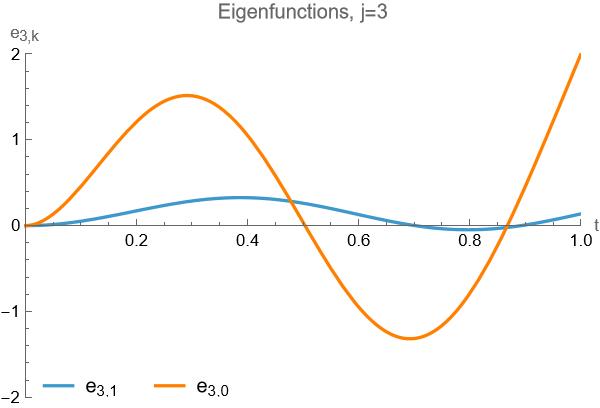}
	\end{minipage}%
	\begin{minipage}{.33\textwidth}
		\raggedright
		\includegraphics[width=5.5cm,height=3cm]{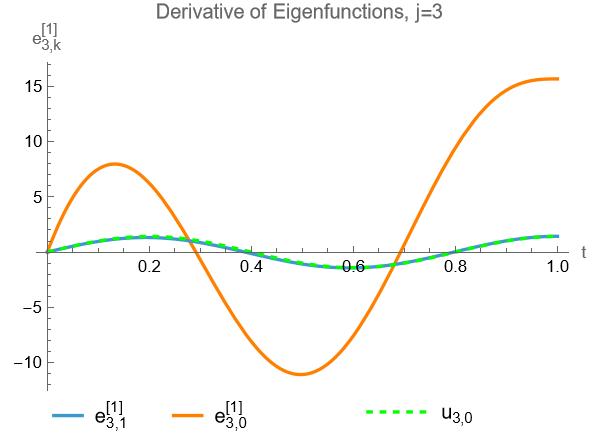}
	\end{minipage}%
	\begin{minipage}{.33\textwidth}
		\raggedright
		\includegraphics[width=5.5cm,height=3cm]{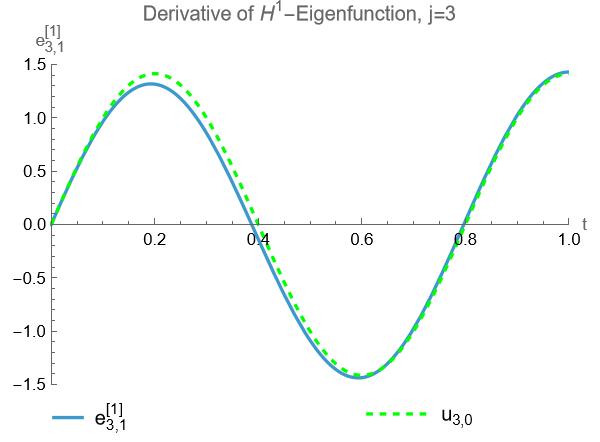}
	\end{minipage}%
\end{figure}

\begin{figure}[H]
	\raggedleft
	\begin{minipage}{.33\textwidth}
		\raggedright
		\includegraphics[width=5.5cm,height=3cm]{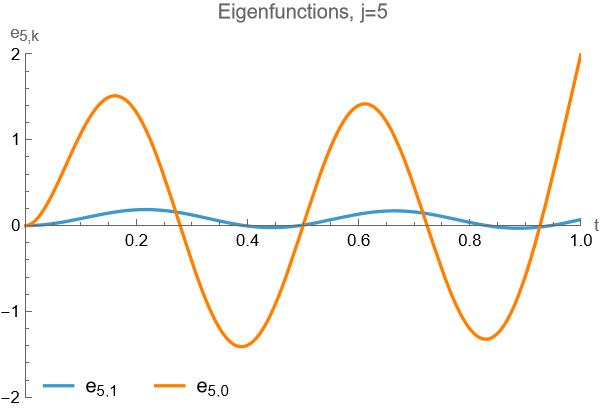}
	\end{minipage}%
	\begin{minipage}{.33\textwidth}
		\raggedright
		\includegraphics[width=5.5cm,height=3cm]{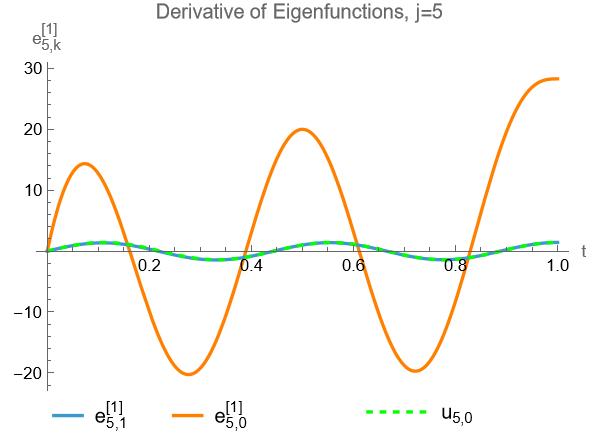}
	\end{minipage}%
	\begin{minipage}{.33\textwidth}
		\raggedright
		\includegraphics[width=5.5cm,height=3cm]{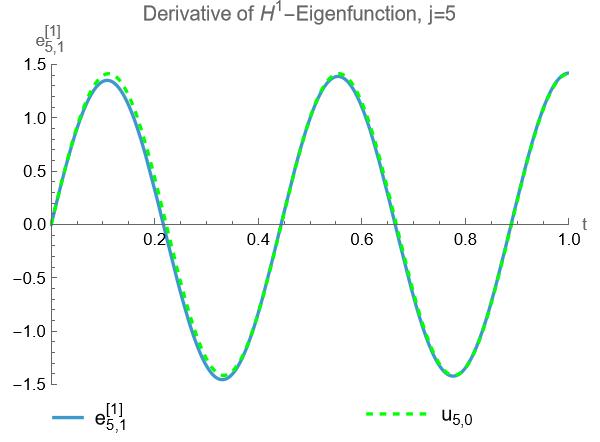}
	\end{minipage}%
\end{figure}

\begin{figure}[H]
	\raggedleft
	\begin{minipage}{.33\textwidth}
		\raggedright
		\includegraphics[width=5.5cm,height=3cm]{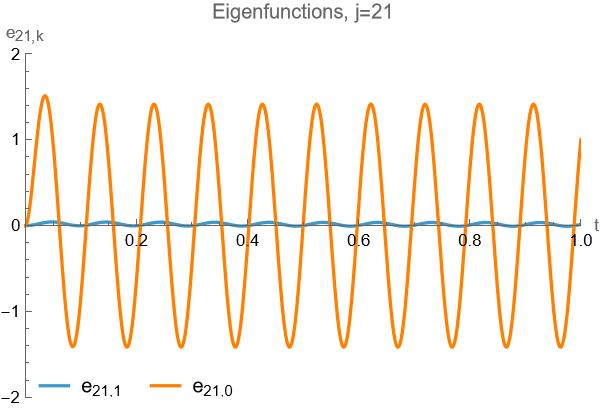}
	\end{minipage}%
	\begin{minipage}{.33\textwidth}
		\raggedright
		\includegraphics[width=5.5cm,height=3cm]{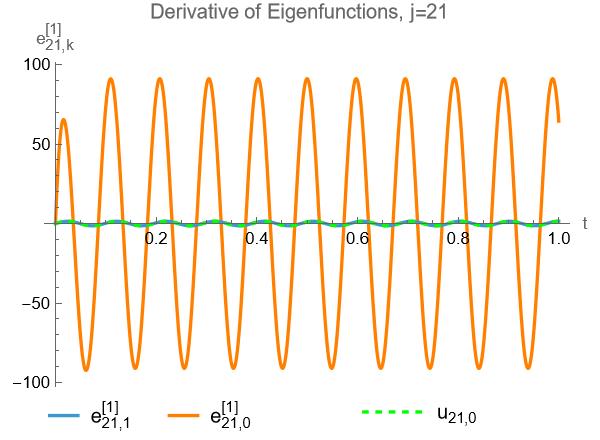}
	\end{minipage}%
	\begin{minipage}{.33\textwidth}
		\raggedright
		\includegraphics[width=5.5cm,height=3cm]{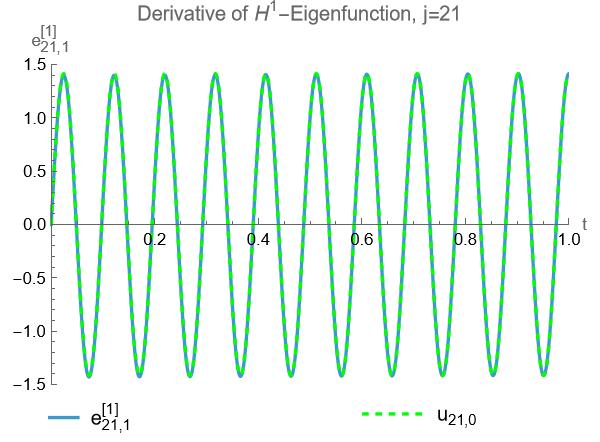}
	\end{minipage}%
\end{figure}

\begin{figure}[H]
	\raggedleft
	\begin{minipage}{.33\textwidth}
		\raggedright
		\includegraphics[width=5.5cm,height=3cm]{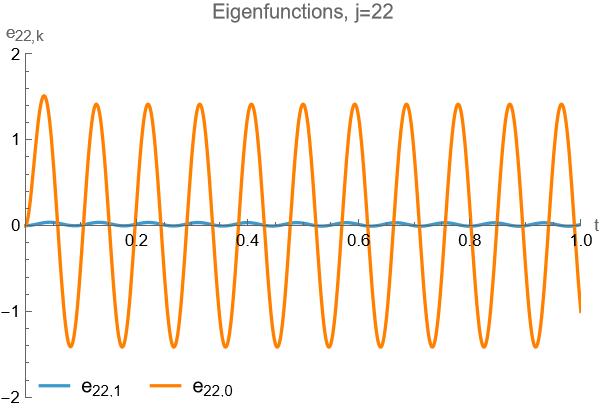}
	\end{minipage}%
	\begin{minipage}{.33\textwidth}
		\raggedright
		\includegraphics[width=5.5cm,height=3cm]{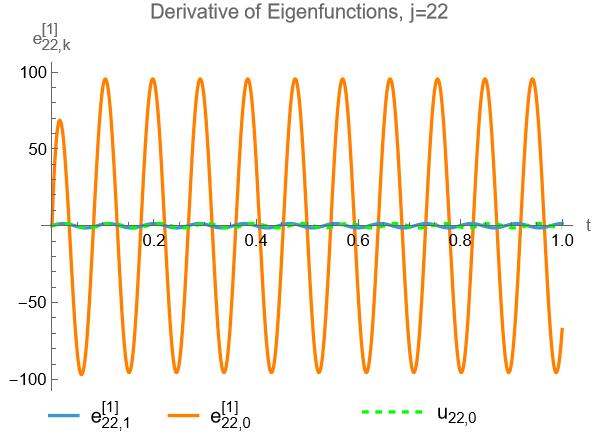}
	\end{minipage}%
	\begin{minipage}{.33\textwidth}
		\raggedright
		\includegraphics[width=5.5cm,height=3cm]{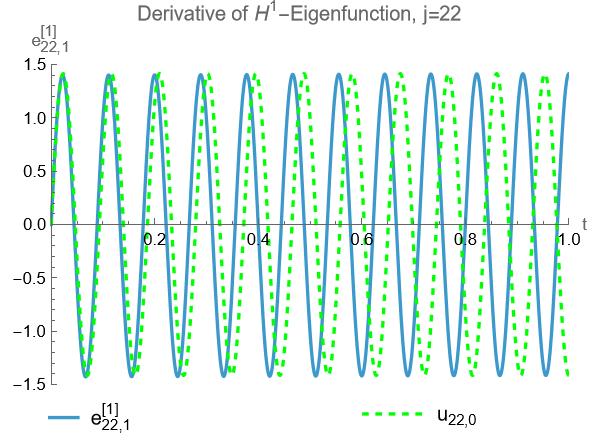}
	\end{minipage}%
\end{figure}

\begin{figure}[H]
	\raggedleft
	\begin{minipage}{.33\textwidth}
		\raggedright
		\includegraphics[width=5.5cm,height=3cm]{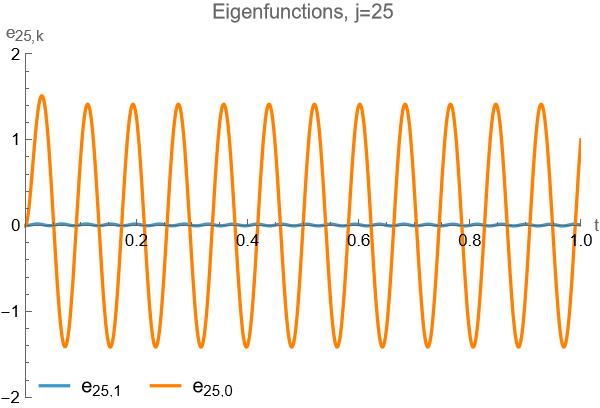}
	\end{minipage}%
	\begin{minipage}{.33\textwidth}
		\raggedright
		\includegraphics[width=5.5cm,height=3cm]{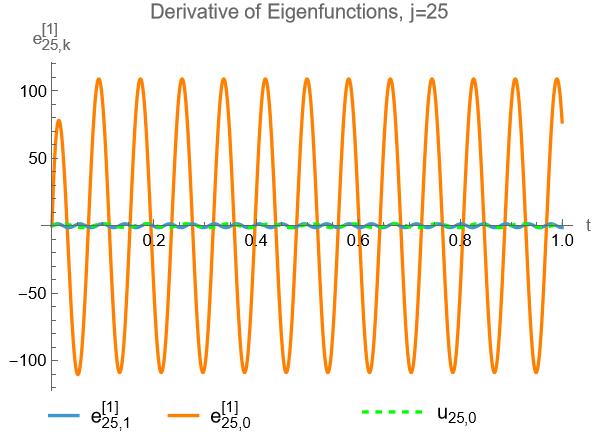}
	\end{minipage}%
	\begin{minipage}{.33\textwidth}
		\raggedright
		\includegraphics[width=5.5cm,height=3cm]{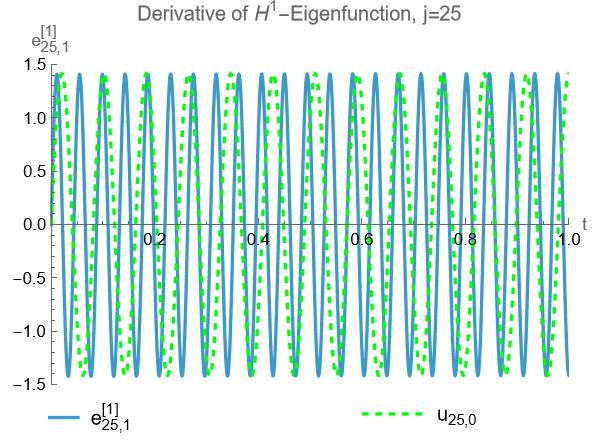}
	\end{minipage}%
	\caption{Comparison of eigenfunctions and their derivatives}
	\label{Fig:Eigenfunctions}
\end{figure}

This difference in eigenfunction behaviour is subsequently reflected in the respective KL- and kernel expansions: 
In order to compare the KL-expansions, we first simulated sample paths of BM via its $L^2$-eigensystem $\{( \lambda_{j,0}^{(1)}, u_{j,0} )\mid j\in\N\}$,
that is, we used $B_{m,0} = \sum_{j\leq m} \lambda_{j,0}^{(1)}\,Z_j\, u_{j,0}$ with $Z_j \sim \NN(0,1)$ i.i.d. and $m=1000$, to model BM,  
and computed the corresponding IBM-sample paths $\widehat{X}_m(t) = \int_0^t B_{m,0}(t)\,dt$ by using a Riemann-approximation. 
We then simulated sample paths of $X_{m,k}$ and its derivative 
using the respective eigensystems $\{( \lambda_{j,k}, e_{j,k} )\mid j\in\N\}$, $k\in\{0,1\}$, that is $X_{m,k}^{[\alpha]} = \sum_{j\leq m} \lambda_{j,k}\, Z_j\, e_{j,k}^{[\alpha]}$ 
for $\alpha = 0,1$ and $m\leq 27$. 
The results are plotted in Figure \ref{Fig:ExpofIBMinH1} and Figure \ref{Fig:KLofIBMinL2}. 
Clearly, the expansion with the $H^1$-eigensystem yields better approximations, not only for the derivative (BM) but also for the process itself (IBM).
This is not in conflict with the mean-square optimality of the classic KL-expansion since $\{ e_{j,1}\mid j\in\N \}$ is not an $L^2$-ONB.

\begin{figure}[H]
	\raggedleft
	\begin{minipage}{.5\textwidth}
		\raggedright
		\includegraphics[width=8cm,height=4cm]{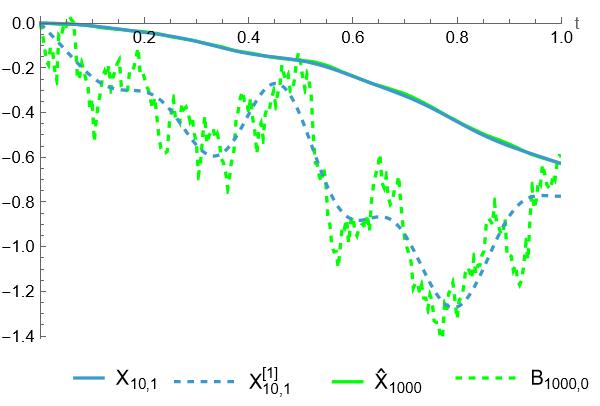}
	\end{minipage}%
	\begin{minipage}{.5\textwidth}
		\raggedright
		\includegraphics[width=8cm,height=4cm]{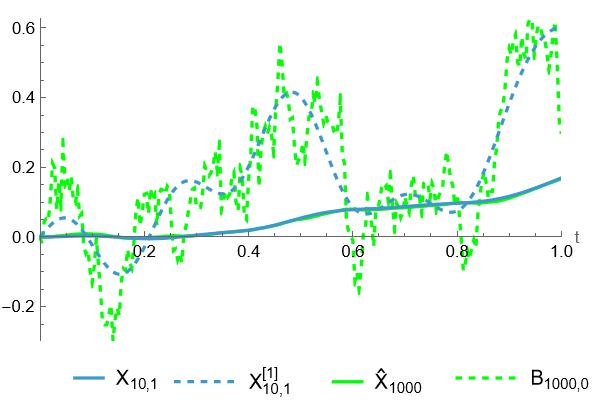}
	\end{minipage}%
\end{figure}
\begin{figure}[H]
	\raggedleft
	\begin{minipage}{.5\textwidth}
		\raggedright
		\includegraphics[width=8cm,height=4cm]{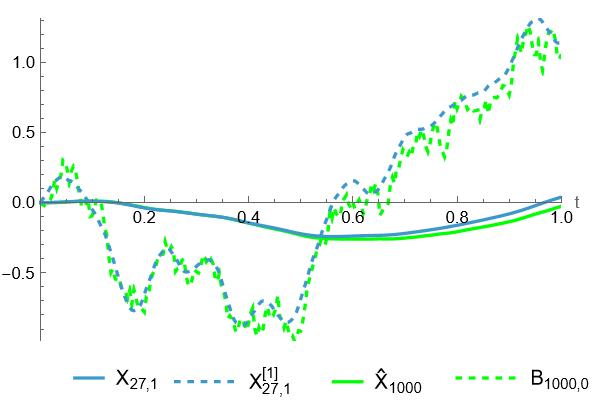}
	\end{minipage}%
	\begin{minipage}{.5\textwidth}
		\raggedright
		\includegraphics[width=8cm,height=4cm]{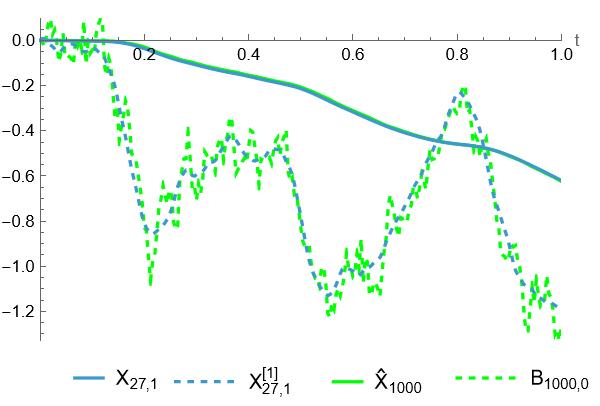}
	\end{minipage}%
	\caption{Sample paths of $H^1$-expansion and corresponding derivative}
	\label{Fig:ExpofIBMinH1}
\end{figure}

\begin{figure}[H]
	\raggedleft
	\begin{minipage}{.5\textwidth}
		\raggedright
		\includegraphics[width=8cm,height=4cm]{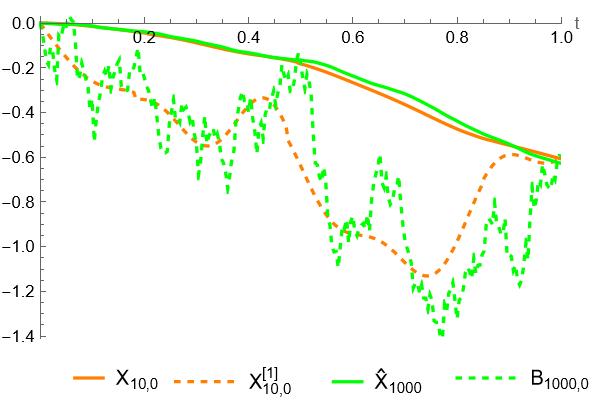}
	\end{minipage}%
	\begin{minipage}{.5\textwidth}
		\raggedright
		\includegraphics[width=8cm,height=4cm]{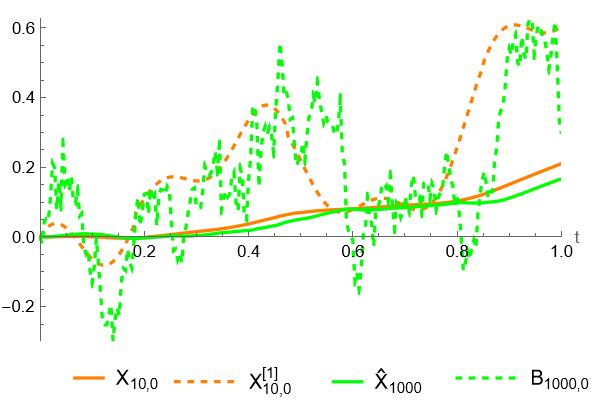}
	\end{minipage}%
\end{figure}
\begin{figure}[H]
	\raggedleft
	\begin{minipage}{.5\textwidth}
		\raggedright
		\includegraphics[width=8cm,height=4cm]{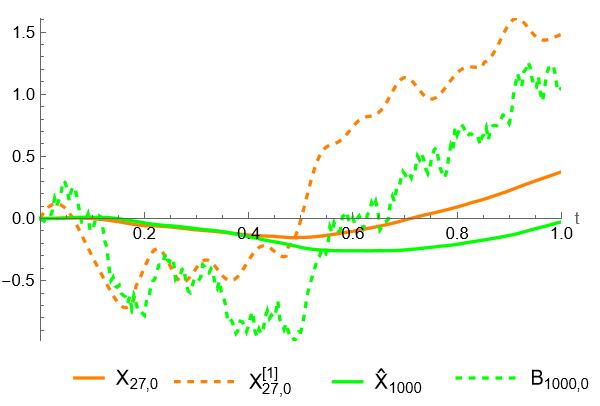}
	\end{minipage}%
	\begin{minipage}{.5\textwidth}
		\raggedright
		\includegraphics[width=8cm,height=4cm]{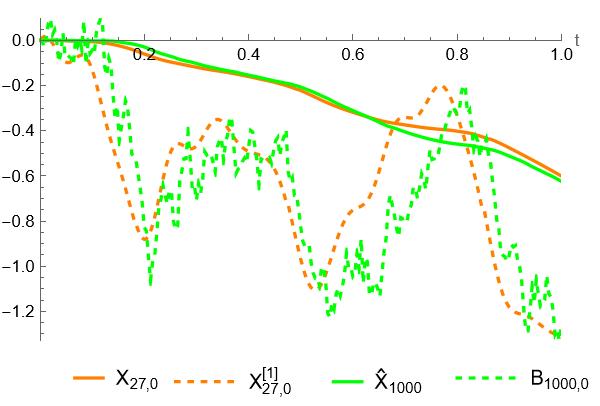}
	\end{minipage}%
	\caption{Sample paths of $L^2$-expansion and corresponding derivative}
	\label{Fig:KLofIBMinL2}
\end{figure}

Regarding the kernel \eqref{Eq:IntegratedBM_kernel} and its derivatives (see Figure \ref{Fig:CovarianceKernel}) 
the respective $H^1$- and $L^2$-expansions with the eigensystems $\{(\lambda_{j,k}, e_{j,k}) \mid j\in\N\}$, $k\in\{0,1\}$, as well as their pointwise error 
(sampled on a grid of resolution $10^{-2}$) are plotted in Figure \ref{Fig:PointwiseErrorH1} and \ref{Fig:PointwiseErrorL2}.
The pointwise error is distributed more 
evenly for the $H^1$-expansion compared to the classic (Mercer) $L^2$-expansion, which exhibits a deteriorating approximation quality near the boundary, 
particularly for the derivatives due to the pathological behaviour of $e_{j,0}^{[1]}$. This is also reflected in the uniform error 
\begin{align*}
	\Delta_{m,k}^{[\alpha,\beta]} = \sup_{x,y \in [0,1]} \Big| \partial^{[\alpha,\beta]} \gamma_X(x,y) 
	- \sum_{j\leq m} \lambda_{j,k}e_{j,k}^{[\alpha]}(x)e_{j,k}^{[\beta]}(y) \Big|, \qquad \alpha,\beta \in\{0,1\},
\end{align*}
which is presented in Table \ref{Tab:UniformError} for different values of $m$. 
\begin{figure}[H]
	\raggedleft
	\begin{minipage}{.5\textwidth}
		\raggedright
		\includegraphics[width=8cm,height=4.5cm]{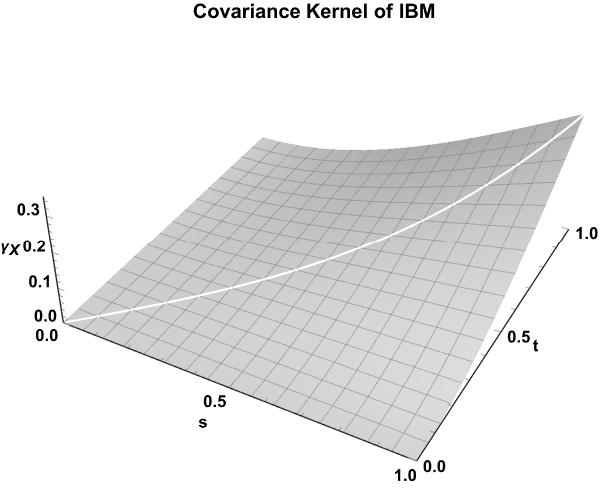}
	\end{minipage}%
	\begin{minipage}{.5\textwidth}
		\raggedright
		\includegraphics[width=8cm,height=4.5cm]{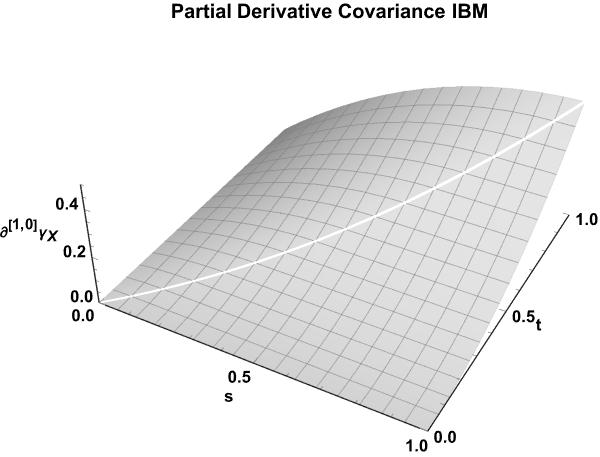}
	\end{minipage}%
\end{figure}
\begin{figure}[H]
	\raggedleft
	\begin{minipage}{.5\textwidth}
		\raggedright
		\includegraphics[width=8cm,height=4.5cm]{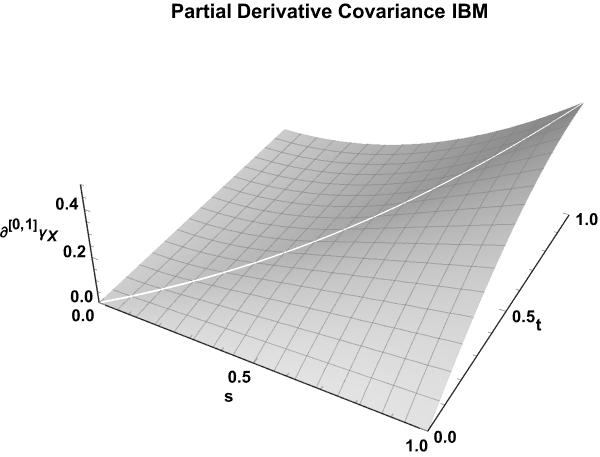}
	\end{minipage}%
	\begin{minipage}{.5\textwidth}
		\raggedright
		\includegraphics[width=8cm,height=4.5cm]{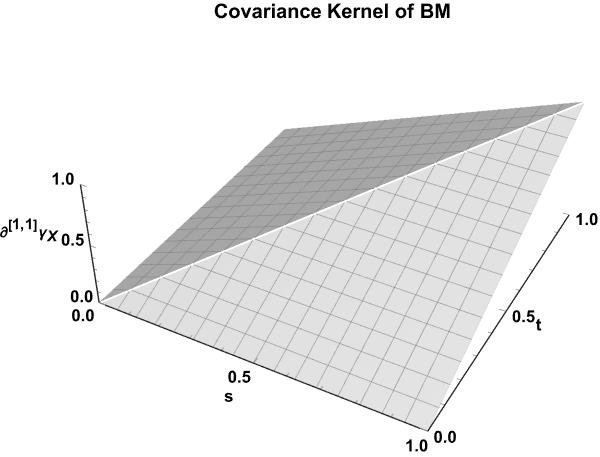}
	\end{minipage}%
	\caption{Covariance Kernel $\gamma_X$ of IBM and its derivatives}
	\label{Fig:CovarianceKernel}
\end{figure}

\begin{table}
\caption{Uniform approximation error}
\centering
\begin{tabular}[t]{lcccccc}
	\toprule
	\multicolumn{1}{c}{$m$} & \multicolumn{2}{c}{$\Delta_{m,k}$} & \multicolumn{2}{c}{$\Delta_{m,k}^{[1,0]}$} & \multicolumn{2}{c}{$\Delta_{m,k}^{[1,1]}$} \\
	\cmidrule(l{3pt}r{3pt}){2-3} \cmidrule(l{3pt}r{3pt}){4-5} \cmidrule(l{3pt}r{3pt}){6-7}
	& $H^1$ & $L^2$ & $H^1$ & $L^2$ & $H^1$ & $L^2$ \\
	\midrule
	3   & $6.491e^{-4}$ & $4.811e^{-4}$ & $4.148e^{-3}$ & $6.979e^{-3}$ & $6.804e^{-2}$ & $1.339e^{-1}$\\
	5   & $1.499e^{-4}$ & $1.074e^{-4}$ & $1.552e^{-3}$ & $2.555e^{-3}$ & $4.063e^{-2}$ & $8.079e^{-2}$\\
	10  & $1.962e^{-5}$ & $1.740e^{-5}$ & $3.985e^{-4}$ & $7.563e^{-4}$ & $2.028e^{-2}$ & $4.386e^{-2}$\\
	20  & $2.530e^{-6}$ & $1.442e^{-5}$ & $9.635e^{-5}$ & $6.357e^{-4}$ & $1.013e^{-2}$ & $3.880e^{-2}$\\
	27  & $1.681e^{-6}$ & $1.417e^{-5}$ & $7.078e^{-5}$ & $6.175e^{-4}$ & $8.680e^{-3}$ & $3.750e^{-2}$\\
	\bottomrule
\end{tabular}
\label{Tab:UniformError}
\end{table}

\begin{figure}[H]
	\raggedleft
	\begin{minipage}{.33\textwidth}
		\raggedright
		\includegraphics[width=5.5cm,height=3.6cm]{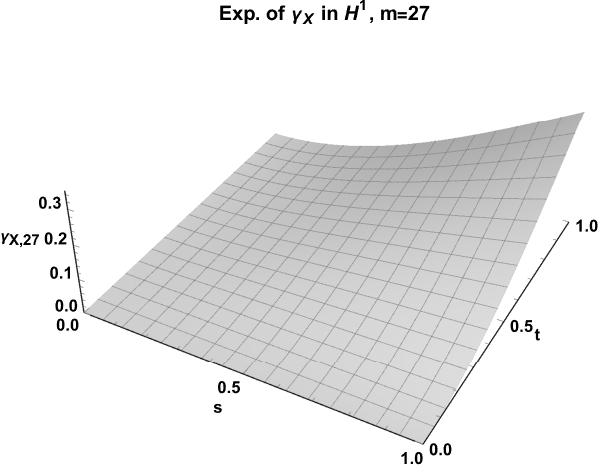}
	\end{minipage}%
	\begin{minipage}{.33\textwidth}
		\raggedright
		\includegraphics[width=5.5cm,height=3.6cm]{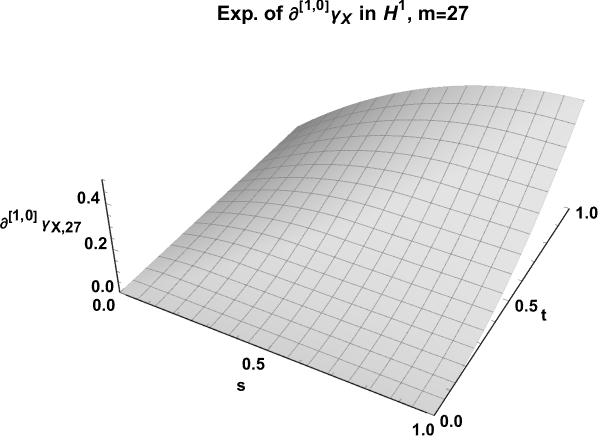}
	\end{minipage}%
	\begin{minipage}{.33\textwidth}
		\raggedright
		\includegraphics[width=5.5cm,height=3.6cm]{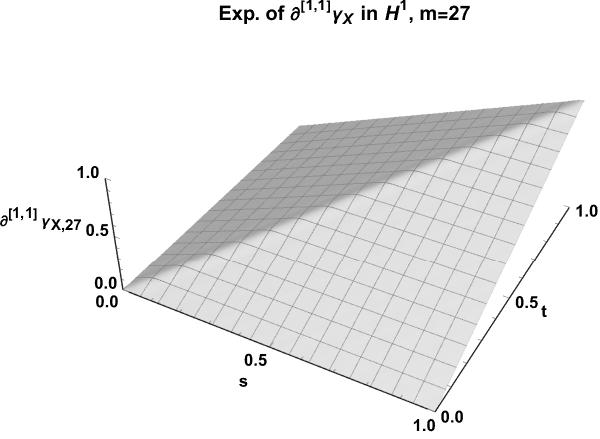}
	\end{minipage}%
\end{figure}

\begin{figure}[H]
	\raggedleft
	\begin{minipage}{.33\textwidth}
		\raggedright
		\includegraphics[width=5.5cm,height=3.6cm]{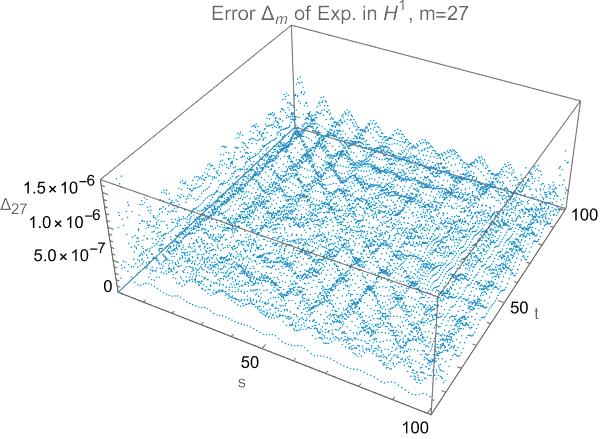}
	\end{minipage}%
	\begin{minipage}{.33\textwidth}
		\raggedright
		\includegraphics[width=5.5cm,height=3.6cm]{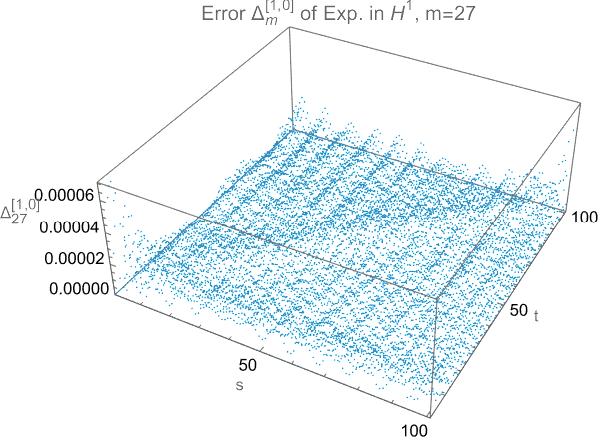}
	\end{minipage}%
	\begin{minipage}{.33\textwidth}
		\raggedright
		\includegraphics[width=5.5cm,height=3.6cm]{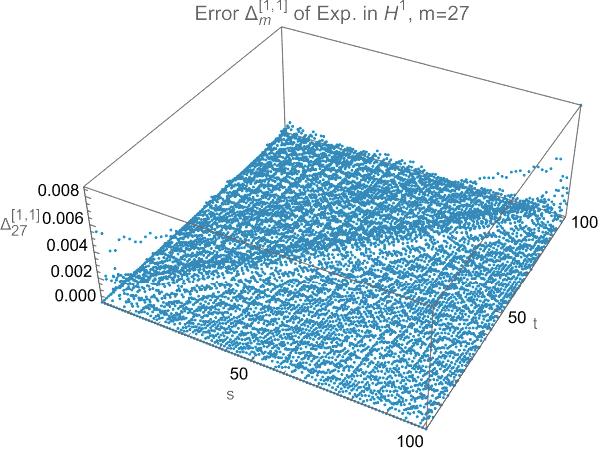}
	\end{minipage}%
	\caption{Kernel expansions with $H^1$-eigenfunctions and pointwise error}
	\label{Fig:PointwiseErrorH1}
\end{figure}

\begin{figure}[H]
	\raggedleft
	\begin{minipage}{.33\textwidth}
		\raggedright
		\includegraphics[width=5.5cm,height=3.6cm]{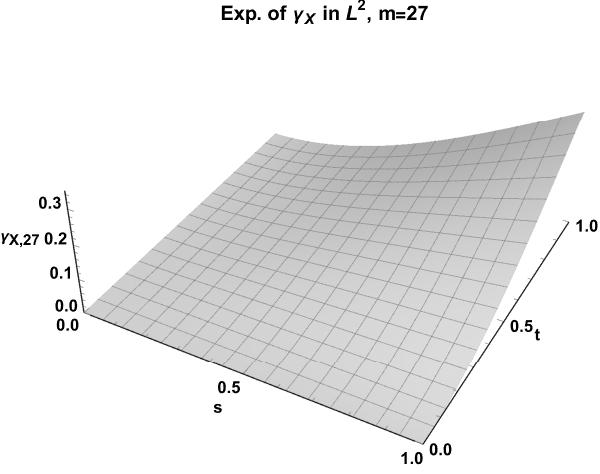}
	\end{minipage}%
	\begin{minipage}{.33\textwidth}
		\raggedright
		\includegraphics[width=5.5cm,height=3.6cm]{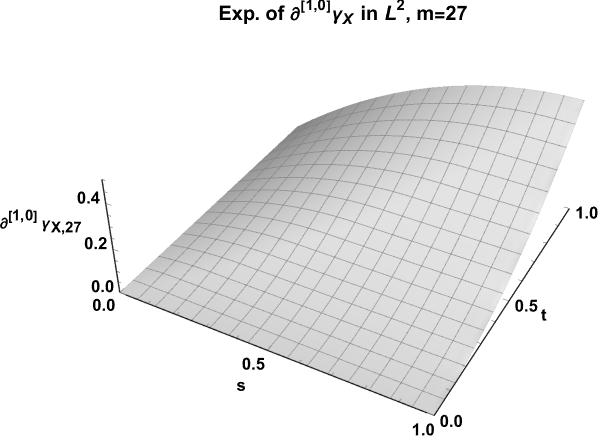}
	\end{minipage}%
	\begin{minipage}{.33\textwidth}
		\raggedright
		\includegraphics[width=5.5cm,height=3.6cm]{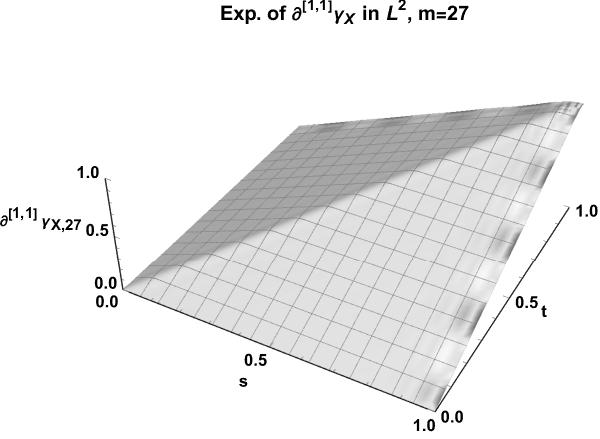}
	\end{minipage}%
\end{figure}

\begin{figure}[H]
	\raggedleft
	\begin{minipage}{.33\textwidth}
		\raggedright
		\includegraphics[width=5.5cm,height=3.6cm]{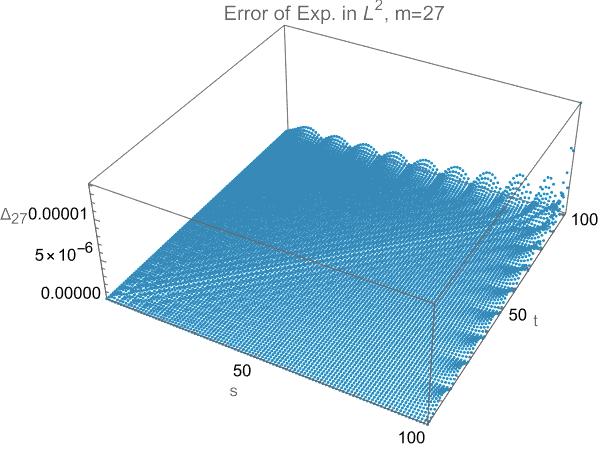}
	\end{minipage}%
	\begin{minipage}{.33\textwidth}
		\raggedright
		\includegraphics[width=5.5cm,height=3.6cm]{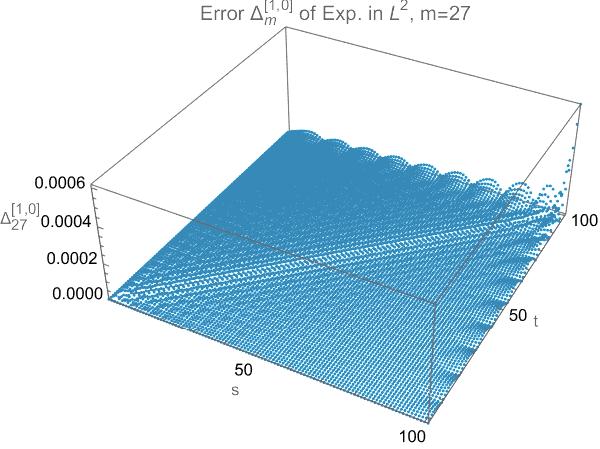}
	\end{minipage}%
	\begin{minipage}{.33\textwidth}
		\raggedright
		\includegraphics[width=5.5cm,height=3.6cm]{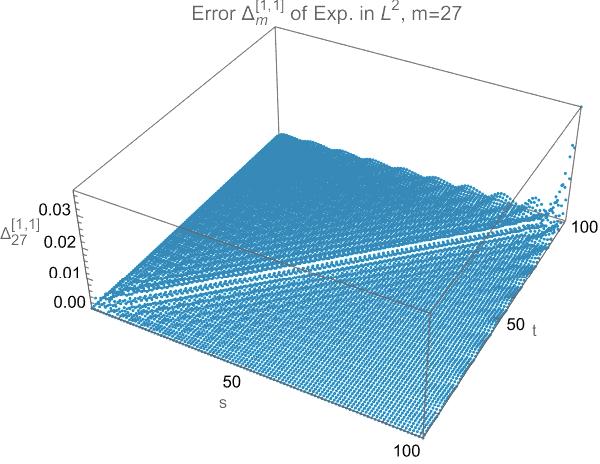}
	\end{minipage}%
	\caption{Kernel expansions with $L^2$-eigenfunctions and pointwise error}
	\label{Fig:PointwiseErrorL2}
\end{figure}
\end{Ex}

\section{Acknowledgement} 
The author is grateful to Jens-Peter Kreiß, Efstathios Paparoditis and Siegfried Hörmann for helpful discussions.\\
\emph{Funding:} This research was funded in whole, or in part, by the Austrian Science Fund (FWF) 
\ead{https://doi.org/10.55776/P35520}{10.55776/P35520}.
For the purpose of open access, the authors have applied a CC BY public copyright licence to any Author Accepted Manuscript version arising 
from this submission.

\newpage
\appendix
\section{Additional results and proofs}
\label{Sec:Appendix_AdditonalResults}

\begin{Lem}\label{Lem:U^kIsAhilbertSpace}
The space $H^{k_\alpha,k_\beta}(\Theta\times\Theta)$ defined by \eqref{Eq:InnerProductMixedOrder} equipped with the inner product \eqref{Eq:InnerProductMixedOrder} 
is a separable Hilbert space. 
\end{Lem}

\begin{proof}
The proof can be taken from the corresponding result for Sobolev spaces $H^k(\Theta)$ (e.g., \cite{Adams1975SobolevSpace}):\\
Let $(h_n)_n$ be a Cauchy sequence in $H^{k_\alpha,k_\beta}(\Theta \times \Theta)$. Then $(\partial^{(\alpha,\beta)}h_n)_n$ is a Cauchy sequence 
in $L^2(\Theta\times \Theta)$ for $|\alpha| \leq k_\alpha$ and $|\beta|\leq k_\beta$. Since $L^2(\Theta\times \Theta)$ is complete, there exists functions 
$h_{\alpha,\beta}$ such that $\partial^{(\alpha,\beta)}h_n \to h_{\alpha,\beta}$ with respect to $\|\cdot\|_{L^2}$. Set $h:= h_{0,0}$ and note that because
$L^2(\Theta\times \Theta) \subset L^1_{loc}(\Theta\times \Theta)$ each $\partial^{(\alpha,\beta)}h_n$ determines a distribution 
$T_{\partial^{(\alpha,\beta)}h_n} \in \D'(\Theta\times\Theta)$ via integration. Thus, for any $\varphi \in C_c^{\infty}(\Theta\times \Theta)$, 
\begin{align*}
	|T_{\partial^{(\alpha,\beta)} h_n}(\varphi) - T_{h_{\alpha,\beta}}(\varphi)|
	&\leq \int_{\Theta} \int_{\Theta} | \partial^{(\alpha,\beta)}h_n(x,y) - h_{\alpha,\beta}(x,y) |\, |\varphi(x,y)| \,dx\,dy\\
	&\leq \| \partial^{(\alpha,\beta)}h_n - h_{\alpha,\beta} \|_{L^2} \, \|\varphi\|_{L^2} \to 0,
\end{align*}
which yields
\begin{align*}
	T_{h_{\alpha,\beta}}(\varphi)
	= \lim_{n\to\infty} T_{\partial^{(\alpha,\beta)} h_n}(\varphi)
	&= \lim_{n\to\infty} (-1)^{|\alpha|+|\beta|} T_{ h_n}( \partial^{[\alpha,\beta]}\varphi)\\
	&= (-1)^{|\alpha|+|\beta|} T_h( \partial^{[\alpha,\beta]}\varphi ).
\end{align*}
Therefore, $h_{\alpha,\beta} = \partial^{(\alpha,\beta)} h$ with respect to $\| \cdot \|_{L^2}$ and, hence, $h \in H^{k_{\alpha,k_{\beta}}}(\Theta\times \Theta)$ 
such that
\begin{align*}
	\|h_n - h\|_{H^{k_{\alpha,k_{\beta}}}}
	= \sum_{|\alpha| \leq k_{\alpha}} \sum_{|\beta| \leq k_{\beta}} \| \partial^{(\alpha,\beta)} h_n - \partial^{(\alpha,\beta)} h \|_{L^2} \to 0.
\end{align*} 
This shows completeness of $H^{k_\alpha,k_\beta}(\Theta\times\Theta)$ and the Hilbert space structure can be established as follows: 
Let $N=\binom{k_{\alpha}+d}{k_{\alpha}}\binom{k_{\beta}+d}{k_\beta}$ and define on $L_N^2 := \prod_{j=1}^N L^2(\Theta\times \Theta)$ the direct sum norm 
$\|h\|_{L_N^2}^2 = \sum_{j=1}^N \|h_j\|_{L^2}^2$. Then $L_N^2$ becomes a separable Hilbert space and to each $h\in H^{k_{\alpha},k_{\beta}}(\Theta \times \Theta)$
corresponds a vector $\J h = (\partial^{(\alpha,\beta)}h;\,|\alpha|\leq k_{\alpha}, |\beta|\leq k_{\beta}) \in L_N^2$. 
Note that $\|\J h\|_{L_N^2} = \|h\|_{U^{k_\alpha,k_\beta}}$, so $\J\colon U^{k_\alpha,k_\beta}(\Theta\times \Theta) \to L_N^2$ is an isometric isomorphism onto a 
subspace $G \subset L_N^2$. Since $H^{k_\alpha,k_\beta}(\Theta\times \Theta)$ is complete, $G=\J(H^{k_\alpha,k_\beta}(\Theta\times \Theta))$ is also complete, 
thus a closed subspace in $L_N^2$ and, hence, a separable Hilbert space. The same conclusion must also hold for 
$H^{k_\alpha,k_\beta}(\Theta\times \Theta) = P^{-1}(G)$.
\end{proof}

\begin{Lem}\label{Lem:UniformConvergence}
Let $k\in \N_0$ and suppose $h\in C^{k,k}(\ol\Theta\times\ol\Theta)$ is symmetric. Furthermore, let $\{ (\lambda_{j,k},e_{j,k})\mid j\in\N \}$ denote the eigensystem 
of the associated kernel operator $T_h \in HS(H^k)$ of order $k$ defined via \eqref{Eq:AssociatedOperator}.
Then, for $f\in C^k(\ol{\Theta})$ and $|\alpha|\leq k$, 
\begin{align*}
	\lim_{m\to \infty} \sup_{x\in \ol\Theta} \Big| [T_hf]^{[\alpha]}(x) - \sum_{j\leq m} \lambda_{j,k} \< f,e_{j,k} \>_{H^k} e_{j,k}^{[\alpha]}(x) \Big| = 0.
\end{align*}
\end{Lem}

\begin{proof}
It follows from Step $1$ in the proof of Theorem \ref{Thrm:Mercer} that $T_h$ is a compact operator on $H_0^k(\Theta)$ and $e_j \in C^k(\ol\Theta)$. 
Moreover, due to \eqref{Eq:SpectralRepresentation}, 
the proclaimed equality already holds with respect to $L^2$-convergence. Thus, it suffices to establish uniform convergence of the series 
$\sum_{j\leq m} |\lambda_j \< f,e_j \>_{H^k} e_j^{[\alpha]}(\cdot)|$. To that end, note that by \eqref{Eq:DerivativeT_hfAsInnerProduct}, 
\eqref{Eq:DerivativeT_hfEigenfunction} and Bessel's inequality, for all $x \in \ol\Theta$,
\begin{align*}
	\sum_{j} |\lambda_j \, e_j^{[\alpha]}(x)|^2
	&= \sum_{j} |[T_h e_j]^{[\alpha]} (x)|^2 \\
	&= \sum_{j} | \< \partial^{[\alpha,0]}h(x, \cdot), e_j \>_{H^k} |^2\\
	&\leq \| \partial^{[\alpha,0]} h(x,\cdot) \|_{H^k}^2\\
	&= \sum_{|\beta|\leq k} \int_\Theta |\partial^{[\alpha,\beta]} h(x,y) |^2 \,dy\\
	&\leq \sum_{|\beta|\leq k} \| \partial^{[\alpha,\beta]}h \|_{\infty}^2 \, \mathrm{vol}(\Theta).
\end{align*}
Now, let $\epsilon > 0$ and choose $N=N(\epsilon)\in \N$ such that $\sum_{j\geq N} |\< f,e_j \>_{H^k}|^2 \leq \epsilon^2$. Then, by Cauchy-Schwarz inequality,
\begin{align*}
	\sum_{j\geq N} |\lambda_j \< f,e_j \>_{H^k} e_j^{[\alpha]}(x)| 
	&\leq \Big( \sum_{j\geq N} |\< f,e_j \>_{H^k}|^2  \Big)^{1/2} \Big( \sum_{j\geq N} |\lambda_j\,e_j^{[\alpha]}(x)|^2  \Big)^{1/2}\\
	&\leq \epsilon \,\Big(\sum_{|\beta|\leq k} \| \partial^{[\alpha,\beta]}h \|_{\infty}^2\Big)^{1/2}\sqrt{\mathrm{vol}(\Theta)},
\end{align*}
which yields the required uniform convergence.
\end{proof} 

\begin{Lem}\label{Lem:TestfunctionsAreSeparating2}
Let $S:= \{ \<\cdot, \varphi\>_{L^2} \circ \partial^{(\alpha)} \mid \varphi \in C_c^{\infty}(\Theta), |\alpha|\leq k \}$, 
where $\partial^{(\alpha)}\colon H^k(\Theta) \to L^2(\Theta)$,
$f\mapsto \partial^{(\alpha)}f = f^{(\alpha)}$. Then $S$ is a separating subset for $H^k(\Theta)$.
\end{Lem}

\begin{proof}
Let $f,g \in H^k(\Theta)$ such that $\| f-g \|_{H^k} \not= 0$. Suppose for all $\varphi \in C_c^{\infty}(\Theta)$ and $|\alpha|\leq k$, 
\begin{align*}
	\< f^{(\alpha)}-g^{(\alpha)},\varphi \>_{L^2} = \int_\Theta ( f^{(\alpha)}(x) - g^{(\alpha)}(x) ) \varphi(x) \,dx = 0.
\end{align*}
Then $f^{(\alpha)} = g^{(\alpha)}$ almost everywhere in $\Theta$ for all $|\alpha| \leq k$ by Lemma 3.26 in \cite{Adams1975SobolevSpace}, which contradicts 
$\| f-g \|_{H^k} \not= 0$. Hence, there has to exist a $\varphi \in C_c^{\infty}(\Theta)$ and a $|\alpha| \leq k$ such that 
$\<f^{(\alpha)},\varphi\>_{L^2} \not= \< g^{(\alpha)},\varphi \>_{L^2}$.
\end{proof}

The following sufficient condition for Gaussianity is well-known. However, due to a lack of reference a functional analytic proof is presented. 
Alternatively, a probabilistic proof may also be given in terms of characteristic functionals. 

\begin{Lem}\label{Lem:GaussianInBanchSpace}
Let $E$ be a separable Banach space and $E_0'\subseteq E_0$ a separating subspace. 
Suppose $X$ is an $E$-valued random element such that $\varphi(X)$ is normally distributed for all 
$\varphi \in E_0'$. Then $\varphi(X)$ is normally distributed for all $\varphi \in E'$, i.e. $X$ is Gaussian.
\end{Lem}

\begin{proof}
Define $G:= \{ \varphi \in E'\mid \varphi(X) ~\text{is normally distr.} \}$ and note that $E_0' \subseteq G \subseteq E'$. Therefore, $G$ is also a separating 
subspace. Moreover,\\
(1.) $G$ is weak*-dense: Indeed, suppose $G$ is not weak*-dense. Then there would exist a $\varphi^* \in E'\backslash \ol G$, where $\ol G$ denotes 
the weak*-closure of $G$.
Since the dual space of $E'$ with respect to the weak*-topology is again $E$, this in turn would imply existence of $u\in E$ by the Hahn-Banach separation 
theorem (e.g., \cite{Werner2018Funktionalanalysis}, Corollary III.1.8) such that $\varphi(u)=0$ for all $\varphi\in \ol G $ but $\varphi^*(u) \not=0$. 
But this contradicts the separation property of $G$ 
and thus $\ol G = E'$.\\
(2.) $G$ is weak*-sequentially closed: For any sequence $(\varphi_n)_n$ in $G$ that converges weak* to some $\varphi\in E'$, that is 
$\lim_{n\to\infty}\varphi_n(u) = \varphi(u)$ for all $u\in E$, it follows that $\varphi_n(X(\omega)) = \varphi(X(\omega))$ for all $\omega \in \Omega$. 
Since $\varphi_n(X)$ is a sequence of Gaussian's and $\varphi_n(X) \to \varphi(X)$ (almost) surely, the limit $\varphi(X)$ must also be Gaussian, 
hence $\varphi \in G$.\\
As a consequence of the Krein-Smulian theorem (e.g., \cite{Hytonen2016AnalysisInBanachSpaces}, Corollary B.1.14), these properties of $G$ yield $G = E'$.
\end{proof}

\begin{proof}[Proof of Lemma \ref{Lem:GaussianElements}]
Suppose $X$ is a Gaussian $H^k(\Theta)$-valued or $H_{\kappa}(\Theta)$-valued random element. Then we have 
$\E\|X\|_{H^k}^p = \E\| X \|_{L_{\kappa}^2}^p < \infty$ for all $p\geq 0$ due to Fernique's celebrated theorem (e.g., \cite{Bogachev1998GaussianMeasures}, 
Theorem 2.8.5). This implies the expectation $\E[X] \in H^k(\Theta)$ and the covariance operator $\Cov_k(X)\in N(H^k)$, respectively,
$\Cov(\J X) \in N(H_{\kappa})$ are well-defined.\\
(i) $\Longleftrightarrow$ (ii): If $X$ is a Gaussian $H^k(\Theta)$-valued random element, then $\<X,f\>_{H^k}$ is normal distributed for all 
$f \in H^k(\Theta)$ such that $\E\<X,f\>_{H^k} = \< \E[X], f\>_{H^k}$ and $\Var\<X,f\>_{H^k} = \< \Cov_k(X)f,f \>_{H^k}$. 
On the other hand, suppose $\< X,f \>_{H^k} \sim \NN( \< \mu, f \>_{H^k}, \< Tf,f \>_{H^k} )$ for all $f\in H^k$, where $\mu \in H^k(\Theta)$ and $T \in N(H^k)$
is self-adjoint and non-negative.
Then $X$ is a Gaussian $H^k(\Theta)$-valued random element and
$\E[X]=\mu$, because $\<\E[X],f\>_{H^k} = \E\<X,f\>_{H^k} =\<\mu, f\>_{H^k}$ for all $f\in H^k(\Theta)$.
Moreover, 
\begin{align*}
	\< \Cov_k(X)f,f \>_{H^k} 
	= \E \< X-\mu, f \>_{H^k}^2
	= \Var\< X,f \>_{H^k}  
	= \< Tf,f \>_{H^k}
\end{align*}
for all $f\in H^k(\Theta)$ and since $\Cov_k(X)$ and $T$ are self-adjoint, this implies $\<\Cov_k(X)f,g\>_{H^k} = \<Tf,g\>_{H^k}$ for all $f,g\in H^k(\Theta)$. 
In particular, for an orthonormal basis $\{ e_j\mid j\in\N \}$ of $H^k(\Theta)$,
\begin{align*}
	\| \Cov_k(X) \|_{N(H^k)} 
	= \sum_j \< \Cov_k(X)e_j,e_j \>_{H^k}
	= \sum_j \< Te_j,e_j \>_{H^k}
	= \|T \|_{N(H^k)}.
\end{align*}
Thus, it follows from Theorem 2.21 in \cite{Simon2005TraceIdeals} that $\|T-\Cov_k(X)\|_{N(H^k)}=0$.\\
(i) $\Longrightarrow$ (iii): It follows from Remark \ref{Rem:ExAndCovInH_kappa} that $\J X=(X^{(\alpha)})_{|\alpha|\leq k}$ is a $H_{\kappa}(\Theta)$-valued random element.
Moreover, $\< \J X, \J f \>_{L_{\kappa}^2} = \<X,f\>_{H^k}$ is normal distributed for any $f \in H^k(\Theta)$ and, hence, $\J X$ is Gaussian. The formulas for the 
expectation and the covariance operator follow directly from Remark \ref{Rem:ExAndCovInH_kappa}.\\
(iii) $\Longrightarrow$ (i): Since $(X^{(\alpha)})_{|\alpha|\leq k}=\J X$ is a $H_{\kappa}(\Theta)$-valued random element and 
$\J ^*\colon H_{\kappa}(\Theta) \to H^k(\Theta)$ is $\B(H_{\kappa})$-$\B(H^k)$-measurable, $X=\J^*\J X$ is $\A$-$\B(H^k)$-measurable.
Moreover, $\< X,f \>_{H^k} = \< \J X,\J f \>_{L^2_\kappa}$ is normally distributed for any $f\in H^k(\Theta)$ and, hence, $X$ is a Gaussian $H^k(\Theta)$-valued 
random element. 
Finally, $\E[\J X] = (\E[X]^{(\alpha)})_{|\alpha|\leq k}$ by assumption, so $\E[X] = \J^*\E[\J X] = \E[X]^{(0)}$. Similarly, for any $f\in H^k(\Theta)$, 
almost everywhere
\begin{align*}
	[\Cov_k(X)f](t)
	&= [\J^*\,\Cov(\J X)\J f](t)
	= \sum_{|\alpha|\leq k} \int_\Theta \partial^{(\alpha,0)}c_X(s,t)\,f^{(\alpha)}(s) \,ds. \qedhere
\end{align*}
\end{proof}

\section{Table of eigenvalues and coefficients in example \ref{Ex:KLExpansionSobolev}}

\begin{table}[H]
\caption{Numerically computed eigenvalues}
\centering
\begin{tabular}[t]{cccc}
	\toprule
	$j$	& $\lambda_{j,0}$ & $\lambda_{j,1}$  & $\lambda_{j,0}^{(1)}$  \\
	\midrule
	1  & 0.08089068167678326 & 0.48086018659210106 & 0.40528473456935109 \\
	2  & 0.00205965240041970 & 0.05188896418232924 & 0.04503163717437234 \\
	3  & 0.00026270533181986 & 0.01679276625548909 & 0.01621138938277404 \\
	4  & 0.00006841221198552 & 0.00846233354496206 & 0.00827111703202757 \\
	5  & 0.00002503515450258 & 0.00506233637514264 & 0.00500351524159693 \\
	6  & 0.00001121888658271 & 0.00337959494580253 & 0.00334946061627563 \\
	7  & 0.00000575104919253 & 0.00241194641180081 & 0.00239813452407900 \\
	8  & 0.00000324455735463 & 0.00180983003299179 & 0.00180126548697489 \\
	9  & 0.00000196663972025 & 0.00140714725673274 & 0.00140236932376938 \\
	10 & 0.00000126039330634 & 0.00112596695689149 & 0.00112267239492895 \\
	11 & 0.00000084458490071 & 0.00092107962719789 & 0.00091901300355862 \\
	12 & 0.00000058696086733 & 0.00076765835211689 & 0.00076613371374168 \\
	13 & 0.00000042049463315 & 0.00064948950852601 & 0.00064845557531096 \\
	14 & 0.00000030907610831 & 0.00055674545278783 & 0.00055594613795521 \\
	15 & 0.00000023223544260 & 0.00048248117921235 & 0.00048190812671742 \\
	16 & 0.00000017785812783 & 0.00042219077563854 & 0.00042173229403679 \\
	17 & 0.00000013850477062 & 0.00037250497547070 & 0.00037216229069729 \\
	18 & 0.00000010945820313 & 0.00033112614241088 & 0.00033084468128110 \\
	19 & 0.00000008764226557 & 0.00029626166312213 & 0.00029604436418506 \\
	20 & 0.00000007100060735 & 0.00026664160413764 & 0.00026645939156433 \\
	21 & 0.00000005812795777 & 0.00024124176933565 & 0.00024109740307516 \\
	22 & 0.00000004804483094 & 0.00020024023217768 & 0.00021919131128683 \\
	23 & 0.00000004005626363 & 0.00014433267360853 & 0.00020014060966388 \\
	24 & 0.00000003366115860 & 0.00006836802969012 & 0.00018346977572175 \\
	25 & 0.00000002849286837 & 0.00004894760841817 & 0.00016879830677607 \\
	26 & 0.00000002427950272 & 0.00001323423080357 & 0.00015581881375215 \\
	27 & 0.00000002081694590 & 0.00001003180514519 & 0.00014428078838354 \\
	\midrule
	$\sum_j$ & 0.08333315959712190 & 0.57899364818404330 & 0.49624779233385769\\
	\bottomrule
\end{tabular}
\label{Tab:Eigenvalues}
\end{table}

\begin{table}
\caption{Coefficients in the eigenfunctions \eqref{Eq:EigenfunctionIntBMinL2} and \eqref{Eq:EigenfunctionIntBMinH1}}
\centering
\begin{tabular}[t]{lcccc}
	\toprule
	\multicolumn{1}{c}{$j$} & \multicolumn{2}{c}{$e_{j,0}$} & \multicolumn{2}{c}{$e_{j,1}$} \\
	\cmidrule(l{3pt}r{3pt}){2-3} \cmidrule(l{3pt}r{3pt}){4-5} 
	& $c_{1,j}$ & $c_{2,j}$ & $c_{1,j}$ & $c_{2,j}$\\
	\midrule
	1  & -0.7340955138 & 1.0000000000 & -0.2081918396 & 0.6174903583\\
	2  & -1.0184673188 & 1.0000000000 & -0.2779430973 & 0.3097626792\\
	3  & -0.9992244965 & 1.0000000000 & -0.1215625025 & 0.1801681088\\
	4  & -1.0000335533 & 1.0000000000 & -0.1058332801 & 0.1291738338\\
	5  & -0.9999985501 & 1.0000000000 & -0.0715441918 & 0.1001298373\\
	6  & -1.0000000627 & 1.0000000000 & -0.0654586157 & 0.0819745505\\
	7  & -1.0000000000 & 1.0000000000 & -0.0505364120 & 0.0692954063\\
	8  & -1.0000000000 & 1.0000000000 & -0.0473800420 & 0.0600683171\\
	9  & -1.0000000000 & 1.0000000000 & -0.0390461987 & 0.0529798202\\
	10 & -1.0000000000 & 1.0000000000 & -0.0371245541 & 0.0474074011\\
	11 & -1.0000000000 & 1.0000000000 & -0.0318083595 & 0.0428834802\\
	12 & -1.0000000000 & 1.0000000000 & -0.0305179082 & 0.0391564374\\
	13 & -1.0000000000 & 1.0000000000 & -0.0268326445 & 0.0360195910\\
	14 & -1.0000000000 & 1.0000000000 & -0.0259070580 & 0.0333524714\\
	15 & -1.0000000000 & 1.0000000000 & -0.0232024012 & 0.0310499527\\
	16 & -1.0000000000 & 1.0000000000 & -0.0225064086 & 0.0290472920\\
	17 & -1.0000000000 & 1.0000000000 & -0.0204370850 & 0.0272854688\\
	18 & -1.0000000000 & 1.0000000000 & -0.0198948211 & 0.0257266396\\
	19 & -1.0000000000 & 1.0000000000 & -0.0182605729 & 0.0243351457\\
	20 & -1.0000000000 & 1.0000000000 & -0.0178262368 & 0.0230874077\\
	21 & -1.0000000000 & 1.0000000000 & -0.0165029430 & 0.0219606241\\
	22 & -1.0000000000 & 1.0000000000 & -0.0150539064 & 0.0200083191\\
	23 & -1.0000000000 & 1.0000000000 & -0.0128051164 & 0.0169879012\\
	24 & -1.0000000000 & 1.0000000000 & -0.0088422614 & 0.0116926873\\
	25 & -1.0000000000 & 1.0000000000 & -0.0075797880 & 0.0098937636\\
	26 & -1.0000000000 & 1.0000000000 & -0.0039302797 & 0.0051446892\\
	27 & -1.0000000000 & 1.0000000000 & -0.0034021301 & 0.0044792012\\
	\bottomrule
\end{tabular}
\label{Tab:Coefficients}
\end{table}

\end{document}